\documentclass[12pt]{amsart}

 \usepackage[dvips]{graphicx}
\usepackage{amsmath, amssymb}
\usepackage{pifont}

\numberwithin{equation}{section}
\newtheorem{theorem}{Theorem}[section]  
\newtheorem{theorem?}{``Theorem''}[section]  

\newtheorem{proposition}[theorem]{Proposition}
\newtheorem{lemma}[theorem]{Lemma}

\theoremstyle{definition}
\newtheorem{definition}[theorem]{Definition}

\theoremstyle{remark}
\newtheorem{remark}[theorem]{Remark}  
\newcommand{\R}{{\mathbb R}}
\newcommand{\C}{{\mathbb C}}
\newcommand{\Q}{{\mathbb Q}}
\newcommand{\N}{{\mathbb N}}
\newcommand{\Z}{{\mathbb Z}}
\newcommand{\1}{{\bf 1}}
\renewcommand{\a}{\alpha}
\renewcommand{\b}{\beta}
\renewcommand{\d}{\partial}
\renewcommand{\1}{{\bf 1}}
\newcommand{\E}{\hat{\mathcal{E}}}

\begin{document}
\title[Weighted oscillatory integrals]
{
Newton polyhedra and weighted oscillatory integrals
with smooth phases
} 
\author{Joe Kamimoto and Toshihiro Nose}
\address{Faculty of Mathematics, Kyushu University, 
Motooka 744, Nishi-ku, Fukuoka, 819-0395, Japan} 
\email{joe@math.kyushu-u.ac.jp}
\address{Faculty of Mathematics, Kyushu University, 
Motooka 744, Nishi-ku, Fukuoka, 819-0395, Japan} 
\address{
{\it Current address} Faculty of Engineering, 
Kyushu Sangyo University,
Matsukadai 2-3-1, 
Higashi-ku, Fukuoka, 813-8503, Japan} 
\email{t-nose@ip.kyusan-u.ac.jp}
%t-nose@math.kyushu-u.ac.jp}
%%%
\thanks{
The first author was supported by 
Grant-in-Aid for Scientific Research (C) (No. 22540199), 
Japan Society for the Promotion of Science. 
}
\maketitle
%\date{}

%%%%%%%%%%%%%%%%%%%%%%%%%%%%%%%%%%%%%%%%%%%%%

%\vspace{9 em}

\begin{abstract}
In his seminal paper, 
A. N. Varchenko precisely investigates
the leading term of 
the asymptotic expansion of an oscillatory integral
with real analytic phase. 
He expresses the order of this term by means of 
the geometry of the Newton polyhedron of the phase. 
The purpose of this paper is to generalize and 
improve his result. 
We are especially interested in the cases that
the phase is smooth and that the amplitude has a zero
at a critical point of the phase. 
In order to exactly treat the latter case, 
a weight function is introduced in the amplitude. 
Our results show that 
the optimal rates of decay for weighted oscillatory integrals, 
whose phases and weights are contained in a certain class 
of smooth functions including the real analytic class,
can be expressed by the Newton distance and multiplicity 
defined in terms of geometrical relationship of 
the Newton polyhedra of the phase and the weight. 
We also compute explicit formulae of the coefficient of
the leading term of the asymptotic expansion 
in the weighted case. 
Our method is based on the resolution of 
singularities constructed 
by using the theory of toric varieties, 
which naturally extends the resolution of Varchenko.  
The properties of poles of local zeta functions,
which are closely related to the behavior of
oscillatory integrals, are also 
studied under the associated situation.
The investigation of this paper 
improves on the earlier joint work with K. Cho. 
\end{abstract}

%\clearpage

\tableofcontents

%%%%%%%%%%%%%%%%%%%%%%%
%%%%%%%%%%%%%%%%%%%%%%%
%\setcounter{section}{-1}

%%%%%%%%%%%%%%%%%%%%%%%%%%%%%%%%%%%%%%%%%%%%%%%%%%%%%%%%%%%%%%%%%%%%%%%%%%%%%%%%%%%%%%%%%%%%%%%%%%%%%%%%%%%%%%%%%%%%%%%%%%%%%%%%%%%%%%%%%%%%%
\section{Introduction}
In this paper, we investigate the asymptotic behavior of 
scalar oscillatory integrals of the weighted form
%%%%%
\begin{equation}\label{eqn:1.1}
I(t;\varphi)=
\int_{\R^n}e^{it f(x)}g(x)\varphi(x)dx
\end{equation}
%%%%%
for large values of the real parameter $t$,
where $f,g,\varphi$ are real-valued smooth 
(infinitely differentiable) functions
defined on 
an open neighborhood $U$ of the origin in $\R^n$ 
and 
%%%
\begin{itemize} 
   \item
$f$ is called the {\it phase} and 
satisfies that $f(0)=|\nabla f(0)|=0$;
%%%
   \item
$g$ is called the {\it weight} 
($g\varphi$ is called the {\it amplitude});
%%%
   \item
   The support of $\varphi$ is contained in $U$.
   \end{itemize}
%%%%%%%%%%%%%

By using a famous Hironaka's resolution of singularities \cite{hir64},
it is known 
(see \cite{jea70}, \cite{mal74} and  Section~8.2 in this paper) 
that
if $f$ is real analytic
and the support of $\varphi$ is contained in a sufficiently small 
open neighborhood of the origin,
then the integral $I(t;\varphi)$ has an asymptotic expansion
of the form
%%%%%
   \begin{equation}\label{eqn:1.2}
   I(t;\varphi)\sim 
   \sum_{\a}\sum_{k=1}^n
   C_{\a k}(\varphi) t^{\a} (\log t)^{k-1} \quad 
   \mbox{as $t \to +\infty$},
   \end{equation}
%%%%%
where $\a$ runs through 
a finite number of arithmetic progressions,
not depending on the amplitude,
which consist of negative rational numbers.
In special cases of the smooth phase,
$I(t;\varphi)$ also admits 
an asymptotic expansion 
of the same form as in (\ref{eqn:1.2}) 
(see \cite{sch91}, \cite{kn13} and Theorem~3.7 in this paper).
In order to see the decay property of $I(t;\varphi)$, 
we are interested in the leading term of (\ref{eqn:1.2})
and define the following index.
%%%%%%%%%%%%%%
\begin{definition}
Let $f,g$ be smooth functions, for which the oscillatory integral
(\ref{eqn:1.1}) admits the asymptotic expansion 
of the form (\ref{eqn:1.2}).
The set $S(f,g)$ consists of pairs $(\alpha,k)$ such that 
for each neighborhood of the origin in $\R^n$, 
there exists a smooth function $\varphi$ 
with support contained in this neighborhood 
for which $C_{\alpha k}(\varphi)\neq 0$
in (\ref{eqn:1.2}). 
The maximum element of the set $S(f,g)$,
under the lexicographic ordering, 
is denoted by $(\b(f,g),\eta(f,g))$, i.e.,
$\b(f,g)$ is the maximum of values $\a$ for which
we can find $k$ so that $(\a,k)$ belongs to $S(f,g)$;
$\eta(f,g)$ is the maximum of integers $k$
satisfying that $(\b(f,g),k)$ belongs to $S(f,g)$.
We call $\b(f,g)$ the {\it oscillation index} of $(f,g)$
and $\eta(f,g)$ the {\it multiplicity} of its index.
(In the unweighted case, 
i.e. $g\equiv 1$, the multiplicity $\eta(f,1)$
is one less than the corresponding multiplicity in \cite{agv88},
p. 183.)
\end{definition}

The aim of this paper is to determine or precisely estimate
the oscillation index and its multiplicity
by means of appropriate information of 
the phase and the weight.
In the unweighted case, many strong results have 
been obtained.   
In particular, 
in a seminal work of Varchenko \cite{var76}
(see also \cite{agv88}),
the oscillation index and its multiplicity
are investigated in detail 
in the case when the phase is real analytic and
satisfies a certain nondegeneracy condition 
(see Theorem~\ref{thm:3.1} in Section~\ref{sec:3}). 
In his analysis,
the theory of toric varieties based on
the geometry of the Newton polyhedron of the phase
plays an important role.
Recently, it was shown in \cite{kn13} that 
the above result of Varchenko can be generalized 
to the case 
that the phase belongs to a wider class
of smooth functions,
denoted by $\hat{\mathcal{E}}(U)$, 
including the real analytic class
(see Theorem~\ref{thm:3.8} in Section~\ref{sec:3}).
On the other hand, 
another approach, 
which is inspired by the work of Phong and Stein 
on oscillatory integral operators
in the seminal paper \cite{ps97},
has been developed and succeeds to give many
strong results 
(\cite{gre09},\cite{gre10ma},\cite{gre10jam},
\cite{ikm10},\cite{im11jfaa},\cite{cgp13}, etc.).
In particular, the two-dimensional case has been 
deeply understood.  
In these papers, the importance of 
resolution of singularities constructed from 
the Newton polyhedron is also strongly recognized. 

Before explaining earlier works 
in the weighted case, 
we comment on significance for  
the investigation in this case. 
Since the weighted case may be considered 
as a special case of the unweighted case, 
unweighted results concerned with the upper bound 
estimates for oscillation index
are also available in the weighted case. 
But, these are ``uniformly'' satisfied 
with respect to the amplitude, so   
more precise results may be obtained
in the case of a specific amplitude. 
Therefore, for a given weight, 
one must try to get more accurate results 
by means of not only the information of the phase 
but also that of the weight. 
%% 
%On the other hand,

Until now, 
there are not so many studies about 
the weighted case, 
but some precise results have been obtained in   
\cite{vas79},\cite{agv88},\cite{py04},\cite{ckn13},\cite{ot13}. 
In these studies, 
the Newton polyhedra of both the phase 
and the weight play important roles. 
In \cite{vas79},\cite{agv88},\cite{ckn13},\cite{ot13}
it was made an attempt to generalize
the results of Varchenko in \cite{var76} 
as directly as possible in the weighted case 
under the nondegeneracy condition on the phase. 
Pramanik and Yang \cite{py04} consider 
the two-dimensional case with the weight of the form
$g(x)=|h(x)|^\epsilon$,
where $h$ is real analytic and $\epsilon$ is positive.
(This $g$ may not be smooth.)
Their approach is based on not only the method of 
Varchenko but also the above-mentioned work of 
Phong and Stein in \cite{ps97}.  
As a result, they succeed to remove the nondegeneracy 
hypothesis on the phase. 
%which is a necessary condition in the results of Varchenko. 
Note that every study, mentioned above, requires 
the assumption of the real analyticity of the phase. 

The purpose of this paper is to improve results in 
our previous paper \cite{ckn13}, 
which generalizes  
the above-mentioned results of Varchenko in \cite{var76}
to the weighted case.  
First, we develop the investigation in \cite{ckn13}
into the case of smooth phases. 
%%%
When the case of smooth functions is treated, 
it must be noticed that the geometrical information
of the Newton polyhedra does not always give 
sufficient analytical information.
To be more specific, 
though flat functions do not appear in 
the information of the Newton polyhedron, 
they may affect the behavior 
of the oscillatory integrals. 
In particular, it must be careful to deal with the case 
when the complement of their Newton polyhedra in $\R_+^n$
is noncompact.
Actually, we recall that 
since smooth weights were considered in \cite{ckn13}, 
the results needed many intricate hypothesis 
to avoid various troubles induced 
by noncompact faces of the Newton polyhedra
(see Theorems~3.2 and 3.3 in this paper). 
Additionally considering the case of smooth phases, 
one must overcome more complicated problems. 
%In order to give essential hypotheses, 
In this paper, 
we use the above-mentioned function class $\hat{\mathcal E}(U)$
in the investigation of the weighted case and 
succeed to generalize the above results 
in \cite{ckn13}.
Indeed, we show that
the real analyticity can be replaced by 
the $\hat{\mathcal E}$-property 
in the assumption on the phase;
many conditions can be written clearly
in terms of the $\hat{\mathcal E}$-property 
in the assumption on the weight
in the previous paper \cite{ckn13}. 
But, there is still room for improvement for
the refined condition of the weight and 
the subtlety of conditions of the weight 
will be explained  
in terms of some examples (see Section~\ref{sec:15}). 

Here, we briefly explain properties of 
the class $\hat{\mathcal E}(U)$, 
which were investigated in detail in \cite{kn13}.
In many earlier investigations 
of oscillatory integrals, 
the function $\gamma$-{\it part}, 
corresponding to each face $\gamma$ 
of the Newton polyhedron, 
plays an important role.
By means of summation,
the $\gamma$-part is simply defined as a function for 
every face $\gamma$ in the real analytic case.
From the viewpoint of this definition,    
the $\gamma$-part is considered as a formal power series 
when $\gamma$ is noncompact in the smooth case.   
This $\gamma$-part may not become a function, so
it is not useful for our analysis. 
From convex geometrical points of view 
(c.f. \cite{zie95}), we give another definition
of the $\gamma$-part, which always becomes 
a function defined near the origin
(see Section~\ref{subsec:2.4}).
This definition is a natural generalization
of that in the real analytic case.
We remark that not all smooth functions admit 
the $\gamma$-part for every face $\gamma$ of
their Newton polyhedra in our sense. 
The class $\hat{\mathcal E}(U)$ is defined to be the set of 
smooth functions admitting the $\gamma$-part
for every face $\gamma$ of its Newton polyhedron  
(see Section~\ref{subsec:2.5}).
Many kinds of $C^{\infty}$ functions are contained 
in this class. 
In particular, it is shown in \cite{kn13} that
the class $\hat{\mathcal E}(U)$ contains 
the Denjoy-Carleman quasianalytic classes,
which are
interesting classes of smooth functions and
have been studied 
from various points of view
(c.f. \cite{bm04},\cite{thi08}). 

Let us explain another
improvement on the preceded studies in \cite{ckn13}. 
As mentioned above, 
the importance of resolution of singularities
has been strongly recognized in
earlier successive investigations of the behavior 
of oscillatory integrals. 
Let us review our analysis from this point of view.  
The resolution in the work of Varchenko \cite{var76} 
is based on the theory of toric varieties. 
His method gives quantitative resolution 
by means of the geometry of the Newton polyhedron of the phase. 
In \cite{kn13}, we have directly generalized this resolution 
to the class $\hat{\mathcal E}(U)$ of smooth functions.   
Furthermore, in order to consider the weighted case, 
some kind of {\it simultaneous} resolution of singularities 
with respect to two functions, 
i.e. the phase and the weight,
must be constructed (see Section~\ref{subsec:8.2}). 
From the viewpoint of the theory of toric varieties, 
simultaneous resolution of singularities reflects
finer simplicial subdivision of a fan constructed 
from the Newton polyhedra of the above two functions.
Therefore, it is essentially important 
to investigate accurate relationship between 
cones of this subdivided fan and 
faces of the Newton polyhedra of the two functions.  
This situation has been investigated in \cite{ckn13}, 
but deeper understanding this relationship in this paper 
(see Section~11.2) gives
many stronger results about the behavior of oscillatory
integrals. 
In particular, we succeed to give 
explicit formulae of the coefficient of the leading term of 
the asymptotic expansion under some appropriate conditions, 
which reveals that the behavior of oscillatory integrals is
decided by some important faces,
which are called {\it principal faces}, of
the Newton polyhedra of the phase and the weight. 
%%%
By the way, 
Pramanik and Yang \cite{py04} also use 
simultaneous resolution of singularities in their analysis. 
Indeed, they extend the method 
of Phong and Stein \cite{ps97} 
by considering the expressions of 
both the phase and the weight 
in terms of the Puiseux serieses of their roots.

%%%%%%%%%%%%%%%%%%%%%%%%%%%%%%%%%%
It is known 
(see, for instance, 
\cite{igu78},\cite{agv88}, 
and Section~14.1 in this paper) 
that the asymptotic analysis of 
oscillatory integral (\ref{eqn:1.1}) 
can be reduced to an investigation of 
the poles of the functions 
$Z_{+}(s;\varphi)$ and $Z_{-}(s;\varphi)$ 
(see (\ref{eqn:8.1}) below), 
which are similar to the (weighted) 
{\it local zeta function} 
%%%%%
\begin{equation}\label{eqn:1.3}
Z(s;\varphi)=
\int_{\R^n} |f(x)|^s g(x)\varphi(x)dx,
\end{equation}
%%%%%
where $f$, $g$, $\varphi$ are the same 
as in (\ref{eqn:1.1}).
The substantial analysis in this paper is to  
investigate properties of poles of 
the local zeta function $Z(s;\varphi)$ and 
the functions $Z_{\pm}(s;\varphi)$ by means of
the Newton polyhedra of the functions 
$f$ and $g$.
%See Section~\ref{se:10} for more details. 
We will give analogous new results about 
properties of poles of these functions. 

%%%%
This paper is organized as follows. 
In Section~2, 
we explain many important words
and their elementary properties, 
which are often used in this paper.  
First, fundamental notion in convex geometry
are recalled.
Second, after recalling the concept of 
Newton polyhedra of smooth functions, 
we give definitions of key words:
Newton distance, Newton multiplicity and 
principal faces.
Third, we define the classes 
$\hat{\mathcal E}[P](U)$
and $\hat{\mathcal E}(U)$ 
of smooth functions by means of 
{\it modified} $\gamma$-parts and
summarize important properties
of these classes. 
In Section~3, 
in order to clarify original parts of 
this paper, 
we roughly explain earlier corresponding results 
in 
\cite{var76},\cite{vas79},\cite{agv88},\cite{py04},\cite{ckn13},\cite{kn13}. 
In Section~4,
our main results relating to the behavior 
of oscillatory integrals are stated.
Note that the result concerning explicit
formulae of the coefficient of the leading term
of the asymptotic expansion (\ref{eqn:1.2}) 
will be given in Section~14. 
The goal of the next three sections is to 
construct resolution of singularities with respect to
several functions 
by using the theory of toric varieties.  
First, we briefly recall how to construct
a toric variety from a given fan
in Section~5.
In Section~6, 
the construction of fans from 
a given polyhedron is explained. 
Here, we give a few lemmas, 
which reveal important relationship of faces of polyhedron
and cones of the fan. 
Finally, 
simultaneous resolution of singularities 
with respect to several functions is constructed 
in Section~7.
In the next six sections, 
the properties of poles of 
local zeta function $Z(s;\varphi)$ and 
similar functions 
$Z_{+}(s;\varphi)$ and $Z_{-}(s;\varphi)$ are investigated 
under the assumption that 
$f$ belongs to the class $\hat{\mathcal E}(U)$
and satisfies some nondegeneracy
condition. 
In Section~8,
we roughly explain how to understand required properties
of poles by means of resolution of singularities. 
In Section~9, 
the model case: 
$f$ and $g$ in (\ref{eqn:1.3}) are monomials,
is exactly investigated. 
In this case, it is easy to compute the coefficient
of the Laurent expansion of the leading pole. 
In Section~10,
we treat the weight $g$ of a special form:
a monomial multiplied by a smooth function. 
By using results in this case,
we can generalize the results about unweighted local zeta
type functions due to Varchenko \cite{var76} and
Kamimoto and Nose \cite{kn13} in the case when
$f$ admits an asymptotic expansion at the origin
of the form of fractional power series.
In Section~11,
more general case when $g$ belongs to the class $\hat{\mathcal E}(U)$ 
is investigated and 
the strongest result concerning the poles of 
local zeta type functions is given. 
By the way, 
the case when $f$ is convenient (see Section~2.2) is easy to be treated. 
This situation is explained in Section~12.
In Section~13, we consider 
the question: what will happen in the positions of the leading poles
of $Z(s;\varphi)$, 
if $f$ and $g$ are exchanged? 
In Section~14, 
after an exact relationship between
oscillatory integrals and local zeta type functions
in terms of the Mellin transform is recalled, 
the results in the above six sections are 
translated into those 
in the case of oscillatory integrals. 
As a result, 
proofs of the theorems in Section~4 are given. 
In the last section, 
we give some examples, which shows the subtlety 
of the condition of the weight
in the assumption of the main theorems.

%%%%%%%%%%%%%%%%%%%%%%%%%%%%%%%%%%%%%%%%%%%%%%%%%%%%%%%%%%%%%%%%%%%%%%%%%%%%%%%%
Some of the results in this paper have been announced 
in \cite{knhir13}.

%%%%%%%%%%%%%%%%%%%%%%%%%%%%%%%%%%%%%%%%%%%%%%%%%%%%%%%%%%%%%%

{\it Notation and symbols.}\quad
%%%%%%%%%%%%%%%%%%
\begin{itemize}%{enumerate}
\item 
We denote by $\Z_+, \Q_+, \R_+$ the subsets consisting of 
all nonnegative numbers in $\Z,\Q,\R$, respectively.
We write $\R_{>0}:=\{x\in\R:x>0\}$.
For $s\in\C$, ${\rm Re}(s)$ expresses the real part of $s$.
\item
We use the multi-index as follows.
For $x=(x_1,\ldots,x_n), y=(y_1,\ldots,y_n) \in\R^n$, 
$\a=(\a_1,\ldots,\a_n)\in\Z_+^n$, 
$p=(p_1,\ldots,p_n)\in\N^n$,
define
%%%%%%
\begin{eqnarray*}
&& 
|x|=\sqrt{|x_1|^{2}+\cdots +|x_n|^{2}}, \quad 
\langle x,y\rangle=x_1 y_1+\dots+x_n y_n, 
\\
&& 
x^{\a}=x_1^{\a_1}\cdots x_n^{\a_n}, \quad
\partial^{\a}=
\left(\frac{\partial}{\partial x_1}\right)^{\a_1}\cdots
\left(\frac{\partial}{\partial x_n}\right)^{\a_n},  
\\
&&
\langle \a\rangle=\a_1+\cdots+\a_n, \quad
\a!=\a_1 !\cdots \a_n!, \quad
0 !=1,
\\
&&
\alpha/p=(\alpha_1/p_1,\ldots,\alpha_n/p_n),
\quad 
x^{\alpha/p}
=x_1^{\alpha_1/p_1}\cdots x_n^{\alpha_n/p_n}.
\end{eqnarray*}
%%%%%%
%%
\item
For $A,B\subset \R^n$ and $c\in\R$, 
we set 
%%%%%
$$
A+B=\{a+b\in\R^n: a\in A \mbox{ and } b\in B\},\quad
c\cdot A=\{ca\in\R^n: a\in A\}.
$$
Moreover, ${\rm Int}(A)$ expresses the interior of the set $A$.
%%%%%
%%
\item
We express by $\1$ the vector $(1,\ldots,1)$ or the set 
$\{(1,\ldots,1)\}$. 
For $x=(x_1,\ldots,x_n)\in\R^n$, 
$x^{\1}$ means $x_1\cdots x_n$.
\item
For a finite set $A$, 
$\# A$ means the cardinality of $A$. 
\item
For a nonnegative real number $r$ and 
a subset $I$ in $\{1,\ldots,n\}$, the map 
$T_I^{r}:\R^n\to\R^n$ is defined by 
%$y_j=0$ if $j\in I$ and $y_j=x_j$ otherwise. 
\begin{equation}\label{eqn:1.4}
(z_1,\ldots,z_n)=T_I^{r}(x_1,\ldots,x_n)\,\, 
\mbox{ with }\,\, z_j:=\begin{cases}
r& 
\quad \mbox{for $j\in I$}, \\
x_j&
\quad \mbox{otherwise}.
\end{cases}
\end{equation}
We define $T_I:=T_I^0$.
For a set $A$ in $\R^n$, 
the image of $A$ by $T_I$ is denoted by $T_I(A)$.  
When $A=\R^n$ or $\Z_+^n$, 
its image is expressed as 
\begin{equation}\label{eqn:1.5}
T_I(A)=\{x\in A: x_j=0 \mbox{ for $j\in I$}\}.
\end{equation}
%%%%%
%%
\item
For a smooth function $f$, 
we denote by Supp($f$) the support of $f$, i.e.,
Supp($f$) is the closure of the set $\{x\in \R^n: f(x)\neq 0\}$. 
\end{itemize}%{enumerate}
%%%%%%%%%%%%%%%%%%%%

%%%%%%%%%%%%%%%%%%%%%%%%%%%%%%%%%%%%%%%%%%%%%%%%%%%%%%%%%%%%%%%%%%%%%%%%%%%%%%%%%%%%%%%%%%%%%%%%%%%%%%%%%%%%%%%%%%%%%%%%%%%%%%%%%%%%%%%%%%%%%%

\section{Newton polyhedra and the class 
$\hat{\mathcal{E}}(U)$}\label{sec:2}

%%%%%%%%%%%%%%%%%%%%%%%%%%%%%%%%%%%%%%%%%%%%%%%%%%%%%%%%%%%%%%%%%%%%%%%%%%%%%%%%%%%%%%%%%%%%%%%%%%%%%%%%%%%%%%%%%%%%%%%%%%%%%%%%%%%%%%%%%%%%%%
\subsection{Polyhedra}\label{subsec:2.1}

Let us explain fundamental notions 
in the theory of convex polyhedra, 
which are necessary for our investigation.
Refer to \cite{zie95}   
for general theory of convex polyhedra.  

For $(a,l)\in \R^n\times\R$, 
let $H(a,l)$ and $H^+(a,l)$ be  
a hyperplane and 
a closed halfspace in $\R^n$ 
defined by
\begin{equation}\label{eqn:2.1}
\begin{split}
&H(a,l):=\{x\in\R^n:\langle a,x\rangle =l\},\\
&H^+(a,l):=\{x\in\R^n:\langle a,x\rangle \geq l\},
\end{split}
\end{equation} 
respectively. 
%%%
A ({\it convex rational}) {\it polyhedron} is  
an intersection of closed halfspaces:
a set $P\subset\R^n$ presented in the form
$
P=\bigcap_{j=1}^N H^+(a^j,l_j)
$
for some $a^1,\ldots,a^N\in\Z^n$ and 
$l_1,\ldots,l_N \in\Z$.

Let $P$ be a polyhedron in $\R^n$. 
A pair $(a,l)\in \Z^n\times\Z$ is said to be 
{\it valid} for $P$ 
if $P$ is contained in $H^+(a,l)$.
A {\it face} of $P$ is any set of the form 
$
F=P\cap H(a,l),
$
where $(a,l)$ is valid for $P$. 
Since $(0,0)$ is always valid, 
we consider $P$ itself as a trivial face of $P$;
the other faces are called {\it proper faces}.  
Conversely, 
it is easy to see that any face is a polyhedron. 
Considering the valid pair $(0,-1)$, 
we see that the empty set is always a face of $P$. 
Indeed, $H^+(0,-1)=\R^n$, but $H(0,-1)=\emptyset$.
We write
%%%%
\begin{equation}\label{eqn:2.2}
{\mathcal F}[P]
=\mbox{the set of all nonempty faces of $P$.}
\end{equation}
%%%%
The {\it dimension} of a face $F$ is the dimension of 
its affine hull
(i.e., the intersection of all affine flats that 
contain $F$), which is denoted by $\dim(F)$. 
The faces of dimensions $0,1$ and $\dim(P)-1$
are called {\it vertices}, {\it edges} and 
{\it facets}, respectively. 
The {\it boundary} of a polyhedron $P$, denoted by 
$\d P$,  
is the union of all proper faces of $P$.  
For a face $F$, $\d F$ is similarly defined. 

\begin{remark}\label{rem:2.1}
It follows from the definition of $H^+(\cdot,\cdot)$ that
for $a\in\Z^n$, $l\in\R$, 
\begin{equation*}
\begin{split}
&H^+(a,dl)=d\cdot H^+(a,l) \,\,
\mbox{ for $d> 0$,} \\
&H^+(a,l+\langle a,b\rangle)=H^+(a,l)+b \,\,
\mbox{ for $b\in\Z^n$.} 
\end{split}
\end{equation*}
In the case of hyperplanes $H(a,l)$,
analogous equations can be obtained.
\end{remark}

Every polyhedron treated in this paper 
satisfies a condition in the following lemma. 
%%%%%%%%%%%%%%%%%%%%%%
\begin{lemma}\label{lem:2.2}
Let $P\subset\R_+^n$ be a polyhedron. 
Then the following conditions are equivalent.
\begin{enumerate}
\item $P+\R_+^n\subset P;$
\item There exists a finite set of pairs 
$\{(a^j,l_j)\}_{j=1}^N\subset \Z_+^n\times \Z_+$ such that
$P=\bigcap_{j=1}^N H^+(a^j,l_j).$
\end{enumerate} 
\end{lemma}
%%%%%%%%%%%%%%%%%%%%%%
\begin{proof}
(i) $\Longrightarrow$ (ii). \quad
Suppose that (ii) does not hold. 
From the definition of the polyhedron, 
$P$ is expressed as 
$P=\bigcap_{j=1}^N H^+(a^j,l_j)$ with 
$(a^j,l_j)\in\Z^n\times \Z$. 
Here, it may be assumed that
the set $A:=\{(a^j,l_j)\}_{j=1}^N$ satisfies 
that $P\cap H(a^j,l_j)\neq\emptyset$ for all $j$.
If $(a,l)\in A$ belongs to $\Z_+^n\times(-\N)$, 
then $P\cap H(a,l)=\emptyset$. 
On the other hand, if there exists $(a,l)\in A$
with $a\in \Z^n\setminus\Z_+^n$, then
the nonempty face $\gamma:=P\cap H(a,l)$ satisfies
$\gamma+\R_+^n\not\subset H^+(a,l)$, 
which implies $P+\R_+^n\not\subset P$.

(ii) $\Longrightarrow$ (i). \quad 
This implication easily follows from the following:
For any $(a,l)\in \Z_+^n\times\Z_+$, 
if $\alpha\in H^+(a,l)$, then
$\alpha+\R_+^n\subset H^+(a,l)$.
\end{proof}

%%%%%%%%%%%%%%%%%%%%%%%%%%%%%%%%%%%%%%%%%%%%%%%%%%%%%%%%%%%%%%%%%%%%%%%%%%%%%%%%%%%%%%%%%%%%%%%%%%%%%%%%%%%%%%%%%%%%%%%%%%%%%%%%%%%%%%%%%%%%%%%
\subsection{Newton polyhedra}\label{subsec:2.2}

Let us define the Newton polyhedron 
and some important classes of smooth functions 
characterized in terms of the Newton polyhedron.

Let $f$ 
be a smooth function defined 
on a neighborhood of the origin in $\R^n$, 
which has the Taylor series at the origin:  
\begin{equation}\label{eqn:2.3}
f(x)\sim
\sum_{\alpha\in{\Z}_+^n} c_{\alpha}x^{\alpha} 
\quad\quad \mbox{ with $c_{\a}=\dfrac{\partial^\a f(0)}{\alpha!}$.}
\end{equation}
%%%
\begin{definition}\label{def:2.3}
The {\it Newton polyhedron} $\Gamma_+(f)$ of $f$
is defined to be the convex hull of the set 
$\bigcup\{\a+\R_+^n:c_\a \neq 0\}$.
%%%%%
\end{definition}
%%%
It is known that the Newton polyhedron 
is a polyhedron (see \cite{zie95}). 
The union of the compact faces of 
the Newton polyhedron $\Gamma_+(f)$ is called 
the {\it Newton diagram} $\Gamma(f)$ of $f$, 
while the topological boundary of $\Gamma_+(f)$ 
is denoted by 
$\d\Gamma_+(f)$. 
The polynomial 
\begin{equation}\label{eqn:2.4}
f_{\Gamma(f)}(x):=\sum_{\alpha\in\Gamma(f)\cap\Z_+^n}
c_{\alpha}x^{\alpha}
\end{equation}
is an important part of the series (\ref{eqn:2.3}). 
The following classes of smooth functions
often appear in this paper.
%\begin{definition}

\begin{itemize}
\item
$f$ 
is said to be {\it flat} if 
$\Gamma_+(f)=\emptyset$  
(i.e. all derivatives of $f$ vanish at the origin).
%(or {\it nonflat} if $\Gamma_+(f)\neq\emptyset$);
%%
\item
$f$ is said to be {\it convenient} if 
the Newton polyhedron $\Gamma_+(f)$
intersects all the coordinate axes.
\end{itemize}
%\end{definition}

%%%%%%%%%%%%%%%%%%%%%%%%%%%%%%%%%%%%%%%%%%%%%%%%%%%
%%%%%%%%%%%%%%%%%%%%%%%%%%%%%%%%%%%%%%%%%%%%%%%%%%%
\subsection{Newton distance and multiplicity}\label{subsec:2.3}
%%%%%%%%%%%%%%%%%%%%%%%%%%%%%%%%%%%%%%%%%%%%%%%%%%%
%%%%%%%%%%%%%%%%%%%%%%%%%%%%%%%%%%%%%%%%%%%%%%%%%%%

Let $f,g$ be nonflat smooth functions 
defined on a neighborhood of the origin in $\R^n$.
We define 
the Newton distance and 
the Newton multiplicity with respect to the pair $(f,g)$. 
At the same time, consider 
important faces of $\Gamma_+(f)$ and  
$\Gamma_+(g)$, which will initially affect  
the analysis of oscillatory integrals. 
Hereafter, we assume that $f(0)=0$.
%and assume that $f,g$ are {\it nonflat}, i.e.,
%$\Gamma(f)$ and $\Gamma(g)$ are nonempty. 
%%%%%%%%%%%%%%%%%%%%%%
\begin{definition}\label{def:2.4}
The {\it Newton distance} of the pair $(f,g)$ is defined by 
%%%%%
\begin{equation}\label{eqn:2.5}
d(f,g):=\max\{d>0:
\d\Gamma_+(f)\cap d\cdot(\Gamma_+(g)+\1)\neq\emptyset\}.
\end{equation}
%%%%%
\end{definition}
%%%%%%%%%%%%%%%%%%%%%%
This distance will be crucial to determine or estimate the 
oscillation index. 
%%%%%%%%%%%%%%%%%%%%%%
%%%
\begin{remark}\label{rem:2.5}
It is easy to see the following. 
\begin{itemize}
\item 
$d(f,g)=\min\{d>0: d\cdot(\Gamma_+(g)+\1)\subset\Gamma_+(f)\}$;
\item
$d(f,g)=
\max\{d(f,x^{\beta}):\beta\in\Gamma_+(g)\}$;
\item
$d(f,g)\leq 
\min\{d(x^{\alpha},g):\alpha\in\Gamma_+(f)\}$
(notice that ``$<$'' is possible).
\end{itemize}
\end{remark}
%%%%%%%%%%%%%%%%%%%%%%%
In \cite{agv88}, p.254, 
the number $d(f,g)$ is called 
the {\it coefficient of inscription} of $\Gamma_+(g)$ in $\Gamma_+(f)$.  
(In \cite{agv88}, this number is defined by 
$\min\{d>0: d\cdot \Gamma_+(g)\subset\Gamma_+(f)\}$,
which must be corrected as in (\ref{eqn:2.5}).)

%%%%%%%%%%%%%%%%%%%%%

%
We define the map $\Phi:\R^n\to\R^n$ as
\begin{equation}\label{eqn:2.6}
\Phi(\beta):=d(f,g)(\beta+\1).
\end{equation}
%%%%
The image of $\Gamma_+(g)$ by the map $\Phi$ comes 
in contact with
the boundary of $\Gamma_+(f)$. 
We denote by $\Gamma_0(f)$ this contacting set on 
$\partial\Gamma_+(f)$ and 
by $\Gamma_0(g)$ the image of $\Gamma_0(f)$ 
by the inverse map of $\Phi$, i.e.,
%%%%%
\begin{eqnarray*}
&&
\Gamma_0(f):=\partial\Gamma_+(f)\cap\Phi(\Gamma_+(g))
\,\,(=\partial\Gamma_+(f)\cap d(f,g)\cdot(\Gamma_+(g)+\1)); \\
&&
\Gamma_0(g):=\Phi^{-1}(\Gamma_0(f))
\,\,\left(
=
\left(
\frac{1}{d(f,g)}\cdot\partial\Gamma_+(f)-\1
\right)\cap\Gamma_+(g)
\right).
\end{eqnarray*}
%%%%%
%%%%%
Note that 
$\Gamma_0(g)$ is a certain union of faces of $\Gamma_+(g)$.
%%%
It is easy to see the following. 
\begin{itemize}
\item 
The map $\Phi:\R^n\to\R^n$ 
is a bijection and so is its restricted map 
$\Phi |_{\Gamma_0(g)}:\Gamma_0(g)\to\Gamma_0(f)$;
\item
If $\beta\in\Gamma_0(g)$, 
then $d(f,g)=d(f,x^{\beta})$;
\item
If $\alpha\in\Gamma_0(f)$, 
then
$d(f,g)\leq d(x^{\alpha},g)$ (notice that ``$<$'' is possible).
\end{itemize}
%%%%%%%%%%%%%%%%%%%%%
Let us define the Newton multiplicity and important faces of 
$\Gamma_+(f)$ and $\Gamma_+(g)$,  
which will play important roles in the investigation of 
multiplicity of the oscillation index, 
by using the map:
\begin{equation}\label{eqn:2.7}
\tau_f:\partial\Gamma_+(f)\to{\mathcal F}[\Gamma_+(f)]
%\quad
%\rho_f:\partial\Gamma_+(f)\to\{1,\ldots,n\}
\end{equation}
defined as follows (see the definition (\ref{eqn:2.2}) of 
${\mathcal F}[\cdot]$).
%%%
%%%
For $\alpha\in\partial\Gamma_+(f)$, let
$\tau_f(\alpha)$ be the smallest face of 
$\Gamma_+(f)$ containing $\alpha$. 
In other words,
$\tau_f(\alpha)$ is the face whose relative interior
contains the point $\alpha\in\partial\Gamma_+(f)$.
%More generally, for a subset $A$ of $\partial\Gamma_+(f)$,
%$\tau_f(A)$ is defined to be the union of $\tau_f(\alpha)$
%for $\alpha\in A$.
%For $\alpha\in\partial\Gamma_+(f)$, 
%define $\rho_f(\alpha):=n-\dim \tau_f(\alpha)$.
%%%
Define
\begin{equation}\label{eqn:2.8}
{\mathcal F}_0[\Gamma_+(f)]:=
\{\tau_f(\alpha)\in{\mathcal F}[\Gamma_+(f)]:
\alpha\in\Gamma_0(f)\}.
\end{equation}
%
%%%
\begin{definition}\label{def:2.6}
The {\it Newton multiplicity} of the pair $(f,g)$ is defined by  
%%%%%
\begin{equation}\label{eqn:2.9}
m(f,g):=\max\{
n-\dim(\tau): \tau\in{\mathcal F}_0[\Gamma_+(f)]
\}.
\end{equation}
%%%%%
\end{definition}
%%%%%%%%%%
\begin{remark}\label{rem:2.7}
It is easy to see the following.
\begin{itemize}
\item $m(f,g)=\max\{m(f,x^{\beta}):\beta\in\Gamma_0(g)\};$
\item
If $\Gamma_0(f)$ contains a vertex
of $\Gamma_+(f)$, then
$m(f,g)=n$. 
\end{itemize}
\end{remark}
%%
%%%%%%%%%%%%%%%%%%%%%%%%%%%%%
%
\begin{definition}\label{def:2.8}
Define
\begin{equation}\label{eqn:2.10}
{\mathcal F}_*[\Gamma_+(f)]:=
\{\tau\in{\mathcal F}_0[\Gamma_+(f)]:
n-\dim(\tau)=m(f,g)\}.
\end{equation}
%%%
The elements of the above set 
are called the {\it principal faces of} 
$\Gamma_+(f)$.  
Define
\begin{equation}\label{eqn:2.11}
{\mathcal F}_*[\Gamma_+(g)]
:=\{
\Phi^{-1}(\tau_*)\cap\Gamma_+(g):\tau_*\in
{\mathcal F}_*[\Gamma_+(f)]
\}.
\end{equation}
%%%%%%%%%%%%%
It is easy to see that
every element of the above set is a face of $\Gamma_+(g)$,
which is called a {\it principal face of} $\Gamma_+(g)$.
The map
$\Psi_*:
{\mathcal F}_*[\Gamma_+(f)]\to
{\mathcal F}_*[\Gamma_+(g)]$
is defined as 
$\Psi_*(\tau_*):=\Phi^{-1}(\tau_*)\cap\Gamma_+(g)$.
It is easy to see that this map is bijective. 
We say that $\tau_*\in{\mathcal F}_*[\Gamma_+(f)]$ 
(resp. $\gamma_*\in {\mathcal F}_*[\Gamma_+(g)]$) 
{\it is associated to} 
$\gamma_*\in {\mathcal F}_*[\Gamma_+(g)]$ 
(resp. $\tau_*\in {\mathcal F}_*[\Gamma_+(f)]$), 
if 
$\gamma_*=\Psi_*(\tau_*)$.
\end{definition}
%%%%%%%%%%%%%%%%%%%%%%%%%%%%%
%
%It is shown in \cite{ckn13} that
%there exist the faces 
%$\gamma_*^{(1)},\ldots,\gamma_*^{(l)}$ 
%of $\Gamma_+(g)$ such that
%$$\{\b\in\Gamma_0(g):
%(\rho_f\circ\Phi)(\b)=m(f,g)\}
%=\biguplus_{j=1}^l \gamma_*^{(j)}\quad\quad 
%(\mbox{disjoint union}).$$ 
\begin{remark}\label{rem:2.9}
In \cite{ckn13}, the union of the faces belonging to 
${\mathcal F}_*[\Gamma_+(g)]$ 
was called the {\it essential set} on $\Gamma_0(g)$. 
It is shown in \cite{ckn13} that
every two faces belonging to ${\mathcal F}_*[\Gamma_+(g)]$ 
are disjoint. 
\end{remark}

\begin{remark}\label{rem:2.10}
It is easy to see the following. 
\begin{itemize}
\item
If $\tau_*\in{\mathcal F}_*[\Gamma_+(f)]$, 
then
$\Psi_*(\tau_*)=
\{\beta\in\Gamma_0(g):
(\tau_f\circ\Phi)(\beta)=\tau_*
\};$
%%%%%%%%%%%%%
\item If $\gamma_*=\Psi_*(\tau_*)
\in{\mathcal F}_*[\Gamma_+(g)]$, then 
$\dim(\gamma_*)\leq\dim(\tau_*)=n-m(f,g)$;
%%%%%%%%%%%%%%
\item If $\gamma_*=\Psi_*(\tau_*)
\in{\mathcal F}_*[\Gamma_+(g)]$, then 
the compactness of $\gamma_*$ is equivalent to that of $\tau_*$; 
%%%%%%%%%%%%%%
\item If $f$ is convenient, then
every principal face of $\Gamma_+(g)$ is compact.
\end{itemize}
\end{remark}

% $j\in\{1,\ldots,l\}$.

%Let $\Gamma^{(k)}$ be the union of $k$-dimensional faces of 
%$\d\Gamma_+(f)$. 
%%%%%%%%%%%%%%%%%%%%%%%%%%%
%\begin{remark}
%When $g$ is flat, we set $d(f,g)=0$ and $m(f,g)=1$. 
%In this paper, we do not treat with the case 
%when $f$ is flat.
%\end{remark}
%%%%%%%%%%%%%%%%%%%%%%%%%%%

%%%%%%%%%%%%%%%%%%%%%%%%%%%
\begin{remark}\label{rem:2.11}
Let us consider the case 
$g(0)\neq 0$. 
Then $\Gamma_+(g)=\R_+^n$. 
In this case, since $d(f,g)$ and $m(f,g)$ are independent
of $g$, we simply denote them by $d(f)$ and $m(f)$, respectively. 
It is easy to see the following:
\begin{itemize}
\item $d(f,g)\leq d(f)$ for general $g$;
\item 
The Newton distance $d(f)$ is determined by 
the point $q_*=(d(f),\ldots,d(f))$, which is the 
intersection of the line $\alpha_1=\cdots=\alpha_n$ 
with $\partial\Gamma_+(f)$;
\item
$\Gamma_0(g)$ and the principal face of $\Gamma_+(g)$ 
contain the point $\{0\}$. 
(Note that they may not equal $\{0\}$.);
\item  
The principal face of $\Gamma_+(f)$ 
is the smallest face $\tau_*$ of $\Gamma_+(f)$ 
containing the point $q_*$;
\item
$m(f)=n-\dim(\tau_*)$.
%\item
%$P(f,g)=Q(f,g)=\R_+^n$.
\end{itemize}
More generally, in the case when 
$\Gamma_+(g)=\{p\}+\R_+^n$ with $p\in\Z_+^n$, 
the geometrical meanings of 
the quantities $d(f,g)$ and $m(f,g)$ 
will be considered 
in Lemmas~\ref{lem:10.4} and \ref{lem:10.5}, below. 
\end{remark}
%%%%%%%%%%%%%%%%%%%%%%%%%%%%

%%%%%%%%%%%%%%%%%%%%%%%%%%%%%%%%%%%%%%%%%%%%%%%%%%%%%%%%%%%%%%%%%%%%%%%%%%%%%%%%%%%%%%%%%%%%%%%%%%%%%%%%%%%%%%%%%%%%%%%%%%%%%%%%%%%%%%%%%%%%%

\subsection{The $\gamma$-part}\label{subsec:2.4}

Let $f$ be a ($C^{\infty}$) smooth function
defined on a neighborhood $V$ of the origin
whose Taylor series at the origin is as in (\ref{eqn:2.3}),
$P\subset\R_+^n$ a nonempty polyhedron in $\R_+^n$
containing $\Gamma_+(f)$ and
$\gamma$ a face of $P$.
Note that $P$ satisfies the condition:
$P+\R_+^n\subset P$ (see Lemma~\ref{lem:2.2}).
%%%%%%%%%
\begin{definition}\label{def:2.12}
We say that $f$ 
{\it admits the $\gamma$-part} 
on an open neighborhood $U\subset V$ of the origin
if for any $x$ in $U$ the limit:
%%%%%
\begin{equation}\label{eqn:2.12}
\lim_{t\to 0}\frac{f(t^{a_1}x_1,\ldots,t^{a_n}x_n)}{t^l}
\end{equation}
%%%%%
exists for {\it all} valid pairs 
$(a,l)=((a_1,\ldots,a_n),l)\in \Z_+^n\times\Z_+$
defining $\gamma$. 
When $f$ admits the $\gamma$-part,
it is known in \cite{kn13}, Proposition~5.2 (iii), that 
the above limits take the same value for any $(a,l)$,
which is denoted by $f_{\gamma}(x)$.
We consider 
$f_{\gamma}$ as a function on $U$,
which is called
the $\gamma$-part of $f$ on $U$.
\end{definition}

\begin{remark}\label{rem:2.13}
From the condition $P+\R_+^n\subset P$,
Lemma~\ref{lem:2.2} implies that
it is sufficient to consider valid pairs 
$(a,l)$ belonging to $``\Z_+^n\times\Z_+$''.
\end{remark}

%%%%%%%%%%%%%%%%%%%%%%%%%%%%%%%%%%%%%%%%%%
\begin{remark}\label{rem:2.14}
We summarize important properties of 
the $\gamma$-part.
See \cite{kn13} for the details.  
%%%%%%%%%%%%%%%%%
\begin{enumerate}
\item 
The $\gamma$-part $f_{\gamma}$  
is a $C^{\infty}$ smooth function defined on $U$. 
%(See Proposition~\ref{pro:7.7}.)  
%%%
\item 
If $f$ admits the $\gamma$-part $f_{\gamma}$
on $U$, then
$f_{\gamma}$ has the quasihomogeneous property:
%%%%%
\begin{equation}\label{eqn:2.13}
f_{\gamma}(t^{a_1} x_1,\ldots, t^{a_n} x_n)=t^l f_{\gamma}(x) 
\mbox{\,\, for \,\, 
$0<t<1$ and $x\in U$},
\end{equation}
%%%%%
where $(a,l)\in\Z_+^n\times\Z_+$ is a valid pair
defining $\gamma$.  
%(See Lemma~\ref{le:6.7} and \ref{le:7.8}.)
%%
\item
For a compact face $\gamma$ of $\Gamma_+(f)$, 
$f$ always admits the $\gamma$-part near the origin
and $f_{\gamma}(x)$ equals the polynomial 
$\sum_{\a\in\gamma\cap\Z_+^n}
c_{\alpha}x^{\alpha}$, which is the same as 
the well-known $\gamma$-part of $f$ in 
\cite{var76},\cite{agv88}. 
Note that $\gamma$ is a compact face if and only if
every valid pair $(a,l)=(a_1,\ldots,a_n)$ defining $\gamma$
satisfies $a_j>0$ for any $j$. 
%(See Lemma~\ref{le:6.6}.)
%%
\item 
Let $f$ be a smooth function and $\gamma$ 
a noncompact face of 
$\Gamma_+(f)$. 
Then, $f$ does not admit the $\gamma$-part
in general
(see Section~\ref{subsec:2.6}). 
If $f$ admits the $\gamma$-part, then
the Taylor series of 
$f_{\gamma}(x)$ at the origin is 
$\sum_{\a\in\gamma\cap\Z_+^n}
c_{\alpha}x^{\alpha}$, 
where the Taylor series of $f$ is as in (\ref{eqn:2.3}).
%That is, the asymptotic expansion of $f_{\gamma}$ at 
%the origin is $\sum_{\a\in\gamma\cap\Z_+^n}
%c_{\alpha}x^{\alpha}$. 
%(See Lemma~\ref{le:6.6}.) 
%%
\item
Let $f$ be a smooth function and 
$\gamma$ a face defined by 
the intersection of $\Gamma_+(f)$ and 
some coordinate hyperplane. 
More exactly, 
there exists a nonempty subset 
$I\subset\{1,\ldots,n\}$ such that
$\gamma=\Gamma_+(f)\cap T_I(\R_+^n)$
(see (\ref{eqn:1.5})).
Altough $\gamma$ is a noncompact face
if $\gamma\neq\emptyset$,
$f$ always admits the $\gamma$-part.
Indeed, $T_I(\R_+^n)$ can be regarded 
as a face of $\R_+^n$, 
which is expressed by $\R_+^n\cap H(a_I,0)$,
where $a_I=(a_1,\ldots,a_n)\in\R_+^n$ 
satisfies that $a_j>0$ iff $j\in I$.
Thus every valid pair defining $\gamma$ takes
the form $(a,l)=(a_I,0)$, 
so $l=0$ implies the existence of the limit
(\ref{eqn:2.12}) (see Section~2.6.2, below).
%Moreover, if
%$\gamma$ defined by a valid pair
%$(T_I(a),0)$ for $a\in\N^n$ 
%and $I\subset\{1,\ldots,n\}$,
% and
%$f_{\gamma}(x)=f(T_{\{1,\ldots,n\}\setminus I}(x))$.
%%
\item 
If $f$ is real analytic and $\gamma$ is a face of 
$\Gamma_+(f)$, then
$f$ admits the $\gamma$-part.
Moreover, $f_{\gamma}(x)$ is real analytic and 
is equal to a convergent power series
$\sum_{\a\in\gamma\cap\Z_+^n}
c_{\alpha}x^{\alpha}$ on some neighborhood of the origin.
%(This easily follows from the definition 
%of $\gamma$-part and Lemma~\ref{le:6.6}.)
\end{enumerate}
%%%%%%%%%%%%%%%%%%%%%%%
\end{remark}
%%%%%%%%%%%%%%%%%%%%%%%%%%%%%%%%%%%%%%%%%%%
\begin{remark}\label{rem:2.15}
The $\gamma$-part $f_{\gamma}$ can be uniquely extended
to be the smooth function defined on a 
much wider region 
$U\cup\{x\in\R^n:
|x_j|<\delta \mbox{ for $j\in V(\gamma)$} \}
$
with the property (\ref{eqn:2.13}) 
for $t\in\R$, 
where $\delta$ is a positive number and 
$V(\gamma):=\{j:a_j=0\}$. 
Here $(a,l)$ is a valid pair defining $\gamma$.
In particular,
if $\gamma$ is compact 
($\Leftrightarrow V(\gamma)=\emptyset$), 
then $f_{\gamma}$ becomes a polynomial and 
is defined on the whole space $\R^n$.
Note that $V(\gamma)$ is independent of the selection 
of valid pairs. 
Hereafter, this extended function will be also 
denoted by $f_{\gamma}$.
\end{remark}
%%%%%%%%%%%%%%%%%%%%%%%%%%%%%%%%%%%%%%%%%%%

%%%%%%%%%%%%%%%%%%%%%%%%%%%%%%%%%%%%%%%%%%%%%%%%%%%%%%%%%%%%%%%%%%%%%%%%%%%%%%%%%%%%%%%%%%%%%%%%%%%%%%%%%%%%%%%%%%%%%%%%%%%%%%%%%
\subsection{The classes $\hat{\mathcal E}[P](U)$ and $\hat{\mathcal E}(U)$}  
\label{subsec:2.5}

Let 
$P$ be a polyhedron (possibly an empty set)
in $\R^n$ satisfying 
$P+\R_+^n\subset P$ when $P\neq\emptyset$.
Let $U$ be an open neighborhood of the origin. 
%%%%%%%%%%%%%%%%%%%%%%%%
\begin{definition}\label{def:2.16}
Denote by 
${\mathcal E}[P](U)$
the set of smooth functions on $U$ 
whose Newton polyhedra are contained in $P$.
Moreover, when $P\neq \emptyset$, we denote by
$
\hat{{\mathcal E}}[P](U)
$
the set of the elements $f$
in ${\mathcal E}[P](U)$
such that $f$ admits the $\gamma$-part on some neighborhood of 
the origin for any face 
$\gamma$ of $P$. 
When $P=\emptyset$, 
$\hat{{\mathcal E}}[P](U)$ is defined to be the set 
$\{0\}$, i.e., 
the set consisting of only the function identically 
equaling zero on $U$. 
\end{definition}
%%
%%%%%%%%%%%%%%%%%%%%%%%%%%%%%%%%%%

We summarize properties of 
the classes ${\mathcal E}[P](U)$ and 
$\hat{\mathcal E}[P](U)$, which 
can be directly seen from their definitions:
\begin{enumerate}
\item 
$\hat{{\mathcal E}}[\R_+^n](U)=
{\mathcal E}[\R_+^n](U)=C^{\infty}(U);$
\item 
If $P_1$, $P_2\subset\R_+^n$ 
are polyhedra with $P_1\subset P_2$, then 
${\mathcal E}[P_1](U)\subset {\mathcal E}[P_2](U)$ and 
$\hat{{\mathcal E}}[P_1](U)\subset \hat{{\mathcal E}}[P_2](U);$ 
\item 
$(C^{\omega}(U)\cap{\mathcal E}[P](U))\subsetneq 
\hat{{\mathcal E}}[P](U)\subsetneq{\mathcal E}[P](U);$
\item
${\mathcal E}[P](U)$ and $\hat{{\mathcal E}}[P](U)$
are $C^{\infty}(U)$-modules and ideals of $C^{\infty}(U)$.
%%%%%%%%%%%%%%%%%%%%%%%%%%%%%%%%%%%
%%%%
%%%%%
\end{enumerate}
The following is an important property of the class 
$\hat{\mathcal E}[P](U)$.
%%%%%%%%%%%%%%%%%%%%
\begin{proposition}[\cite{kn13}] \label{pro:2.17}
If $P$ is a nonempty polyhedron, then
the following conditions are equivalent.
\begin{enumerate}
\item 
$f$ belongs to the class 
$\hat{{\mathcal E}}[P](U)$;
\item 
There exist a finite set $S$ in 
$P\cap \Z_+^n$ and 
$\psi_{p}\in C^{\infty}(U)$ for $p\in S$ 
such that
\begin{equation}
f(x)=\sum_{p\in S} x^{p}\psi_{p}(x).
\label{eqn:2.14}
\end{equation} 
\end{enumerate}
Note that the above expression is not unique. 
\end{proposition}
%%%%%%%%%%%%%%%%%%%%

\begin{definition}\label{def:2.18}
%%%%%
$
\hat{\mathcal E}(U):=
\{f\in C^{\infty}(U):f\in\hat{\mathcal E}[\Gamma_+(f)](U)\}.
$ 
\end{definition}
%%%%%
It is easy to see the following properties of 
the class $\hat{\mathcal E}(U)$.
\begin{enumerate}
\item 
$C^{\omega}(U)\subsetneq        
\hat{{\mathcal E}}(U)\subsetneq
C^{\infty}(U);$
\item
When $f$ is flat but $f\not\equiv 0$, 
$f$ does not belong to $\hat{\mathcal E}(U)$.
\end{enumerate}
%%%%%%%%%%%%%%%%%%%%%%%%

The class $\hat{\mathcal E}(U)$ contains many kinds of 
smooth functions.
%%%%%%%%%%%%%%%%%%%%
\begin{itemize}
\item
$\hat{{\mathcal E}}(U)$ contains 
the function identically equaling zero
on $U$. 
\item
Every real analytic function defined on $U$ 
belongs to $\hat{\mathcal E}(U)$. 
(From (vi) in Remark~\ref{rem:2.14}.)
\item 
If $f\in C^{\infty}(U)$ is convenient, then
$f$ belongs to $\hat{\mathcal E}(U)$. 
(In this case, every proper noncompact face of $\Gamma_+(f)$ 
can be expressed by the intersection of $\Gamma_+(f)$ and 
some coordinate hyperplane. Therefore,
(iii), (v) in Remark~\ref{rem:2.14} imply this assertion.)
\item
In the one-dimensional case, 
every nonflat smooth function
belongs to $\hat{\mathcal E}(U)$. 
(This is a particular case of the above convenient case.) 
\item
The Denjoy-Carleman (quasianalytic) classes
are contained in $\hat{\mathcal E}(U)$. 
(See Proposition~6.10 in \cite{kn13}.) 
\end{itemize}

Unfortunately, 
the algebraic structure of $\hat{\mathcal{E}}(U)$ is poor.
Indeed, it is not
closed under addition.
For example, consider
$f_1(x_1,x_2)=x_1+x_1 \exp (-1/x_2^2)$
and
$f_2(x_1,x_2)=-x_1$.
Indeed, both
$f_1$ and $f_2$ belong to $\hat{\mathcal{E}}(U)$, but
$f_1+f_2(=\exp(-1/x_2^2))$ 
does not belong to $\hat{\mathcal{E}}(U)$.

%%%%%%%%%%%%%%%%%%%%%%%%%%%
\subsection{Examples}\label{subsec:2.6}

\subsubsection{Example 1}
In order to show what kinds of functions belong to the classes
$\hat{\mathcal{E}}(U)$, $\hat{\mathcal{E}}[P](U)$, 
let us consider the following two-dimensional example, 
which lives near the border of these classes:

\begin{equation*}
\begin{split}
&f_k(x)=f_k(x_1,x_2):=x_1^2x_2^2+x_1^ke^{-1/x_2^2},
\quad 
k\in\Z_+; \\
&P=
\{(\alpha_1,\alpha_2)\in\R_+^2:
\alpha_1\geq 1, \,\, \alpha_2\geq 1\}.
\end{split}
\end{equation*}
Of course, $f_k$ is not real analytic around the origin. 
The set of the proper faces of $\Gamma_+(f_k)$ and $P$ consists of  
$\gamma_1,\gamma_2,\gamma_3$ and $\tau_1,\tau_2,\tau_3$, where
\begin{equation*}
\begin{split}
&\gamma_1=\{(2,\a_2):\a_2\geq 2\},
\,\,
\gamma_2=\{(2,2)\},\,\,
\gamma_3=\{(\a_1,2):\a_1\geq 2\},\\
&
\tau_1=\{(1,\a_2):\a_2\geq 1\}, \,\,
\tau_2=\{(1,1)\},\,\,
\tau_3=\{(\a_1,1):\a_1\geq 1\}.
\end{split}
\end{equation*}
It is easy to see that if $j=2,3$, 
then $f_k$ admits the $\gamma_j$-part
and $\tau_j$-part near the origin for all $k\in\Z_+$ and 
they are written as 
$
(f_k)_{\gamma_2}(x)=(f_k)_{\gamma_3}(x)=x_1^2x_2^2$
and 
$ 
(f_k)_{\tau_2}(x)=(f_k)_{\tau_3}(x)\equiv 0.
$
Consider the $\gamma_1$-part and $\tau_1$-part of $f_k$ for
$k\in\Z_+$.
The situation depends on the parameter $k$ as follows.
\begin{itemize}
\item $(f_0)_{\gamma_1}$ and $(f_0)_{\tau_1}$ cannot be defined.
\item $(f_1)_{\gamma_1}$ cannot be defined but 
$(f_1)_{\tau_1}(x)=x_1e^{-1/x_2^2}$.
\item $(f_2)_{\gamma_1}(x)=f_2(x)$ and $(f_2)_{\tau_1}(x)\equiv 0$.
\item If $k\geq 3$, then 
$(f_k)_{\gamma_1}(x)=x_1^2x_2^2$ and $(f_k)_{\tau_1}(x)\equiv 0$.
\end{itemize}
From the above, we see that
$f_k\in\hat{{\mathcal E}}(U)$ if and only if $k\geq 2$; 
$f_k\in\hat{{\mathcal E}}[P](U)$ if and only if $k\geq 1$.
Notice that  
$\tau_1\cap \Gamma_+(f_1)=\emptyset$ but
$(f_1)_{\tau_1}(x)=x_1e^{-1/x_2^2}\not\equiv 0$.

%%%%%%%%%%%%%%%%%%%%%%%%%%%%%%%%%%%%%%%%%
\subsubsection{Example 2}

Let $f$ be a smooth function defined near the origin in $\R^3$.
As mentioned in Remark~2.14 (v), 
when $\gamma$ is a face defined by 
the intersection of $\Gamma_+(f)$ and 
some coordinate hyperplane, 
$f$ always admits the $\gamma$-part. 
But, the condition that 
$\gamma$ is a noncompact face of $\Gamma_+(f)$ contained 
in some coordinate hyperplane
is not sufficient for the admission of the $\gamma$-part.
Indeed, the following three-dimensional example shows this 
subtle situation:
$$f(x_1,x_2,x_3)=x_1^2+e^{-1/x_2^2} \mbox{ and }
\gamma=\{(2,\alpha_2,0):\alpha_2\geq 0\}.$$
It is easy to check that the valid pair $((1,0,1),2)$ 
defines the face $\gamma$, but
the limit (\ref{eqn:2.12}) does not exists 
when $x_2\neq 0$.
This implies that the above $f$ does not belong to
the class $\hat{\mathcal{E}}(U)$.
%%%%%%%%%%%%%%%%%%%%%%%%%%%%%%%%%%%%%%%%%%%%%%%%%%%%%%%%%%%%%%%%%%%%%%%%%%%%%%%%%%%%%%%%%%%%%%%%%%%%%%%%%%%%%%%%%%%%%%%%%%%%%%%%%%%%%%%%%%%%%%%%

\section{Earlier studies}\label{sec:3}

As mentioned in the Introduction, 
there are many kinds of generalizations of 
the work of Varchenko in \cite{var76} %, \cite{agv88}
relating to the behavior of
the oscillatory integral $I(t;\varphi)$
in (\ref{eqn:1.1}).
In order to reveal a motivation of our investigation,
after stating results of Varchenko, 
we explain some earlier results in 
\cite{var76},\cite{vas79},\cite{agv88},\cite{py04},\cite{ckn13},\cite{kn13}, 
which are deeply connected with our study.
Note that some conditions in the assumptions of the previous results
can be considered as typical cases of the assumptions 
in our new theorems in Section~\ref{sec:4}, so
they are sometimes more useful for the practical applications.

Throughout this section, 
the following three conditions are assumed:
Let $U$ be an open neighborhood of the origin in $\R^n$.
%%%%%%%%%%%%%%%%%%%%%%%%%
\begin{enumerate}
\item[(A)] 
$f$ is a nonflat smooth function defined on $U$ satisfying 
that $f(0)=|\nabla f(0)|=0$; %and $\Gamma_+(f)\neq\emptyset$;
\item[(B)] 
$g$ is a nonflat smooth function defined on $U$; 
%satisfying $\Gamma_+(g)\neq\emptyset$;
%%
\item[(C)]
$\varphi$ is a smooth function 
whose support is contained in $U$. 
\end{enumerate}
%%%%%%%%%%%%%%%%%%%%%%%%%%
%

\subsection{Results of Varchenko}\label{subsec:3.1}
Let us
recall a part of famous results due to 
Varchenko in \cite{var76} 
and Arnold, Gusein-Zade and Varchenko \cite{agv88}
in the case 
when $f$ is real analytic on $U$ and $g\equiv 1$.
These results require the following condition. 
%%%%%%%%%%%%%
%%%%%%%%%%%%%%%%%%%%%%%%%%%
\begin{enumerate}
\item[(D)] $f$ is real analytic on $U$
and is {\it nondegenerate} over $\R$ with respect to 
the Newton polyhedron $\Gamma_+(f)$, i.e.,
for every compact face $\gamma$ of $\Gamma_+(f)$, 
the $\gamma$-part $f_{\gamma}$ satisfies 
   \begin{equation}\label{eqn:3.1}
   \nabla f_{\gamma}=
   \left(
   \frac{\d f_{\gamma}}{\d x_1},\ldots,
   \frac{\d f_{\gamma}}{\d x_n}\right)
   \neq (0,\ldots,0)\quad
   \mbox{on the set $(\R\setminus\{0\})^n$.}
   \end{equation}
\end{enumerate}

%%%%%%%%%%%%%%%%%%%%%%%%%%
\begin{theorem}
[\cite{var76},\cite{agv88}]\label{thm:3.1}
If $f$ satisfies the condition $(D)$, 
%$f$ is real analytic on $U$ and nondegenerate
%over $\R$ with respect to its Newton polyhedron. 
then the following hold$:$
%%%%%%%%%%%%%%%%%%%%%%%%%%
\begin{enumerate}
\item[(i)]
The progression $\{\a\}$
in $(\ref{eqn:1.2})$
belongs to finitely many arithmetic progressions,
which are obtained by using the theory of 
toric varieties based on the geometry of
the Newton polyhedron $\Gamma_+(f)$.
$($See Remark~\ref{rem:4.2}, below.$)$
\item[(ii)] $\beta(f,1)\leq -1/d(f);$
\item[(iii)] If at least one of the following conditions 
is satisfied:
\begin{enumerate}
\item $d(f)>1;$
\item $f$ is nonnegative or nonpositive on $U;$
\item 
$1/d(f)$ is not an odd integer and $f_{\tau_*}$ 
does not vanish on $U\cap(\R\setminus\{0\})^n$,
where $\tau_*$ is the principal face of $\Gamma_+(f)$
$($see Remark~\ref{rem:2.11}$)$,
\end{enumerate}
then 
$\beta(f,1)=-1/d(f)$ and $\eta(f,1)=m(f)$.
\end{enumerate}
%%%%%%%%%%%%%%%%%%%%%%%%%
\end{theorem}
%%%%%%%%%%%%%%%%%%%%%%%%%

There have been many kinds of studies, 
which generalize or improve the above theorem. 
Now,
let us consider the following two kinds of generalizations of 
the Varchenko's result: 
\begin{itemize}
\item Consider the weighted case 
under the real analytic assumption on the phase.
\item Weaken the regularity of the phase 
in the unweighted case. 
\end{itemize}

%%%%%%%%%%%%%%%%%%%%%%%%%%%%%%%%%%%%%%%%%%%%%%
\subsection{Weighted case}\label{subsec:3.2}

%%%%%
The following theorem naturally generalizes the assertion (ii) 
in Theorem~3.1.
%%%%%%%%%%%%%%%%%%%%%%%%%%%%%%%%%%%%%%%%%%%%%%%%%%%%%%%%%%%
\begin{theorem}[\cite{ckn13}]\label{thm:3.2}
Suppose that {\rm (i)} $f$ satisfies the condition $(D)$ and
{\rm (ii)} at least one of the following conditions is 
satisfied$:$
   \begin{enumerate}
   \item[(a)] 
      $f$ is convenient$;$
   \item[(b)] 
      $g$ is convenient$;$
   \item[(c)] 
      $g$ is real analytic on $U;$
   \item[(d)] 
      $g$ is expressed as 
      $g(x)=x^p\tilde{g}(x)$ 
      on $U$, where 
      $p\in\Z_+^n$ and $\tilde{g}$ is a smooth function 
      defined on $U$ with $\tilde{g}(0)\neq 0$.
   \end{enumerate}  
%If the support of $\varphi$ is contained in 
%a sufficiently small neighborhood of the origin,
%then there exists a positive constant $C(\varphi)$ independent of 
%$t$ such that 
%%%%%
%\begin{equation*}
%   |I(t;\varphi)|\leq C(\varphi)t^{-1/d(f,g)}
%   (\log t)^{m(f,g)-1}
%   \quad \mbox{
%   for $t\geq 1$.}
%\end{equation*}
%%%%%
Then, we have
$\b(f,g)\leq -1/d(f,g)$.   
\end{theorem}
%% 
%\begin{remark}
%In \cite{}, only weaker estimate was obtained. 
%The above more accurate estimate can be obtained
%through the same computation of the coefficient 
%as in \cite{kntro}.
%\end{remark}

The following theorem partially generalizes the assertion (iii)
in Theorem~3.1.
%%%%%%%%%%%%%%%%%%%%%%%%%%%%%%%%%%%%%%%
\begin{theorem}[\cite{ckn13}]\label{thm:3.3}
   Suppose that 
{\rm (i)} 
$f$ satisfies the condition $(D)$,
{\rm (ii)} 
      at least one of the following two conditions 
      is satisfied$:$ 
      \begin{enumerate}
      \item[(a)] 
      $f$ is convenient and $g_{\gamma_*}$ is nonnegative or nonpositive 
     on $U$ for \underline{all} principal faces $\gamma_*$ 
     of $\Gamma_+(g);$          
     \item[(b)] 
         $g$ is expressed as $g(x)=x^p\tilde{g}(x)$ on $U$, 
         where every component of $p\in\Z_+^n$ is even and 
         %every component of $p\in\Z_+^n$ is even and 
         $\tilde{g}$ is a smooth function defined on $U$ 
         with $\tilde{g}(0)\neq 0$
      \end{enumerate}
and   
{\rm (iii)}
      at least one of the following two conditions 
      is satisfied$:$ 
%%%
      \begin{enumerate}
      \item[(c)]
         $d(f,g)>1;$
      \item[(d)]
         $f$ is nonnegative or nonpositive on $U$.
      \end{enumerate} 
   Then  
   the equations 
   $\beta(f,g)=-1/d(f,g)$ and 
   $\eta(f,g)= m(f,g)$ hold. 
%%%%%
\end{theorem}
%%%%%%%%%%%%%%%%%%%%%%%%%%%%%%%%%%%%%%%%%%%%%%%%%%%%%%%%

\begin{remark}\label{rem:3.4}
Similar results to the above two theorems 
have been obtained in 
\cite{vas79}, \cite{agv88}.
In our language, the results in \cite{agv88} 
can be stated as follows. 

\vspace{0.5em}
(Theorem 8.4 in \cite{agv88}, p 254)
{\it If 
$f$ is real analytic and is nondegenerate over $\R$ 
with respect to its Newton polyhedron,
then 
\begin{enumerate}
\item $\beta(f,g)\leq -1/d(f,g)$;
\item If $d(f,g)>1$ and 
$\Gamma_+(g)=\{p\}+\R_+^n$ with $p\in\Z_+^n$, 
then $\beta(f,g)= -1/d(f,g).$
\end{enumerate}
}
\vspace{0.5em}

Unfortunately, more additional assumptions are 
necessary to obtain the above assertions (i), (ii). 
Indeed, consider 
the following two-dimensional example: 
\begin{equation*}
f(x_1,x_2)=x_1^4;\quad
g(x_1,x_2)=x_1^2x_2^2+e^{-1/x_2^2},
\end{equation*}
which is a special case of the example in 
Section~15.1.
It follows from the computation in Section~15.1
that this example violates (i), (ii). 
As for (ii), even if $g$ is real analytic, 
the one-dimensional case indicates 
that at least some condition on the power $p$ 
is needed. 
Indeed, consider the case that
$f(x)=x^2$ and $g(x)=x^p$ 
where $p$ is an odd positive integer.
In this case, the coefficient of the first candidate 
term $t^{-(p+1)/2}$ of the asymptotic expansion 
(\ref{eqn:})
vanishes (see \cite{ckn13}, Section~7.4 for the details).
It is easy 
to find counterexamples in higher dimensional case.
\end{remark}

\begin{remark}\label{rem:3.5}
Pramanik and Yang \cite{py04} obtained a similar result 
relating to the above equation 
``$\beta(f,g)=-1/d(f,g)$'' 
in the case when 
the dimension is two  
and the weight has the form $g(x)=|h(x)|^{\epsilon}$ where
$h$ is real analytic and $\epsilon$ is positive. 
%Their result in Theorem 3.1 (a) 
%does not need any additional assumptions. 
%We explain this reason roughly. 
Their approach is based on the Puiseux series expansions 
of the roots of $f$ and $h$, 
which is inspired by the work of
Phong and Stein in \cite{ps97}. 
Their definition of {\it Newton distance}, 
which is different from ours, 
is given through the process of 
a good choice of coordinate system.
As a result, their result does not 
need the nondegeneracy condition of the phase.  
%In Section~, we will consider the case of the weight of 
%the form $g(x)=|h(x)|^{\epsilon}$ 
%by means of our Newton distance. 
\end{remark}

%The following theorem is not a generalization of Theorem~, but 
The following theorem shows an interesting 
``symmetry property'' with respect to 
the phase and the weight. 
%%%%%%%%%%%%%%%%%%%%%%%%%%%%%%%%%%%%%%%%%%%%%%%%%%%%%%%%
\begin{theorem}[\cite{ckn13}]\label{thm:3.6}
Suppose that $f$, $g$ satisfy the condition $(D)$ 
and that they are convenient and nonnegative or nonpositive on $U$.
%If the support of $\chi$ is
%contained in a sufficiently small neighborhood 
%of the origin, then 
Then we have
$\b(x^{\1}f,g)\b(x^{\1}g,f)\geq 1$,
where $x^{\1}=x_1\cdots x_n$.
Moreover, the following two conditions are equivalent$:$  
\begin{enumerate}
\item $\b(x^{\1}f,g)\b(x^{\1}g,f)=1$$;$ 
\item There exists a positive rational number $d$ such that 
$\Gamma_+(x^{\1} f)=d\cdot\Gamma_+(x^{\1} g)$. 
\end{enumerate}
If the condition {\rm (i)} or {\rm (ii)} is satisfied, 
then we have
$\eta(x^{\1}f,g)=\eta(x^{\1}g,f)=n$.
\end{theorem}

%%%%%%%%%%%%%%%%%%%%%%%%%%%%%%%%%%%%%%%%%%%%%%%%%%%%%%%%
\subsection{Weakened regularity of the phase}\label{subsec:3.3}

Let us consider the case when the phase satisfies a weaker 
regularity condition:
%%%
%%%%%
\begin{enumerate}
\item[(E)] $f$ belongs to the class $\hat{\mathcal{E}}(U)$
and is nondegenerate over $\R$ with respect to 
its Newton polyhedron. % $\Gamma_+(f)$.
\end{enumerate}
%%%%%
It is known in \cite{kn13} that
$I(t;\varphi)$ also has an asymptotic expansion 
in the case when the phase satisfies the above condition.
%condition $(D)$,.
The following theorem generalizes the assertion (i)
in Theorem~\ref{thm:3.1}.

%%%%%%%%%%%%%%%%%%%%%%%%%
\begin{theorem}[\cite{kn13}]\label{thm:3.7}
%th:3.1.5****************
If $f$ satisfies the condition $(E)$
and the support of $\varphi$ is contained in a 
sufficiently small neighborhood of the origin,
then $I(t;\varphi)$ admits an asymptotic expansion
of the form $(\ref{eqn:1.2})$,
where
%the coefficients $C_{\a k}(\varphi)$ are
%distributions of $\varphi$ and
$\{\a\}$ belongs to the same progressions 
as in the case when the phase is $f_{\Gamma(f)}$
as in $(\ref{eqn:2.4})$.
$($Since $f_{\Gamma(f)}$ is a polynomial,
the progressions can be exactly constructed 
as in \cite{var76}.$)$
\end{theorem}
%%%%%%%%%%%%%%%%%%%%%%%%%

Furthermore, 
Varchenko's results can be directly generalized to
the case when the phase belongs to the class $\hat{\mathcal{E}}(U)$.
%%%%%%%%%%%%%%%%%%%%%%%%%%%%%%%%%%%%%%%%%%%%%
\begin{theorem}[\cite{kn13}]\label{thm:3.8}
If $f$ satisfies the condition $(E)$ and $g\equiv 1$, then
the assertions {\rm (ii),(iii)} in Theorem~3.1 hold.
\end{theorem}
%%%%%%%%%%%%%%%%%%%%%%%%%%%%%%%%%%%%%%%%%%%%%

\begin{remark}\label{rem:3.9}
In \cite{kn13}, more precise results are obtained. 
\end{remark}

\begin{remark}\label{rem:3.10}
Some kind of restrictions to the regularity of the phase
is necessary in the above theorem. 
Consider the following two-dimensional example:
$
f(x_1,x_2)
=
x_1^2+ 
e^{-1/|x_2|^{\alpha}}$ $(\alpha>0)$ and 
$g\equiv 1$,
which is given by 
Iosevich and Sawyer in \cite{is97}.
%%%
Note that the above $f$ 
satisfies the nondegeneracy condition 
(\ref{eqn:3.1}) but
it does not belong to $\hat{\mathcal E}(U)$. 
%Therefore, 
%the existence of asymptotic expansions
%is not obtained for the integral $I(t;\varphi)$. 
It is easy to see the following:
$d(f)=2$, $m(f)=1$,
$f_{\tau_*}(x_1,x_2)=x_1^2$.
%%%
It was shown in \cite{is97} that  
$
|I(t;\varphi)|\leq Ct^{-1/2}(\log t)^{-1/\alpha} 
$
for $t\geq 2$.
In particular, 
we have $\lim_{t\to\infty} t^{1/2}I(t;\varphi)=0$.
The pattern of an asymptotic expansion 
of $I(t;\varphi)$ 
in this example might be different from that in (\ref{eqn:1.2}).
\end{remark}

%%%%%%%%%%%%%%%%%%%%%%%%%%%%%%%%%%%%%%%%%%%%%%%%%%%
\section{Main results}\label{sec:4}

By mixing two kinds of generalizations in the previous section
and deeper understanding the resolution of singularities
for the phase and the weight,
we can generalize and improve the results in \cite{ckn13}. 
Furthermore, 
the results can be stated in more clear form by using the class 
$\hat{\mathcal{E}}(U)$. 
This means that properties of $\hat{\mathcal{E}}(U)$
play crucial roles in the sufficient condition on 
the phase and the weight. 
Throughout this section, 
the three conditions (A), (B), (C) 
at the beginning of Section~3 are assumed, where
$U$ is an open neighborhood of the origin in $\R^n$.

First, let us give a sharp estimate for $I(t;\varphi)$. 
Since the class $\hat{\mathcal{E}}(U)$ contains 
many kinds of smooth functions as in Section~2.5, 
the following theorem generalizes and improves 
Theorem~\ref{thm:3.2} in Section~3.
%%%%%%%%%%%%%%%%%%%%%%%%%
\begin{theorem}\label{thm:4.1}
Suppose that 
{\rm (i)} $f$ satisfies the condition $(E)$
(see Section~3.3) 
 and
{\rm (ii)} 
at least one of the following two conditions is satisfied$:$
\begin{enumerate}
\item[(a)]
$g$ belongs to the class
$\hat{\mathcal E}(U);$
\item[(b)]
$f$ is convenient.
\end{enumerate}
If the support of $\varphi$ is
contained in a sufficiently small neighborhood 
of the origin, then
there exists a positive constant $C(\varphi)$ independent of 
$t$ such that 
%%%%%
\begin{equation}\label{eqn:4.1}
   |I(t;\varphi)|\leq C(\varphi)t^{-1/d(f,g)}
   (\log t)^{m(f,g)-1}
   \quad \mbox{
   for $t\geq 2$.}
\end{equation}
%%%%%
In particular, we have
$\b(f,g)\leq -1/d(f,g)$.
\end{theorem}
%%%%%%%%%%%%%%%%%%%%%%%%%%%%%%%%%%%%%%%%%%%%%%%%%%%%%%%
%%%%%%%%%%%%%%%%%%%%%%%%%
\begin{remark}\label{rem:4.2}
From the proof of the above theorem, 
we can see that, under the same assumptions, 
the set $\{\a\}$ in the asymptotic expansion (\ref{eqn:1.2})
belongs to the following set:
%%%%%
\begin{equation*}
\left\{-\frac{l_g(a)+\langle a\rangle+\nu}{l_f(a)}:
\nu\in \Z_+,a\in \tilde{\Sigma}^{(1)}\right\}
\cup(-\N),
\end{equation*}
%%%%%
where $l_g(a),l_f(a)\in \Z_+$ are as in $(\ref{eqn:7.1})$
and $\tilde{\Sigma}^{(1)}$ is a finite set of vectors
in $\Z_+^n$ defined in Theorem~\ref{thm:10.1}.
Of course, the same assertion holds in the cases of 
Theorem~\ref{thm:3.1} (i) and Theorem~\ref{thm:3.7}.
Note that $l_g(a)=0$ for any $a\in\Z_+^n$ if $g(0)\neq 0$. 
\end{remark}
%%%%%%%%%%%%%%%%%%%%%%%%%
\begin{remark}\label{rem:4.3}
Let us consider the case when
the weight $g$ is expressed as 
$g(x)=x^p\psi(x)$ where 
$p\in\Z_+^n$ and
$\psi\in C^{\infty}(U)$. 
(In this case, 
$g$ does not belong to $\hat{\mathcal E}(U)$
if and only if $\psi$ is a nonzero flat function.) 
Then, without the assumption (ii), 
the following estimate is obtained
under the same assumptions:
%%%
\begin{equation}\label{eqn:4.2}
   |I(t;\varphi)|\leq C(\varphi)t^{-1/d(f,x^p)}
   (\log t)^{m(f,x^p)-1}
   \quad \mbox{
   for $t\geq 2$.}
\end{equation}
%%%
This estimate will be shown in Section~\ref{subsec:14.3} 
(see also Remark~\ref{rem:10.3}). 
Note that these assertions hold even if $\psi$ 
is a nonzero flat function. 
%(equivalently $g$ is flat). 
In this case, $g$ does not belong to 
$\hat{\mathcal E}(U)$, because
$g$ is also a nonzero flat function.  
When $\psi$ vanishes at the origin
in high order
(in particular $\psi$ is flat), 
the reader might feel that 
the estimate (\ref{eqn:4.2}) is not optimal. 
But, we will give a simple example
showing some kind of the optimality of the
estimate (\ref{eqn:4.2}) even in the flat case
(see Example~1 in Section~\ref{sec:15}). 
\end{remark}
%%%%%%%%%%%%%%%%%%%%%%%%%

Next, let us consider the case when
the equality $\beta(f,g)=-1/d(f,g)$ holds. 
The following theorem generalizes and improves 
Theorem~3.3 in Section~3.
%%%%%%%%%%%%%%%%%%%%%%%%%%%%%%%%%%%%%%%%%%%%%%%%%%%%%%%
\begin{theorem}\label{thm:4.4}
Suppose that the conditions {\rm (i)}, {\rm (ii)} in Theorem~4.1 
are satisfied, 
{\rm (iii)}
there exists a principal face $\gamma_*$ of $\Gamma_+(g)$ such that 
   $g_{\gamma_*}$ is nonnegative or nonpositive on $U$ and
%  on $U\cap(\R\setminus\{0\})^n$ and   
{\rm (iv)} 
   at least one of the following three conditions 
   is satisfied$:$ 
%%%%%%%%%%%%%%%%%%%%%
   \begin{enumerate}
   \item[(a)]
      $d(f,g)>1;$
   \item[(b)]
      $f$ is nonnegative or nonpositive on $U;$
   \item[(c)] 
      $1/d(f,g)$ is not an odd integer and  
      $f_{\tau_*}$ does not vanish on $U\cap(\R\setminus\{0\})^n$
where $\tau_*$ is a principal face of $\Gamma_+(f)$ associated to 
$\gamma_*$ in {\rm (iii)}.
%%%%%%%%%%%%%%%%%%%%%% 
\end{enumerate}

%%%%%%%%%%%%%%%%%%%%%
%%
Then
the equations
$\beta(f,g)=-1/d(f,g)$
and $\eta(f,g)= m(f,g)$ hold. 
\end{theorem}
%%%%%%%%%%%%%%%%%%%%%%%
%%%%%%%%%%%%%%%%%%%%%%%
\begin{remark}\label{rem:4.5}
In Theorem~\ref{thm:14.2} in Section~\ref{sec:14}, 
we give explicit formulae for the coefficient 
of the leading term of 
the asymptotic expansion (\ref{eqn:1.2})
under the assumptions (i)-(iii).
These explicit formulae show that 
the above coefficient essentially depends 
on the principal face-parts $f_{\tau_*}$ and $g_{\gamma_*}$. 
The above (i)-(iv) are sufficient
conditions for the nonvanishing of the leading term. 
\end{remark}
%%%%%%%%%%%%%%%%%%%%%%%
%%%%%%%%%%%%%%%%%%%%%%%

\begin{remark}\label{rem:4.6}
Let us discuss about the necessities of the hypotheses (i)-(iii). 

(i)\quad 
As explained in Remark~\ref{rem:3.10},
some kind of restrictions to the regularity of the phase
is necessary for the two equalities in the above theorem. 

(ii)\quad
When $f$ is convenient, the strong assumption on 
the weight $g$ is no longer needed. 
Without the convenience of $f$, 
it is a complex issue to give ``good '' sufficient conditions
on $g$.
These problems will be discussed in Section~12.1.
Furthermore, in Section~\ref{sec:15}, 
we give a few examples showing the subtlety of 
the sufficient condition (a): 
``$g\in\hat{\mathcal E}(U)$''.
Now, let us assume that $f,g$ satisfy
all the conditions in Theorem~\ref{thm:4.4} 
except the condition (ii). 
First, we show some kind of necessity of 
the condition (a). 
Example~1 shows the existence of $f$ and $g$ 
such that $g\not\in\hat{\mathcal E}(U)$
and the equations: $\beta(f,g)=-1/d(f,g)$ in the theorem does not hold.
Next, we show that there is room for improvement on the condition (a).
By observing Example~2, 
it is natural to feel that
(a) is too strong and to expect that 
this condition can be weakened 
by the condition: 
\begin{enumerate}
\item[(a$'$)]
$g$ belongs to the class
$\hat{\mathcal E}[\Phi^{-1}(\Gamma_+(f))\cap\R_+^n](U)$.
\end{enumerate}
(see the definition of $\Phi$ in (\ref{eqn:2.6})).
Unfortunately, Example~3 violates this expectation.

These kinds of subtle situation, which are
seen in the above (i) and (ii), occurs 
from the fact that 
the geometry of the Newton polyhedra do not 
provide 
sufficient analytic information 
about smooth functions. 
%We will discuss about this problem in Section~ in detail. 

(iii)\quad 
The most interesting condition in (iii) 
is (c). 
The necessity of the non-oddness of $1/d(f,g)$ is shown 
by means of the example, 
which was given by Pramanik and Yang \cite{py04} (Section~6).
Actually, consider the two-dimensional example:
$f(x_1,x_2)=x_1x_2$, $g(x_1,x_2)=x_2^2$. 
In this case, $d(f,g)=1$,
$f_{\tau_*}(x_1,x_2)=x_1x_2$, $g_{\gamma_*}(x_1,x_2)=x_2^2$. 
Though it is easy to see all the other conditions in the assumptions 
in the above theorem, 
the oscillation index $\beta(f,g)$ equals $-3$
($\neq -1/d(f,g)=-1$).
\end{remark}

Finally, Theorem~\ref{thm:3.6} 
can be generalized in the following form. 

%%%%%%%%%%%%%%%%%%%%%%%%%%%
\begin{theorem}\label{thm:4.7}
Suppose that $f$, $g$ satisfy the condition $(E)$ and that
they are nonnegative or nonpositive on $U$. 
Then we have
$\b(x^{\1}f,g)\b(x^{\1}g,f)\geq 1$.
Moreover, the following two conditions are equivalent$:$  
%%%%%%%%%%%%%%%%%%%%%%%%%%%
\begin{enumerate}
\item $\b(x^{\1}f,g)\b(x^{\1}g,f)=1;$ 
\item There exists a positive rational number $d$ such that 
$\Gamma_+(x^{\1} f)=d\cdot\Gamma_+(x^{\1} g)$. 
\end{enumerate}
%%%%%%%%%%%%%%%%%%%%%%%%%%%%
If the condition {\rm (i)} or {\rm (ii)} 
is satisfied, then we have
$\eta(x^{\1}f,g)=\eta(x^{\1}g,f)=n$.
\end{theorem}
%%%%%%%%%%%%%%%%%%%%%%%%%%%%

\begin{remark}
It is needless to say that
the behevior of $I(t;\varphi)$ is independent of 
the exchanges of the integral variables. 
Therefore, 
if there exists a coordinate in which 
$f$ and $g$ satisfy the assumptions in each theorem,
then the respected assertion holds. 
\end{remark}

%%%%%%%%%%%%%%%%%%%%%%%%%%%%%%%%%%%%%%%%%%%%%%%%%%%%%%%%%%%%%%%%%%%%%%%%%%%%%%%%%%%%%%%%%%%%%%%%%%%%%%%%%%%%%%%%%%%%%%%%%%%%%%%%%%%%%%%%%%%%%
%%%%%%%%%%%%%%%%%%%%%%%%%%%%%%%%%%%%%%%%%%%%%%%%%%%%%%%%%%%%%%%%%%%%%%%%%%%%%%%%%%%%%%%%%%%%%%%%%%%%%%%%%%%%%%%%%%%%%%%%%%%%%%%%%%%%%%%%%%%%
%%%%

%%%%%%%%%%%%%%%%%%%%%%%%%%%%%%%%%%%%%%%%%%%%%%%%%%%%%%%%%%%%%%%%%%%%%%%%%%%%%%%%%%%%%%%%%%%%%%%%%%%%%%%%%%%%%%%%%%%%%%%%%%%%%%%%%%%%%%%%%%%%%

\section{Toric varieties}\label{sec:5}

In the analysis of Varchenko in \cite{var76}, 
the theory of toric varieties plays a crucial role.
%%%
Our analysis is also based on this theory. 
In this section, we recall the method to construct 
a toric variety from a given fan.
Refer to 
\cite{ful93} 
for general theory of toric varieties. 

%%%%%%%%%%%%%%%%%%%%%%%%%%%%%%%%%%%%%%%%%%%%%%%%%%%%%%%%%%%%%%%%%%%%%%%%%%%%%%%%%%%%%%%%%%%%%%%%%%%%%%%%%%%%%%%%%%%%%%%%%%%%%%%%%%%%%%%%%%%%%%

\subsection{Cones and fans}\label{subsec:5.1}

Let us recall the definitions of important terminology:  
{\it cone} and {\it fan}.  

A {\it rational polyhedral cone} 
$\sigma\subset \R^n$ is a cone
generated by finitely many elements of $\Z^n$. 
In other words, 
there are $u_1,\ldots,u_k \in \Z^n$ such that  
%%%%%
$$ 
\sigma=\{\lambda_1 u_1+\cdots+\lambda_k u_k \in {\R}^n:
\lambda_1,\ldots,\lambda_k\geq 0\}. 
$$
%%%%%
We say that $\sigma$ is {\it strongly convex} if 
$\sigma\cap(-\sigma)=\{0\}$. 
By regarding a cone as a polyhedron in $\R^n$, 
the definitions of {\it dimension}, {\it face}, 
{\it edge}, {\it facet} for the cone 
are given by the same way as in Section~\ref{subsec:2.1}.  

The {\it fan} is defined to be a finite collection $\Sigma$ 
of cones in ${\R}^n$ with the following properties:
%%%%%%%%%%%%%%%%%%%
\begin{itemize}
\item
Each $\sigma\in \Sigma$ 
is a strongly convex rational polyhedral cone;
\item 
If $\sigma\in\Sigma$ and $\tau$ is a face of $\sigma$, then 
$\tau\in \Sigma$;
\item 
If $\sigma, \tau\in\Sigma$, 
then $\sigma\cap \tau$ is a face of each.    
\end{itemize}
%%%%%%%%%%%%%%%%%%%
For a fan $\Sigma$,
the union $|\Sigma|:=\bigcup_{\sigma\in\Sigma}\sigma$ 
is called the {\it support} of $\Sigma$. 
For $k=0,1,\ldots,n$, we denote by $\Sigma^{(k)}$  
the set of $k$-dimensional cones in $\Sigma$.  
The {\it skeleton} of a cone $\sigma\in\Sigma$ is 
the set of all of its primitive 
integer vectors 
(i.e., with components relatively prime in $\Z_+$)
in the edges of $\sigma$. 
It is clear that the skeleton of $\sigma$ 
generates $\sigma$ itself
and that the number of the elements of 
skeleton in $\Sigma^{(k)}$ is not less than $k$. 
Thus, the set of skeletons of the cones 
belonging to $\Sigma^{(k)}$ is also expressed 
by the same symbol $\Sigma^{(k)}$.

It is known (see \cite{ful93}) that 
there exists a {\it simplicial subdivision} 
$\tilde{\Sigma}$
of $\Sigma$,  
that is, $\tilde{\Sigma}$ is a fan satisfying the following properties:
%%%%%%%%%%%%%%%%
\begin{itemize}
\item 
The fans $\Sigma$ and $\tilde{\Sigma}$ have the same support; 
\item 
Each cone of $\tilde{\Sigma}$ lies in some cone of $\Sigma$; 
\item 
The skeleton of any cone belonging to $\tilde{\Sigma}$ 
can be completed 
to a base of the lattice dual to $\Z^n$.
\end{itemize}
%%%%%%%%%%%%%%%%

%%%%%%%%%%%%%%%%%%%%%%%%%%%%%%%%%%%%%%%%%%%%%%%%%%%%%%%%%%%%%%%%%%%%%%%%%%%%%%%%%%%%%%%%%%%%%%%%%%%%%%%%%%%%%%%%%%%%%%%%%%%%%%%%%%%%%%%%%%%%%%
%%%%%%%%%%%%%%%%%%%%%%%%%%%%%%%%%%%%%%%%%%%%%%%%%%%%%%%%%%%%%%%%%%%%%%%%%%%%%%%%%%%%%%%%%%%%%%%%%%%%%%%%%%%%%%%%%%%%%%%%%%%%%%%%%%%%%%%%%%%

\subsection{Construction of toric varieties}\label{subsec:5.2}

Let $\Sigma_0$ be a fan satisfying $|\Sigma_0|=\R_+^n$.
Fix a simplicial subdivision $\Sigma$ of $\Sigma_0$. 
For an $n$-dimensional cone $\sigma\in\Sigma$, 
let 
$a^1(\sigma),\ldots,a^n(\sigma)$ be the skeleton of 
$\sigma$, ordered once and for all. 
Here, we set the coordinates of the vector $a^j(\sigma)$ as 
%%%%%
$$
a^j(\sigma)=(a^j_1(\sigma),\ldots,a^j_n(\sigma)).
$$
%%%%%
With every such cone $\sigma$, we
associate a copy of $\R^n$ which is denoted by 
$\R^n(\sigma)$.
We denote by
$ 
\pi(\sigma):\R^n(\sigma) \to \R^n
$
the map defined by  
$\pi(\sigma)(y_1,\ldots,y_n)=
(x_1,\ldots,x_n)$ with
%%%%%
\begin{equation}\label{eqn:5.1} 
x_j=\prod_{k=1}^n y_k^{a_j^k(\sigma)}
= y_1^{a_j^1(\sigma)}\cdots y_n^{a_j^n(\sigma)}
, \quad\quad 
j=1,\ldots,n.
\end{equation}
%%%%%
We remark that the following conditions are equivalent:
%%%%%%%%%%%%%%%%
\begin{itemize}
\item 
$a^1(\sigma),\ldots ,a^{n}(\sigma)$ can be completed 
to a base of the lattice dual to $\Z^n$;
\item
The inverse map of $\pi(\sigma)$ is rational; 
\item
$\det (a^j_k(\sigma))_{1\leq j,k\leq n}=\pm 1$.
\end{itemize}
%%%%%%%%%%%%%%%%
Let $Y_{\Sigma}$ be the union of $\R^{n}(\sigma)$ for $\sigma$
which are glued along the images of $\pi(\sigma)$. 
Indeed, for any $n$-dimensional cones $\sigma,\sigma'\in\Sigma$, 
two copies $\R^n(\sigma)$ and $\R^n(\sigma')$ can be
identified with respect to a rational mapping:
$\pi^{-1}(\sigma')\circ \pi(\sigma):
\R^n(\sigma)\to \R^n(\sigma')$ 
(i.e., $x\in\R^n(\sigma)$ and $x'\in\R^n(\sigma')$ will coalesce 
if $\pi^{-1}(\sigma')\circ\pi(\sigma):x\mapsto x'$).
Then it is known (see \cite{ful93}) that
%%%%%%%%%%%%%%%%%%%%%%%%%%%%%%%%%%%%%%%%
\begin{itemize}
\item
$Y_{\Sigma}$ is an $n$-dimensional 
real algebraic manifold;
\item
The map 
$\pi:Y_{\Sigma}\to\R^n$
defined on each $\R^n(\sigma)$ as 
$\pi(\sigma):\R^n(\sigma)\to\R^n$ is proper;
\item
Each $\R^n(\sigma)$ is densely embedded in $Y_{\Sigma}$.
\end{itemize}
%%%%%%%%%%%%%%%%%%%%%%%%%%%%%%%%%%%%%%%%%
The manifold $Y_{\Sigma}$ is called the 
(real) {\it toric variety} associated with $\Sigma$.    
We call the pair $(Y_{\Sigma},\pi)$ the real resolution 
of singularities associated to $\Sigma$.        
%%%%

The following properties of $\pi(\sigma)$ are useful for 
the analysis in Sections~8--13. 
They can be easily seen, so we omit their proofs.

%%%%%%%%%%%%%%%%%%%%%%%%%%%%%%%%%%%%%%%%%%%
\begin{lemma}\label{lem:5.1}
\begin{enumerate}
\item
The set of the points in $\R^n(\sigma)$
in which $\pi(\sigma)$ is not an isomorphism is
a union of coordinate hyperplanes.
%More precisely,
%$\pi^{-1}(\sigma)(0)$ is expressed as the union of
%$T_I(\R^n)$ for all subsets $I$ in $\{1,\ldots,n\}$
%satisfying $(\pi(\sigma)\circ T_I)(\R^n)=0$.
%%
%%%%%%%%%%%%%%%%%%%%%%%%%%%%%%%%%%%%%%%%%%%
%%%%%%%%%%%%%%%%%%%%%%%%%%%%%%%%%%%%%%%%%%%
\item
The Jacobian of the mapping $\pi(\sigma)$
is equal to
%%%%%
$$
J_{\pi(\sigma)}(y)=
\epsilon\prod_{j=1}^n y_j^{\langle a^j(\sigma)\rangle -1},
$$
%%%%%
where $\epsilon$ is $1$ or $-1$.
\end{enumerate}
\end{lemma}

%%%%%%%%%%%%%%%%%%%%%%%%%%%%%%%%%%%%%%%%%
%%%%%%%%%%%%%%%%%%%%%%%%%%%%%%%%%%%%%%%%%
%%%%%%%%%%%%%%%%%%%%%%%%%%%%%%%%%%%%%%%%%
\section{Construction of fans from polyhedra}\label{sec:6}
%%%%%%%%%%%%%%%%%%%%%%%%%%%%%%%%%%%%%

Let $P\subset\R_+^n$ be an $n$-dimensional polyhedron 
satisfying $P+\R_+^n\subset P$ 
(see Lemma~\ref{lem:2.2} in Section~2).
In this section, we explain a method to construct 
fans from a given polyhedron $P$. 
Furthermore, 
we consider precise relationships between 
cones of this fan and faces of $P$.
We denote by $(\R^n)^{\vee}$ the dual space of $\R^n$ 
with respect to the standard inner product. 
%%%%%%%%%%%%%%%%%%%%%%%%%%%%%%%%%%%%%
\subsection{Fans associated with polyhedra}
%%%%%%%%%%%%%%%%%%%%%%%%%%%%%%%%%%%%%
For $a=(a_1,\ldots,a_n)\in(\R^n)^{\vee}$ with
$a_j\geq 0$, 
define
%%%%%
\begin{equation}\label{eqn:6.1}
l(a)=\min\left\{
\langle a,\alpha\rangle: \alpha\in P \right\}
\end{equation}
%%%%%
and 
$\gamma(a)=
\{\alpha\in P :\langle a,\alpha\rangle=l(a)\}
(=H(a,l(a)) \cap P)$.
We introduce an equivalence relation $\sim$ 
in $(\R^n)^{\vee}$ by $a\sim a'$ 
if and only if $\gamma(a)=\gamma(a')$. 
For any $k$-dimensional face $\gamma$ of $P$, 
there is an equivalence class $\gamma^{\vee}$ 
which is defined by 
%%%%%
\begin{equation}\label{eqn:6.2}
\begin{split}
\gamma^{\vee}:=&
\{a\in (\R^n)^{\vee}:
\gamma(a)=\gamma \mbox{ and $a_j\geq 0$ 
for $j=1,\ldots,n$}\}\\
(=& \{a\in (\R^n)^{\vee}: \gamma=H(a,l(a))\cap P 
\mbox{ and $a_j\geq 0$ for $j=1,\ldots,n$} \}.)
\end{split}
\end{equation}
%%%%%
Here, $P^{\vee}:=\{0\}$.
The closure of $\gamma^{\vee}$, 
denoted by $\overline{\gamma^{\vee}}$,
is expressed as
%%%%%
\begin{equation}\label{eqn:6.3}
\overline{\gamma^{\vee}}=
\{a\in (\R^n)^{\vee}: \gamma \subset H(a,l(a))\cap P 
\mbox{ and $a_j\geq 0$ for $f=1,\ldots,n$} \}.
\end{equation}
%%%%%
It is easy to see that 
$\overline{\gamma^{\vee}}$ 
is an $(n-k)$-dimensional strongly convex rational 
polyhedral cone in $(\R^n)^{\vee}$ and, moreover, that 
the collection of $\overline{\gamma^{\vee}}$ 
gives a fan $\Sigma_P$, which is called   
the {\it fan associated with polyhedron $P$}.
Note that 
$|\Sigma_P|=\R_+^n$.

Furthermore, let us consider 
$n$-dimensional polyhedra $P_1,\ldots, P_m \subset\R_+^n$
satisfying $P_j+\R_+^n\subset P_j$ for all $j$.
Let $\Sigma_{P_j}$ be the fan associated with $P_j$.
It is easy to see that 
the collection of $\sigma_1\cap\cdots\cap\sigma_m$
for all $\sigma_j\in\Sigma_{P_j}$
gives a fan,
which is called the {\it fan associated with the polyhedra
$P_1,\ldots, P_m$}.
We remark that any simplicial subdivision 
of this fan is also a simplicial subdivision 
of $\Sigma_{P_j}$ for each $j$.

%%%%%%
%%%%%%%%%%%%%%%%%%%%%%%%%%%%%%%%%%%%%%%%%
\subsection{Lemmas}\label{subsec:6.2}

Let us recall some lemmas given in \cite{kn13}.  
They will be useful for the analysis 
in Sections~\ref{sec:8}--\ref{sec:13}. 
These lemmas also play important roles in 
the proofs of Lemma~\ref{lem:7.1} and Theorem~\ref{thm:7.2}, 
which are concerned with 
toric resolutions of singularities 
in the class $\hat{{\mathcal E}}(U)$ 
(see \cite{kn13}).

The following symbols are used in this subsection.
%%%
\begin{itemize}
%\item
%$P\subset\R_+^n$ is a polyhedron satisfying $P+\R_+^n\subset P$; 
%%%
\item
$\Sigma_0$ is the fan associated with the polyhedron
$P$; 
%%%
\item
$\Sigma$ is a simplicial subdivision of $\Sigma_0$;
\item
$\Sigma^{(n)}$ consists of $n$-dimensional cones in $\Sigma$;
%%%
\item 
$a^1(\sigma),\ldots,a^n(\sigma)$ is the skeleton
of $\sigma\in\Sigma^{(n)}$, ordered once and for all;
%%%
\item
${\mathcal P}(\{1,\ldots,n\})$
is the set of all subsets in 
$\{1,\ldots,n\}$;
%%%
\item
${\mathcal F}[P]$ is the set of 
nonempty faces of $P$;
%\item
%When $I\in{\mathcal P}(\{1,\ldots,n\})$, 
%we write $J:=\{1,\ldots,n\}\setminus I$;
\item
$H(\cdot,\cdot)$, $l(\cdot)$ 
are as in (\ref{eqn:2.1}), (\ref{eqn:6.1}), 
respectively.
\end{itemize}

%Let $\sigma\in\Sigma^{(n)}$, 
%$\gamma\in {\mathcal F}[P]$ and
%$I\in{\mathcal P}(\{1,\ldots,n\})$.
%%%%%%%%%%%%%%%
It is easy to see that the two maps 
\begin{equation*}
\gamma:{\mathcal P}(\{1,\ldots,n\})\times
\Sigma^{(n)}\to {\mathcal F}[P];
\quad 
I:{\mathcal F}[P]\times 
\Sigma^{(n)}\to{\mathcal P}(\{1,\ldots,n\})
\end{equation*}
can be defined by 
\begin{eqnarray}
&&
\label{eqn:6.4}
\gamma(I,\sigma)
:=
\bigcap_{j\in I}
H(a^j(\sigma),l(a^j(\sigma)))\cap P, \\
&&
\label{eqn:6.5}
I(\gamma,\sigma)
:=
\{j:\gamma\subset 
H(a^j(\sigma),l(a^j(\sigma)))\}.
\end{eqnarray}
Here set $\gamma(\emptyset,\sigma):=P$. 
%It is easy to see that $\gamma(I,\sigma)\in{\mathcal F}[P]$ and 
Note that $I(P,\sigma)=\emptyset$.

%For $\gamma\in{\mathcal F}[P]$, 

%%%%%%%%%%%%%%%%%%%%%%%%%%%%%%%%%%%%%%%%%%%%%
\begin{lemma}\label{lem:6.1}
For $\sigma\in\Sigma^{(n)}$, 
$\gamma\in {\mathcal F}[P]$,
$I\in{\mathcal P}(\{1,\ldots,n\})$, 
we have the following. 
\begin{enumerate}
\item
$\gamma\subset\gamma(I(\gamma,\sigma),\sigma)$ and
$\dim(\gamma)\leq n-\# I(\gamma,\sigma)$.
%%%%%%%%%%%%%%%
\item 
$\gamma=\gamma(I,\sigma)
\Longrightarrow I\subset I(\gamma,\sigma)
\Longrightarrow \dim(\gamma)\leq n-\#  I$.
%%%%%%%%%%%%%%%%
\end{enumerate}
\end{lemma}
%%%%%%%%%%%%%%%%%%%%%%%%%%%%%%%%%%%%%%%%%%%%%
\begin{proof}
The assertions in (i) are directly seen from the definitions of 
$\gamma(I,\sigma)$ and $I(\gamma,\sigma)$.
The first implication in (ii) 
is shown as follows: 
$\gamma=\gamma(I,\sigma)
\Rightarrow
\gamma
=
\bigcap_{j\in I}
H(a^j(\sigma),l(a^j(\sigma)))\cap P
\Rightarrow
\gamma\subset H(a^j(\sigma),l(a^j(\sigma)))
\mbox{ for $j\in I$}
\Rightarrow
I\subset I(\gamma,\sigma)$.
From the inequality in (i), 
the second implication in (ii) is obvious.
\end{proof}

%%%%%%%%%%%%%%%%%%%%%%%%%%%%%%%%%%%%%%%%%%%%%
Next, consider the case 
when $\dim(\gamma)=n-\#  I(\gamma,\sigma)$. 
Define
\begin{equation}\label{eqn:6.6}
\Sigma^{(n)}(\gamma):=\{
\sigma\in\Sigma^{(n)}:\dim(\gamma)=n-\#  I(\gamma,\sigma)
\}.
\end{equation}
Note that $\Sigma^{(n)}(P)=\Sigma^{(n)}$. 

%%%%%%%%%%%%%%%%%%%%%%%%%%%%%%
\begin{lemma}\label{lem:6.2}
For $\sigma\in\Sigma^{(n)}$, 
$\gamma\in {\mathcal F}[P]$,
$I\in{\mathcal P}(\{1,\ldots,n\})$, 
we have the following. Here $\gamma^{\vee}$ is as in 
$(\ref{eqn:6.2})$.
\begin{enumerate}
\item 
$\#  I(\gamma,\sigma)=\dim(\gamma^{\vee}\cap\sigma).$
%%%
\item
$\Sigma^{(n)}(\gamma)
=\{
\sigma\in\Sigma^{(n)}:\dim(\gamma^{\vee} \cap\sigma)
=\dim (\gamma^{\vee})
\}\neq\emptyset.
$
%%%
\item 
If $\sigma\in\Sigma^{(n)}(\gamma)$, then
$\gamma=\gamma(I(\gamma,\sigma),\sigma)$.
\end{enumerate}
\end{lemma}
%%%%%%%%%%%%%%%%%%%%%%%%%%%%%

%%%%%%%%%%%%%%%%%%%%%%
\begin{proof}
(i)\quad
This equation follows from the following equivalences.
$$
j\in I(\gamma,\sigma)
\Leftrightarrow
\gamma\subset H(a^j(\sigma),l(a^j(\sigma)))
\Leftrightarrow
a^j(\sigma)\in \overline{\gamma^{\vee}}
\Leftrightarrow
a^j(\sigma)\in 
\overline{\gamma^{\vee}}\cap\sigma,
$$
where $\overline{\gamma^{\vee}}$ denotes
the closure of ${\gamma^{\vee}}$. 
Note that the second equivalence follows from 
(\ref{eqn:6.3}).

%%%
(ii)\quad
Putting the equation in (i) and 
$\dim(\gamma^{\vee})=n-\dim(\gamma)$
together, we see the equality of the sets.
Since the support of the fan $\Sigma$ is $\R_+^n$, 
there exists $\sigma$ such that 
$\dim(\gamma^{\vee} \cap\sigma)
=\dim (\gamma^{\vee})$, which implies 
$\Sigma^{(n)}(\gamma)\neq\emptyset$.  

%%%
(iii)\quad 
Lemma~\ref{lem:6.1} (i),(ii) and 
the assumption imply 
$\dim(\gamma)=\dim(\gamma(I(\gamma,\sigma),\sigma))$. 
In fact, 
$
\dim(\gamma)\leq\dim(\gamma(I(\gamma,\sigma),\sigma))
\leq n-\# I(\gamma,\sigma)=\dim(\gamma).
$
Since $\gamma\subset\gamma(I(\gamma,\sigma),\sigma)$
from Lemma~6.1 (i),  
the above dimensional equation yields 
$\gamma=\gamma(I(\gamma,\sigma),\sigma)$. 
\end{proof}
%%%%%%%%%%%%%%%%%%%%%%

\begin{remark}
%From (\ref{eqn:}), $\gamma$ can be considered as 
%the map $\gamma:{\mathcal P}(\{1,\ldots,n\})\times\Sigma^{(n)}
%\longrightarrow{\mathcal F}(P)$.
It follows from 
Lemma~\ref{lem:6.2} (ii) and (iii) 
that the map $\gamma$ is surjective. 
\end{remark}
%%%

The following lemma is concerned with the property of
the map $\pi(\sigma)$ in (\ref{eqn:5.1}) when 
the face $\gamma(I,\sigma)$ is compact. 
%%%%
\begin{lemma}\label{lem:6.4}
If $\gamma=\gamma(I,\sigma)$, then 
the following conditions are equivalent.
\begin{enumerate}
\item $\gamma$ is compact$;$
\item $\pi(\sigma)(T_I(\R^n))=0$.
\end{enumerate}
\end{lemma}
%%%%%%%%%%%
\begin{proof}
See Proposition~8.6 in \cite{kn13}.
\end{proof}

%%%%%%%%%%%%%%%%%%%%%%%%%%%%
%%%%%%%%%%%%%%%%%%%%%%%%%%%%%%%%%%%%%%%%%%%%%%%%%%%%%%%%%%%%%%%%%%%%%%%%%%%%%%%%%%%%%%%%%%%%%%%%%%%%%%%%%%%%%%%%%%%%%%%%%%%%%%%%%%%%%%

\section{
Resolution of singularities 
in the class $\hat{\mathcal{E}}(U)$
}\label{sec:7}

In this section, 
we recall a result in \cite{kn13} 
about resolution of singularities 
in the class $\hat{\mathcal{E}}(U)$. 
Furthermore, we try to get a simultaneous resolution of
singularities for several functions belonging 
to $\hat{\mathcal{E}}(U)$, 
which is crucial for our analysis.

Hereafter, the following symbols are often used:
\begin{equation}\label{eqn:7.1}
l_{f}(a):=\min\{\langle a,\a\rangle:\a\in\Gamma_+(f)\},
\end{equation}
where $f$ is a smooth function defined near the origin, 
and 
\begin{eqnarray}
&&
T_I(\R^n)=\{y\in\R^n: y_j=0 \mbox{ if $j\in I$}\} \,\,\,
\mbox{(as in (\ref{eqn:1.5}))},\nonumber\\
&&
\tilde{T}_I(\R^n):=\{y\in\R^n: y_j=0 \mbox{ if and only if $j\in I$}\}
\label{eqn:7.2} \\
&&\quad\quad\quad
(=\{y\in T_I(\R^n): y_j\neq 0 \mbox{ if $j\not\in I$}\}) 
\nonumber.
\end{eqnarray}

By using the lemmas in the previous section, 
we see that an $\hat{\mathcal{E}}$ function 
can be expressed as in the normal crossing form 
sufficiently near the origin. 
The proof is given by using the lemmas
in the previous section. 
The details are seen in \cite{kn13}.

%%%%%%%%%%%%%%%%%%%%%%%%%%%%
\begin{lemma}[\cite{kn13}]\label{lem:7.1}
Let $f$ satisfy the condition $(A)$ in Section~3 and 
belong to $\hat{\mathcal{E}}(U)$,
where $U$ is a neighborhood of the origin,
let $\Sigma$ be a simplicial subdivision of the fan $\Sigma_0$
associated with the Newton polyhedra $\Gamma_+(f)$
(see Section~6.1) and
let $\sigma$ be an $n$-dimensional cone in $\Sigma$ whose skeleton 
is $a^1(\sigma),\ldots,a^n(\sigma)\in\Z_+^n$.
Then for any $\sigma\in\Sigma^{(n)}$, 
there exists a smooth function $f_{\sigma}$ defined on 
the set $\pi(\sigma)^{-1}(U)$ such that 
$f_{\sigma}(0)\neq 0$ and 
\begin{equation}\label{eqn:7.3}
f(\pi(\sigma)(y))=
\left(\prod_{j=1}^n y_j^{l_f(a^j(\sigma))}\right)f_{\sigma}(y)
\quad\quad \mbox{for $y\in\pi(\sigma)^{-1}(U)$}.
\end{equation}
Furthermore, if $\gamma=\gamma(I,\sigma)$
$($see $(\ref{eqn:6.4})$$)$, then we have 
\begin{equation}\label{eqn:7.4}
f_{\gamma}(\pi(\sigma)(y))=
\left(\prod_{j=1}^n y_j^{l_f(a^j(\sigma))}\right)
f_{\sigma}(T_I(y))
\quad\quad \mbox{for $y\in\pi(\sigma)^{-1}(U)$.}
\end{equation}
\end{lemma}
%%%%%%%%%%%%%%%%%%%
Notice that $f_{\sigma}$ may still have a singularity 
at a point where $f_{\sigma}$ vanishes. 
But an additional nondegeneracy assumption induces
that these kinds of singularities do not appear. 
This fact is exactly stated as in Theorem~8.10 
in the paper \cite{kn13}.
This result will be generalized in the following theorem.
It shows the existence of some kind of  
simultaneous resolution of singularities 
for several functions 
belonging to the class $\hat{\mathcal{E}}(U)$ and
satisfying the nondegeneracy condition. 
%This theorem will be crucial in the analysis later, 
%because we need to resolve singularities for 
%several functions simultaneously. 

%%%%%%%%%%%%%%%%%%%%%%%%%%%%%%%%%%%%%%%%%%%%%%
\begin{theorem}\label{thm:7.2}
Let $f_1,\ldots, f_m$ satisfy the condition $(A)$ in Section~3
and belong to $\hat{\mathcal{E}}(U)$,
where $U$ is a neighborhood of the origin,
let $\Sigma$ be a simplicial subdivision of the fan $\Sigma_0$
associated with the Newton polyhedra 
$\Gamma_+(f_1),\ldots,\Gamma_+(f_m)$ and 
let $\sigma$ be an $n$-dimensional cone in $\Sigma$, 
whose skeleton is 
$a^1(\sigma),\ldots,a^n(\sigma)\in\Z_+^n$.
Then for any $k\in\{1,\ldots,m\}$ and any $\sigma\in\Sigma^{(n)}$,
there exists a smooth function 
$f_{k,\sigma}$
defined on the set $\pi(\sigma)^{-1}(U)$ satisfying that
$f_{k,\sigma}(0)\neq 0$ and
%%%%%
\begin{equation}\label{eqn:7.5}
f_k(\pi(\sigma)(y))
=\left(\prod_{j=1}^n
y_j^{l_{f_k}(a^j(\sigma))} 
\right)
f_{k,\sigma}(y),
\end{equation} 
%%%%%
for $y\in\pi(\sigma)^{-1}(U)$,
where $l_{f_k}(\cdot)$ is as in $(\ref{eqn:7.1})$.

Furthermore,
if $f_k$ is nondegenerate over $\R$ 
with respect to $\Gamma_+(f_k)$ and    
a set $I\subset\{1,\ldots,n\}$ satisfies 
$\pi(\sigma)(\tilde{T}_I(\R^n))=0$,  
then the set 
$\{y\in \tilde{T}_I(\R^n):f_{k,\sigma}(y)=0\}$
is nonsingular,
i.e.,
the gradient of the restriction of the function 
$f_{k,\sigma}$ to $\tilde{T}_I(\R^n)$ does not vanish at the 
points of the set 
$\{y\in \tilde{T}_I(\R^n): f_{k,\sigma}(y)=0\}$.
\end{theorem}
%%%%%%%%%%%%%%%%%%%%%%%%%%%%%%%%%%%%%%%%%%%%%%%%%
%%%%%%%%%%%%%%%%%%%%%%%%%%%%%%%%%%%%%%%%%%%%%%%%%
\begin{remark}\label{rem:7.3}
Let $b=(b_1,\ldots,b_n)$ be a point on $\tilde{T}_I(\R^n)$
satisfying $f_{k,\sigma}(b)=0$.
%We use the notation $J=\{1,\ldots,n\}\setminus I$.
By the implicit function theorem,
there exists a local coordinate around $b$ in which
$f_k\circ\pi(\sigma)$ can be expressed 
in a normal crossing form.
To be more specific, 
there exists a local diffeomorphism
$\phi$ defined around $b$ such that
$y=\phi(u)$ with $b=\phi(b)$ and 
%%%%%
\begin{equation}
(f_k\circ\pi(\sigma)\circ\phi)(u)=
(u_i-b_i)\left(\prod_{j\in I}u_j^{l_{f_k}(a^j(\sigma))}\right),
\end{equation}
%%%%%
where $y_j=u_j$ for $j\in I$ and $i\in\{1,\ldots,n\}\setminus I$.
\end{remark}
%%%%%%%%%%%%%%%%%%%%%%%%%%%%%%%%%%%%%%%%%%%%%%%%%%%
%%%%%%%%%%%%%%%%%%%%%%%%%%%%%%%%%%%%%%%%%%%%%%%%%%%
\begin{proof}[Proof of Theorem~\ref{thm:7.2}.] 
The case $m=1$ in the theorem has been shown in \cite{kn13}.
The case $m\geq 2$ can be similarly shown by using the fact that
$\Sigma$ is a simplicial subdivision of each fan 
$\Sigma_{f_k}$ for all $k$,
where $\Sigma_{f_k}$ is the fan associated with
$\Gamma_+(f_k)$.
\end{proof}
%%%%%%%%%%%%%%%%%%%%%%%%%%%%%%%%%%%%%%%%%%%%%%%%%%%%%%%%%%%%%%%%%%%%%%%%%%%%%%%%%%%%%%%%%%%%%%%%%%%%%%%%%%%%%
%%%%%%%%%%%%%%%%%%%%%%%%%%%%%%%%%%%%%%%%%%%%%%%%%%%%%%%%%%%%%%%%%%%%%%%%%%%%%%%%%%%%%%%%%%%%%%%%%%%%%%%%%%%%%%%%%%%%%%%%%%%%%%
%%%%%%%%%%%%%%%%%%%%%%%%%%%%%%%%%%%%%%%%%%%%%%%%%%%%%%%%%%%%%%%%%%%%%%%%%%%%%%
%%%%%%%%%%%%%%%%%%%%%%%%%%%%%%%%%%%%%%%%%%%%%%%%%%%%%%%%%%%%%%%%%%%%%%%%%%%%%%%%%%%%%%%%%%%%%%%%%%%%%%%%%%%%%%%%%%%%%%%%%%%%%%%%%%%%%%%%%%%
%%%%%%%%%%%%%%%%%%%%%%%%%%%%%%%%%%%%%%%%%%%%%%%%%%%%%%%%%%%%%%%%%%%%%%%%%%%%%%%%
\subsection{Remarks}
The proposition, below, is a weaker version of Theorem~7.2.
But,  
when one wants to get a simultaneous resolution 
of singularities with respect to several functions,
it might be more convenient to deal with their product.
Indeed, after computing the product, 
we can construct the resolution via only one function
from this proposition.

%%%%%%%%%%%%%%%%%%%%%%%%%%%%%%%%%%%%%%%%%%%%%%
\begin{proposition}\label{thm:7.}
Let $f_1,\ldots, f_m$ be as in Theorem~7.2 and 
let $F(x):=\prod_{k=1}^m f_k(x)$. 
Let $\Sigma$ be a simplicial subdivision of the fan $\Sigma_F$
associated with the Newton polyhedron 
$\Gamma_+(F)$ and 
let $\sigma$ be an $n$-dimensional cone in $\Sigma$, 
whose skeleton is 
$a^1(\sigma),\ldots,a^n(\sigma)\in\Z_+^n$.
Then for any $k\in\{1,\ldots,m\}$ and any $\sigma\in\Sigma^{(n)}$,
there exists a smooth function 
$f_{k,\sigma}$
defined on the set $\pi(\sigma)^{-1}(U)$ satisfying that
$f_{k,\sigma}(0)\neq 0$ and the equation $(\ref{eqn:7.5})$
holds
%%%%%
for $y\in\pi(\sigma)^{-1}(U)$.
%where $l_{f_k}(\cdot)$ is as in $(\ref{eqn:7.1})$.

Furthermore,
if $F$ is nondegenerate over $\R$ 
with respect to $\Gamma_+(F)$ and    
a set $I\subset\{1,\ldots,n\}$ satisfies 
$\pi(\sigma)(\tilde{T}_I(\R^n))=0$,  
then the set 
$\{y\in \tilde{T}_I(\R^n):f_{k,\sigma}(y)=0\}$
is nonsingular.
\end{proposition}
%%%%%%%%%%%%%%%%%%%%%%%%%%%%%%%%%%%%%%%%%%%%%%%%%

In order to prove Proposition~7.4, 
the following lemmas are important. 
We denote 
$\R_{>0}^n:=\{x=(x_1,\ldots,x_n)\in\R^n:
x_j>0 \mbox{ for $j=1,\ldots,n$} \}$.
%%%%%%%%%%%%%%%%%%%%%%%%
\begin{lemma}
Let $a=(a_1\ldots,a_n)\in \R_{>0}^n$.
Let $\Gamma(a):=H(a,l_F(a))\cap\Gamma_+(F)$ 
and $\gamma_k(a):=H(a,l_{f_k}(a))\cap\Gamma_+(f_k)$
for $k=1,\ldots,m$.
$($Note that the faces $\Gamma(a)$ and 
$\gamma_k(a)$ are compact for $k=1,\ldots,m$.$)$
Then we have
\begin{enumerate}
\item 
$F_{\Gamma(a)}(x)=
\prod_{k=1}^m
(f_k)_{\gamma_k(a)}(x)$;
\item
$l_F(a)=\sum_{k=1}^m l_{f_k}(a)$,
\end{enumerate}
where
$F_{\Gamma(a)}$ is the $\Gamma(a)$-part
of $F$ and  $(f_k)_{\gamma_k(a)}$ is the $\gamma_k(a)$-part
of $f_k$ for $k=1,\ldots,n$.
Furthermore, 
if every $f_k$ belongs to $\hat{\mathcal E}(U)$ for 
$k=1,\ldots,m$,
then
the above $(i)$,$(ii)$ hold for any $a\in\R_+^n$.
\end{lemma}
%%%%%%%%%%%%%%%%%%%%%%%%%%
\begin{proof}
The following equation holds:
$$
\frac{F(t^{a_1}x_1,\ldots,t^{a_n}x_n)}{t^{l_F(a)}}
=
\left(
\prod_{k=1}^m 
\frac{f_k(t^{a_1}x_1,\ldots,t^{a_n}x_n)}{t^{l_{f_k}(a)}}
\right)
\cdot
t^{\sum_{k=1}^ml_{f_k}(a)-l_F(a)},
$$
where $a\in\R_+^n$, $t>0$ is small and
$x$ is in a small neighborhood of the origin.
Consider the limits as $t\to 0$ of the left and right sides 
of the above equation.
Then the lemma can be easily shown.
\end{proof}
%%%%%%%%%%%%%%%%%%%%%%%%%%

%%%%%%%%%%%%%%%%%%%%%%%%
\begin{lemma}
Let $a=(a_1\ldots,a_n)\in\R_{>0}^n$. 
Let $f$ be a smooth function satisfying the condition $(A)$
in Section~3,
$\Sigma_f$ the fan associated to $\Gamma_+(f)$
and $\gamma(a)$ the compact face of $\Gamma_+(f)$
defined by $H(a,l_f(a))\cap\Gamma_+(f)$.  
Then the following two conditions are equivalent.
\begin{enumerate}
\item 
The $\gamma(a)$-part $f_{\gamma(a)}$ of $f$
equals to a monomial;
\item
There exists an $n$-dimensional cone in $\Sigma_f$ 
such that $a\in{\rm Int}(\sigma)$,
\end{enumerate}
where ${\rm Int}(\sigma)$ means the interior of
$\sigma$.
\end{lemma}
%%%%%%%%%%%%%%%%%%%%%%%%
\begin{proof}
Since $a\in\R_{>0}^n$, $\gamma(a)$ is a vertex
if and only if $f_{\gamma(a)}$ is a monomial.
Moreover, 
from the definition of cone in Section~6.1, 
the face $\gamma(a)$ is a vertex if and only if 
$a$ belongs to the interior of some $n$-dimensional cone
of $\Sigma_f$. 
\end{proof}
%%%%%%%%%%%%%%%%%%%%%%%%
\begin{lemma}
Let $\Sigma_0$ be the fan associated 
with the Newton polyhedra 
$\Gamma_+(f_1),\ldots,\Gamma_+(f_m)$ 
(as in Theorem~7.2). 
Then we have $\Sigma_F=\Sigma_0$.
\end{lemma}
%%%%%%%%%%%%%%%%%%%%%%%%
\begin{proof}
By an inductive argument, it suffices to show the lemma
in the case when $m=2$.
Let $\Sigma_F^{(n)}$ 
(resp. $\Sigma_{f_1}^{(n)}$,$\Sigma_{f_2}^{(n)}$,$\Sigma_0^{(n)}$) be
the set of $n$-dimensional cones 
contained in $\Sigma_F$ 
(resp. $\Sigma_{f_1}$,$\Sigma_{f_2}$,$\Sigma_0$).

Let us consider the following two claims:
\begin{enumerate}
\item[(a)]
For any $\sigma\in\Sigma_F^{(n)}$, 
there exist $\sigma_1\in\Sigma_{f_1}^{(n)}$ and
$\sigma_2\in\Sigma_{f_2}^{(n)}$ such that
$\sigma\subset\sigma_1\cap\sigma_2$.
\item[(b)]
For any $\sigma_1\in\Sigma_{f_1}^{(n)}$ and
$\sigma_2\in\Sigma_{f_2}^{(n)}$, 
there exists $\sigma\in\Sigma_F^{(n)}$ such that
${\rm Int}(\sigma_1\cap\sigma_2)\subset\sigma$.
\end{enumerate}
%%%%%%%%%%%%%%%

[Proof of the claim (a).]
Let $\sigma\in\Sigma_F^{(n)}$.
Since $F_{\Gamma(a)}(x)$ is a monomial
for any $a\in {\rm Int}(\sigma)$ from Lemma~~7.6 
and
$F_{\Gamma(a)}(x)=
(f_1)_{\gamma_1(a)}(x)\cdot(f_2)_{\gamma_2(a)}(x)$
from Lemma~7.5 (i),
$(f_1)_{\gamma_1(a)}(x)$ and 
$(f_2)_{\gamma_2(a)}(x)$ are monomials.
Applying Lemma~7.5 again, 
we see
$a\in{\rm Int}(\sigma_1)\cap{\rm Int}(\sigma_2)
={\rm Int}(\sigma_1\cap\sigma_2)$.

[Proof of the claim (b).]
We only consider the case when 
${\rm Int}(\sigma_1\cap\sigma_2)\neq\emptyset$.
If $a\in{\rm Int}(\sigma_1\cap\sigma_2)$, then
Lemmas~7.5 and 7.6 imply that
$F_{\Gamma(a)}(x)$ is a monomial. 
Thus there exists a cone $\sigma\in\Sigma_F^{(n)}$ 
such that $a\in{\rm Int}(\sigma)$.

By using the above two claims (a), (b), 
the lemma can be shown as follows.
For any $\sigma\in\Sigma_F^{(n)}$, 
there exist $\sigma_1\in\Sigma_{f_1}^{(n)}$, 
$\sigma_2\in\Sigma_{f_2}^{(n)}$ and 
$\tilde{\sigma}\in\Sigma_F^{(n)}$ such that
$\sigma\subset\sigma_1\cap\sigma_2\subset\tilde{\sigma}$.
Since $\sigma$ and $\tilde{\sigma}$ belong to
$\Sigma_F^{(n)}$ and satisfy $\sigma\subset\tilde{\sigma}$,
$\sigma$ and $\tilde{\sigma}$ are the same cones.
Thus $\sigma=\sigma_1\cap\sigma_2\in\Sigma_0^{(n)}$,
which implies $\Sigma_F^{(n)}\subset\Sigma_0^{(n)}$.
On the other hand, 
the inclusion $\Sigma_0^{(n)}\subset\Sigma_F^{(n)}$
can be  similarly shown. 
Therefore, we see that the two fans 
$\Sigma_0$ and $\Sigma_F$ are the same.

\end{proof}
%%%%%%%%%%%%%%%%%%%%%%%%%%%%%%%%%%%%%
\begin{lemma}
If $F$ is nondegenerate over $\R$ 
with respect to its Newton polyhedron, 
then so is every $f_k$ for $k=1,\ldots,m$.  
\end{lemma}
%%%%%%%%%%%%%%%%%%%%%%%%%%%%%%%%%%%%%
\begin{proof}
By an inductive argument, it also suffices to show the lemma
in the case when $m=2$.

Suppose that $f_1$ is not nondegenerate.
Then there exists a compact face $\gamma$ of 
$\Gamma_+(f_1)$ 
and a point $x_*\in(\R\setminus\{0\})^n$ 
such that 
$\frac{\partial(f_1)_{\gamma}}{\partial x_j}(x_*)=0$
for $j=1,\ldots,n$.
Let $a\in\R_{>0}^n$ satisfy that
$(a,l_{f_1}(a))$ is a valid pair defining $\gamma$.
Moreover, from the Euler identity:
$\sum_{j=1}^n a_j x_j \frac{\partial(f_1)_{\gamma}}{\partial x_j}(x)
=(f_1)_{\gamma}(x)$,
we have $(f_{1})_{\gamma}(x_*)=0$.
Since
$F_{\Gamma(a)}(x)=
(f_1)_{\gamma}(x)\cdot(f_2)_{\gamma_2(a)}(x)$ from Lemma~7.5 (i),
we have
$$
\frac{\partial F_{\Gamma(a)}}{\partial x_j}(x_*)
=
\frac{\partial(f_1)_{\gamma}}{\partial x_j}(x_*)\cdot(f_2)_{\gamma_2(a)}(x_*)+
(f_1)_{\gamma}(x_*)\cdot
\frac{\partial(f_2)_{\gamma_2(a)}}{\partial x_j}(x_*)=0,
$$
for $j=1,\ldots, m$.
This shows that $F$ is not nondegenerate.
\end{proof}
%%%%%%%%%%%%%%%%%%%%%%%%%%%%%%%%%%%%%%

%%%%%%%%%%%%%%%%%%%%%%%%%%%%%%%%%%%%%%%%%%%%%%%%%
\begin{remark}
The converse of the assertion in Lemma~7.8 does not hold in general. 
Indeed, consider the case $f_1(x_1,x_2)=f_2(x_1,x_2)=x_1^2-x_2^2$. 
Therefore, the assumptions in Proposition~7.4 
are stronger than those in Theorem~7.2. 
\end{remark}
%%%%%%%%%%%%%%%%%%%%%%%%%%%%%%%%%%%%%%%%%%%%%%%%%
%\begin{remark}
%If $f$ is nondegenerate
%\end{remark}
%%%%%%%%%%%%%%%%%%%%%%%%%%%%%%%%%%%%%%%%%%%%%%%%%

\begin{proof}[Proof of Proposition~7.4]
Applying Lemmas 7.7 and 7.8 to Theorem~7.2, 
we can easily show the proposition.
\end{proof}

%%%%%%%%%%%%%%%%%%%%%%%%%%%%%%%%%%%%%%%%%%%%%%%%%%%%%%%%%%%%%%
\section{Poles of local zeta type functions}\label{sec:8}

The subsequent six sections are devoted to the 
investigation of poles of weighted local zeta functions
and similar functions. 
Throughout these sections, we usually assume the following
without any mentioning: 
Let $U$ be an open neighborhood of the origin in $\R^n$ and 
the functions $f$, $g$, $\varphi$ satisfy 
the conditions (A), (B), (C) 
in the beginning of Section~\ref{sec:3}.
But we sometimes consider the case when 
$g$ is a nonzero flat function in Section~10, 
which does not satisfy the condition (B).

By means of the Mellin transform, 
there exists a clear relationship 
between the oscillatory integral (\ref{eqn:1.1})
and some functions  
similar to the local zeta function (\ref{eqn:1.3})
(see Section~\ref{subsec:14.1}).
As is well known in 
\cite{jea70},\cite{mal74},\cite{var76},\cite{igu78},\cite{agv88}, etc.,
through this relationship, 
the investigation of the behavior of oscillatory integrals 
is reduced to 
that of the properties of poles of these functions.   
The results in Section~4 will be proved 
in Section~\ref{subsec:14.3}
by using the results in these six sections.

%%%%%%%%%%%%%%%%%
\subsection{Functions similar to the local zeta functions}\label{subsec:8.1}
We investigate the properties of 
poles of the functions:
%%%%%
\begin{equation}\label{eqn:8.1}
Z_{+}(s;\varphi)=\int_{\R^n} f(x)_{+}^s g(x)\varphi(x)dx, \quad 
Z_{-}(s;\varphi)=\int_{\R^n} f(x)_{-}^s g(x)\varphi(x)dx, 
\end{equation}
%%%%%
where 
$$
f(x)_+=\max\{f(x),0\} \mbox{ \,\,and \,\,} 
f(x)_-=\max\{-f(x),0\}
$$ 
and the ({\it weighted}) 
{\it local zeta function}:
%%%%%
\begin{equation}\label{eqn:8.2}
Z(s;\varphi)=\int_{\R^n} |f(x)|^s g(x)\varphi(x)dx.
\end{equation}
%%%%%
We call $g$ the {\it weight}.
From the properties of $Z_{+}(s;\varphi)$ and $Z_{-}(s;\varphi)$, 
we can easily obtain analogous properties 
of $Z(s;\varphi)$  
by using the relationship: 
$Z(s;\varphi)=Z_{+}(s;\varphi)+Z_{-}(s;\varphi)$.

It is easy to see that $Z_{\pm}(s;\varphi)$ and $Z(s;\varphi)$
are holomorphic functions on the region 
${\rm Re} (s)>0$. 
Moreover, 
if there exists a ``simultaneous resolution of singularities''
with respect to the functions $f$ and $g$, 
then 
$Z_{\pm}(s;\varphi)$ and $Z(s;\varphi)$ 
can be meromorphically continued
to the whole complex plane and, moreover,  
precise properties of their poles 
can be seen.
Following the analysis of Varchenko in \cite{var76} 
(see also \cite{agv88}),
let us explain this process. 
%*****************************************

%%%%%%%%%%%%%%%%%%%%%%%%%%%%%%%%%%%%%%%%%%%%%%%%%%%%%%%%%%%%%%%
\subsection{How to use resolution of singularities}
\label{subsec:8.2}
%Analytic continuation of 
%$Z_{\pm}(s;\varphi)$ and $Z(s;\varphi)$}

For the moment, 
we assume the existence of a proper real analytic mapping 
$\pi$ from some $n$-dimensional real analytic manifold $Y$ 
to $\R^n$ such that 
at each point of the set $\pi^{-1}(0)$, 
there exist smooth local coordinates 
$y=(y_1,\ldots,y_n)$ satisfying the following properties:
\begin{enumerate}
\item
There exist nonnegative integers $l_j, \tilde{l}_j$
and smooth functions $\tilde{f}, \tilde{g}$ defined 
near the origin
with $\tilde{f}(0),\tilde{g}(0)\neq 0$ such that 
\begin{eqnarray}\label{eqn:8.3}
f(\pi(y))=\left(\prod_{j=1}^n y_j^{l_j}\right)\tilde{f}(y), \quad
g(\pi(y))=\left(\prod_{j=1}^n y_j^{\tilde{l}_j}\right)\tilde{g}(y). 
\end{eqnarray}
\item
The Jacobian of the mapping $\pi$ has the form
\begin{equation}\label{eqn:8.4}
J_{\pi}(y)=\pm \prod_{j=1}^n y_j^{\tilde{m}_j-1},
\end{equation}
where $\tilde{m}_j$ are positive integers. 
\item
On a neighborhood of the origin 
in $\R^n$, 
$\pi$ is a diffeomorphism outside a proper analytic subset in 
$\R^n$.
\end{enumerate}
The above proper map
$\pi:Y\to\R^n$
is called a 
{\it simultaneous resolution of singularities} 
with respect to $f$ and $g$.

By means of the above map $\pi:Y\to\R^n$, 
we have 
\begin{equation}\label{eqn:8.5}
\begin{split}
&Z_{\pm}(s;\varphi)
=\int_{\R^n} f(x)_{\pm}^s g(x)\varphi(x)dx\\
&\quad\quad
=\int_{Y} f(\pi(y))_{\pm}^s g(\pi(y))\varphi(\pi(y))J_{\pi}(y)dy.
\end{split}
\end{equation}
Furthermore, 
by using an appropriate smooth finite partition of unity,
substituting (\ref{eqn:8.3}),(\ref{eqn:8.4}) and 
applying the orthant decomposition: 
\begin{equation}\label{eqn:8.6}
\int_{\R^n} F(x)dx=\sum_{\theta\in\{-1,1\}^n}
\int_{\R_+^n}F(\theta_1 x_1,\ldots,\theta_n x_n)dx,
\end{equation}
%%%%%
where $F$ is a continuous function on $\R^n$,
the above (\ref{eqn:8.5}) implies that 
the functions $Z_{\pm}(s;\varphi)$ can be written 
in a certain sum of functions of the following form:
\begin{equation}\label{eqn:8.7}
\int_{\R_+^n}
\left(\prod_{j=1}^n y_j^{l_j s+\tilde{l}_j+\tilde{m}_j -1}\right)
\psi(y;s) dy,
\end{equation}
where $\psi(\cdot;s)$ is
a $C_0^{\infty}$ function of $y$ on $\R^n$ for any $s\in\C$ 
and 
$\psi(y;\cdot)$ is an entire function on $\C$
for any $y\in\R^n$.
As we will show in the next section, 
the above integral is easy to be treated and 
a direct computation gives 
many important properties of its poles. 
%Though there remains delicate problem about cancellation etc, 
%we can understand many important properties of poles 
%of the functions $Z_{\pm}(s;\varphi)$. 
From the above arguments, 
it becomes an essentially important 
problem to find an appropriate
map $\pi$ in our investigation of 
the functions $Z_{\pm}(s;\varphi)$ and 
the oscillatory integral $I(t;\varphi)$. 
%Roughly speaking, in order to get a strong result 
%about oscillatory integrals, we have to find some 
%suitable resolution of singularities.

Let us discuss the existence 
of the simultaneous resolution of singularities with respect ot $f$ and $g$
and its application in earlier studies of oscillatory integrals.
First, we consider the unweighted case 
(i.e., $g\equiv 1$). 
If $f$ is real analytic near the origin, 
then a famous Hironaka's theorem \cite{hir64}
implies the existence of a desired map $\pi:Y \to \R^n$ 
(see also \cite{ati70}).
Moreover, an easy computation in the next section
shows that poles of the integral (\ref{eqn:8.7})
is contained in arithmetic progressions.
As a result, 
one can see that the distribution of
poles of $Z_{\pm}(s;\varphi)$ are contained in a union of a 
finite number of arithmetic progressions 
and, moreover, that 
an asymptotic expansion of 
the oscillatory integral $I(t;\varphi)$ in (\ref{eqn:1.1}) 
can be expressed as in (\ref{eqn:1.2}).  
These analyses are seen in the earlier works
\cite{jea70},\cite{mal74}, etc. 
This kind of application of 
Hironaka's resolution theorem to the analysis
can be traced back to the works of 
Berenstein and Gel'fand \cite{bg69}
and Atiyah \cite{ati70}.
By the way, these results are still very abstract. 
In order to determine or exactly estimate the oscillation index
and its multiplicity, 
one has to get more quantitative information 
of the leading pole of the integrals in (\ref{eqn:8.1}).  
Varchenko \cite{var76} constructs resolution of singularities 
by using the theory of toric varieties 
and describes the values 
of $l_j,\tilde{m}_j$ in (\ref{eqn:8.3}),(\ref{eqn:8.4}) 
in terms of the geometrical information 
of the Newton polyhedron of $f$ under the assumption 
(D) in Section~3.
As a result, he gave precise results
in Theorem~\ref{thm:3.1} in Section~3.
Furthermore, 
the authors \cite{kn13} obtain an analogous result
when $f$ satisfies the condition:
%%%%%
\begin{itemize}
\item[(E)]
$f$ belongs to the class $\hat{\mathcal{E}}(U)$
and is nondegenerate over $\R$ with respect to its
Newton polyhedron.
\end{itemize}
%%%%%

Next, 
let us consider the weighted case (i.e., $g\not\equiv 1$).
When the weight is a monomial, 
the above analysis of Varchenko 
can be directly generalized and 
the values of $\tilde{l}_j$ are 
also computed by using the resolution for the function $f$ only
(see \cite{vas79} and Section~\ref{sec:10} in this paper).  
More general weighted case essentially requires  
a simultaneous resolution of singularities 
with respect to $f$ and $g$.
Remember that a map
$\pi:Y_{\Sigma}\to\R^n$ 
in Section~\ref{sec:7} has been constructed
in the case of $\hat{\mathcal{E}}(U)$, 
which is similar to the
simultaneous resolution of singularities 
with respect to $f$ and $g$,
by using the geometrical information 
of the Newton polyhedra of $f$ and $g$
under some conditions. 
Though this map $\pi$ does not satisfy all the conditions 
for resolution of singularities with respect to $f$ and $g$, 
we can obtain sufficiently precise results 
about the poles of $Z_{\pm}(s;\varphi)$ by means of this map 
in Sections~10, 11 and 12.
%Pramanik and Yang 
%As a related problem,
%it should be pointed out that there has also been
%many studies regarding the existence of
%critical integrability exponents
%%%%%
%$$
%\sup \left\{\epsilon >0 ; 
%\int_{V}|f(x)|^{-\epsilon}<\infty \right\},
%$$
%%%%%
%where $V$ is a bounded open set containing the origin
%(see \cite{}, \cite{}, \cite{}, etc.).

%%%%%%%%%%%%%%%%%%%%%%%%%%%%%%%%%%%%%%%%%%%%%%%%%%%%%%%%%%%%%%%%%%%%%%%%%%%%%%%%%%%%%%%%%%%%%%%%%%%%%%%%%%%%%%

%%%%%%%%%%%%%%%%%%%%%%%%%%%%%%%%%%%%%%%%%%%%%%%%%%%%%%%%%%%%%%%%%%%%%%%%

\section{Computation on poles of the elementary integrals}\label{sec:9}

In this section, 
we investigate poles of 
an elementary integral of the form:
\begin{equation}\label{eqn:9.1}
L(s)=
\int_{\R_+^n}
\left(\prod_{j=1}^n y_j^{l_j s+m_j -1}\right)
\psi(y;s) dy,
\end{equation}
where $l_j\in\Z_+$, $m_j\in\N$ and $\psi$
%(\cdot;\cdot)$ 
is as in (\ref{eqn:8.7}).
(Put $m_j=\tilde{l}_j+\tilde{m}_j$, then
(\ref{eqn:8.7}) is given.)

%%%%%%%%%%%%%%%%%%%%%%%%%%%%%%%%%%%%%%%%
\subsection{One-dimensional model}\label{subsec:9.1}

By repeated integrals, 
the analysis of poles of the integral 
$L(s)$ is essentially reduced to 
that of the following simple one-dimensional 
integral:
\begin{equation}\label{eqn:9.2}
L_1(s)=\int_0^{\infty}y^s\psi(y)dy, 
\end{equation}
where $\psi\in C_0^{\infty}(\R)$. 
Observing the convergence of the integral,
we see that $L_1$ is holomorphic on the region
${\rm Re}(s)>-1$. 
Furthermore, 
we obtain the following.
%%%%%%%%%%%%%%%%%%%%%%%%%%%%%%%%%
\begin{lemma}\label{lem:9.1}
The function $L_1$ can be meromorphically extended 
to the whole complex plane  and its
poles are contained in $(-\N)$. 
Moreover, we have
\begin{equation}\label{eqn:9.3}
\lim_{s\to-k}(s+k)L_1(s)
=\frac{1}{(k-1)!}\psi^{(k-1)}(0) \quad \mbox{for $k\in\N$}.
\end{equation}
In particular, 
\begin{equation}\label{eqn:9.4}
\lim_{s\to-1}(s+1)L_1(s)
=\psi(0).
\end{equation}
\end{lemma}
%%%%%%%%%%%%%%%%%%%%%%%%%%%%%%%%%
\begin{proof}
The integrations by parts imply 
\begin{equation}\label{eqn:9.5}
\begin{split}
&\int_0^{\infty}y^s\psi(y)dy
=\frac{-1}{s+1}\int_0^{\infty}y^{s+1}\psi'(y)dy=
\cdots\\
&\quad\quad\quad
\cdots=\frac{(-1)^k}{(s+1)\cdots(s+k)}
\int_0^{\infty}y^{s+k}\psi^{(k)}(y)dy,
\end{split}
\end{equation}
for any nonnegative integer $k$.
The assertions in the lemma easily 
follow from the above equations. 
\end{proof}

%%%%%%%%%%%%%%%%%%%%%%%%%%%%%%%%%%%%%%%%%%%%%%%%%%%%%
\subsection{Positions and orders of poles}\label{subsec:9.2}

Observing the one-dimensional model, we can see 
important properties of poles of the function 
$L(s)$ in (\ref{eqn:9.1}).
Let 
$$
B:=\{j:l_j\neq 0\} \subset \{1,\ldots,n\}.
$$
It is easy to see that
$L(s)$ is holomorphic on the region 
${\rm Re}(s) > \max\{-(m_j-1)/l_j:j\in B\}$.
%; j=1,\ldots,n \}
Moreover, the following can be seen.  
%%%%%%%%%%%%%%%%%%%%%%%%%%%%%%%%%%%%%%%%%%%
\begin{proposition}[\cite{gs64},\cite{agv88}]\label{pro:9.2}
The function $L(s)$ can be analytically continued
to the whole complex plane as a meromorphic function
and its poles belong to the set 
%$n$ arithmetic progressions
%%%%%
\begin{equation}\label{eqn:9.6}
\left\{-\frac{m_j+\nu_j}{l_j}:
\nu_j\in \Z_+, j\in B \right\}.
\end{equation}
%%%%%
Moreover, suppose that $p$ belongs to the above set and let 
$$
A(p)
=
\left\{
j\in B:\mbox{there exists $\nu_j\in\Z_+$ such that } 
-\frac{m_j+\nu_j}{l_j}=p
\right\}.
$$
Then, if 
$p$ is a pole of $L(s)$, then the order of the pole of
$L(s)$ at $p$ is at most $\# A(p)$.
\end{proposition}
%%%%%%%%%%%%%%%%%%%%%%%%%%%%%%%%%%%%%%%%%%%%
\begin{proof}
Observing the proof of Lemma~\ref{lem:9.1}, 
we can easily show the above from 
the integrations by parts
with respect to $y_j$ for all $j\in B$. 
\end{proof}
%%%%%%%%%%%%%%%%%%%%%%%%%%%%%%%%%%%%%%%%%%%%
Here, we return to the investigation of poles of 
$Z_{\pm}(s;\varphi)$, $Z(s;\varphi)$ in Section~\ref{sec:8}.
When there exists a ``simultaneous resolution of singularities''
with respect to $f$ and $g$, 
$Z(s;\varphi), Z_{\pm}(s;\varphi)$ can be 
written in a certain sum of the integrals of the form 
(\ref{eqn:9.1}). 
Applying Proposition~\ref{pro:9.2} to these integrals,  
we see that their poles are contained in 
some union of sets of the form (\ref{eqn:9.6}). 

%%%%%%%%%%%%%%%%%%%%%%%%%%%%%%%%%%%%%%%%%%%
\subsection{Coefficients}\label{subsec:9.3}

Let us compute the coefficients
of terms in the Laurent expansion of $L(s)$
at the above candidate poles.
This computation will be useful for the investigations 
in the latter sections.

Let
$s_*:=\max\left\{-m_j/l_j: j\in B\right\}$
and $A_*:=A(s_*)\subset B$.
%$A=\left\{j\in\{1,\ldots,n\}; 
%\b=-m_j/l_j, l_j\neq 0\right\}$.
%%%%%
%\begin{equation*}
%\begin{split}
%\b=\max\left\{-\frac{m_j}{l_j}; l_j\neq 0\right\},\\
%A=\left\{j; \b=-\frac{m_j}{l_j}, l_j\neq 0\right\}.
%\end{split}
%\end{equation*}
%%%%%
Note that $s_*$ is the leading pole of $L(s)$ and 
$\# A_*$ is its order, 
if the following coefficient does not vanish. 
%%%%%%%%%%%%%%%%%%%%%%%%%%%%%%%%%%%%%%%%%%%%%%%
\begin{proposition}\label{pro:9.3}
The coefficient of $(s-s_*)^{-\# A_*}$
in the Laurent expansion of $L(s)$ at $s=s_*$
is equal to
%%%%%
\begin{equation*}
\begin{cases}
&\displaystyle
\frac{\psi(0;s_*)}{\prod_{j=1}^n l_j}, \,\,\,\,\,\,
\mbox{ if $A_*=B=\{1,\ldots,n\}$,}
\\
&\displaystyle 
\frac{1}{\prod_{j\in A_*}l_j}
\int_{\R_+^{n-\# A_*}}
\left(\prod_{j\notin A_*}y_j^{l_j s_*+m_j-1}\right)
\psi(T_{A_*}(y);s_*)
\prod_{j\notin A_*}dy_j, \,\,\,
\mbox{ otherwise},
\end{cases}
\end{equation*}
%%%%%
where $T_{A_*}(\cdot)$ is as in $(\ref{eqn:1.4})$.
\end{proposition}
%%%%%%%%%%%%%%%%%%%%%%%%%%%%%%%%%%%%%%%%%%%%%%%
%%%%%%%%%%%%%%%%%%%%%%%%%%%%%%%%%%%%%%%%%%%%%%%
\begin{proof}
Recalling the proof of Lemma~\ref{lem:9.1}, 
we see that the integration by parts
with respect to each $y_j$ for $j\in A_*$
and the computation of 
$\lim_{s\to s_*} (s-s_*)^{\# A_*} L(s)$
give the lemma.
\end{proof}
%%%%%%%%%%%%%%%%%%%%%%%%%%%%%%%%%%%%%%%%%%%%%%%
In the analysis of $Z_{\pm}(s;\varphi)$ in the latter sections, 
we must consider the case when 
the coefficients of some poles have particular properties. 
(Actually, this case is induced by the zero set of
the function $f_{\sigma}$ in Lemma~\ref{lem:7.1}.)
This property will be understood through the following 
functions:
%%%%%
\begin{equation*}
L_{\pm}(s)=
\int_{\R_+^n}
\left(
y_n^s
\prod_{j\in D}
y_j^
{l_j s+m_j-1}
\right)
\psi(y_1,\ldots,y_{n-1},\pm y_n)dy,
\end{equation*}
%%%%%
where $D$ is a subset of $B\setminus \{n\}$.

%%%%%%%%%%%%%%%%%%%%%%%%%%%%%%%%%%%%%%%%%%%%%%
\begin{lemma}\label{lem:9.4}
For $\lambda\in -\N$,
let $A_{\lambda}$ be a subset in
$D$ defined by
$
A_{\lambda}
=\{j:
l_j\lambda+m_j-1\in -\N\}.
$
Then the functions $L_{\pm}(s)$ have at 
$s=\lambda$ poles of order not higher 
than $\#  A_{\lambda}+1$.
Moreover, the following holds$:$
Let $a_{\lambda}^{\pm}$ be the coefficients of 
$(s-\lambda)^{-\#  A_{\lambda}-1}$ in the Laurent expansions of 
$L_{\pm}(s)$ at $s=\lambda$, 
respectively. 
Then the equation:
 $a_{\lambda}^+=(-1)^{\lambda-1}a_{\lambda}^-$
for $\lambda\in-\N$
holds.
\end{lemma}
%%%%%%%%%%%%%%%%%%%%%%%%%%%%%%%%%%%%%%%%%%%%%%

%%%%%%%%%%%%%%%%%%%%%%%%%%%%%%%%%%%%%%%%%%%%%%
\begin{proof}
Let $k_j$ be positive integers defined by 
$k_j=\min\{k\in \N: l_j\lambda+m_j+k>0\}$
for $j\in D$.
We remark that 
$l_j\lambda+m_j+k_j=1$ if and only if $j\in A_{\lambda}$.
It easily follows from Proposition~\ref{pro:9.2}
that the functions $L_{\pm}(s)$ can be 
analytically extended to $\C$ as meromorphic functions
and they have at $s=\lambda$ poles of order not
higher than $\# A_{\lambda}+1$.

By using the integration by parts,
a direct computation gives 
$a_{\lambda}^{\pm}=(\pm 1)^{\lambda-1}C_{\lambda}$,
where $C_{\lambda}$ is given as follows.
If $D=A_{\lambda}=\{1,\ldots,n-1\}$, then
%%%%%
\begin{equation*}
C_{\lambda}=\frac{1}{(-\lambda-1)!}
\left(\prod_{j=1}^{n-1}
\frac{1}{l_j(-l_j \lambda-m_j)!}\right)
(\partial^{\alpha}\psi)(0).
\end{equation*}
Otherwise, 
\begin{equation*}
\begin{split}
&C_{\lambda}=\frac{1}{(-\lambda-1)!}
\left(\prod_{j\in A_{\lambda}}
\frac{1}{l_j(-l_j \lambda-m_j)!}\right)
\left(\prod_{j\in D\setminus A_{\lambda}}
\frac{1}{\prod_{\nu=1}^{k_j}
(-l_j \lambda-m_j+1-\nu)}\right)\\
&\quad\quad\times
%\begin{cases}
%(\partial^{\alpha}\psi)(0),
%\quad\quad \mbox{if $D=A_{\lambda}=\{1,\ldots,n-1\}$,}\\
\displaystyle 
\int_{\R_+^{n-\# A_{\lambda}-1}}
\left(
\prod_{j\in D\setminus A_{\lambda}}
y_j^{l_j\lambda+m_j+k_j-1}
\right)
\left(
\partial^{\alpha}\psi
\right)
\left(
T_{A_{\lambda}\cup\{n\}}(y)
\right)
\prod_{j\not\in A_{\lambda}\cup\{n\}}
dy_j. \quad 
%\mbox{otherwise},
%\end{cases}
\end{split}
\end{equation*}
%%%%%
Here $\a=(\a_1,\ldots,\a_n)$
satisfies that
$\a_j=k_j-1$ for $j\in A_{\lambda}$,
$\a_j=k_j$ for $j\in D\setminus A_{\lambda},$
$\a_n=-\lambda-1$
and $\a_j=0$ otherwise.
From the above equations,
we see that 
$a_{\lambda}^+=(-1)^{\lambda-1}a_{\lambda}^-$.
\end{proof}
%%%%%%%%%%%%%%%%%%%%%%%%%%%%%%%%%%%%%%%%%%%%%%

%%%%%%%%%%%%%%%%%%%%%%%%%%%%%%%%%%%%%%%%%%%%%%%%%%%%%%%%%%%%%%%%%%%%%%%%%%%%%%%%%%%%
\section{The case of monomial type weight}\label{sec:10}

The purpose of this section is to investigate
the poles of 
$Z_{\pm}(s;\varphi)$ 
in (\ref{eqn:8.1}) and 
$Z(s;\varphi)$
in (\ref{eqn:8.2})
in the case when
the weight $g$ is expressed as 
a monomial $x^p$ multiplied by a smooth function $\psi$
defined on a neighborhood $U$ of the origin in $\R^n$.
In this case,  
$g$ does not belong to $\hat{\mathcal E}(U)$
if and only if $\psi$ is a nonzero flat function. 

We explain reasons why we treat this special 
weight case in detail. 
First of all, 
this weight has a special form
such that the analysis of Varchenko can be directly generalized. 
Moreover, many parts of the investigation of 
more general weight case in Section~11
coincide with those of this simpler case. 
Therefore, the analysis in this section 
not only reduces the complexity of 
that in more general weight case in Section~11
but also clarifies essentially important parts.
Second of all, 
in this case, 
some particular results can
be obtained,
which cannot be covered by the results
in Section~11. 
To be more specific, 
we consider the case when 
the weight is flat in this section 
(see Remark~\ref{rem:10.3}, below).
This result affects the phenomenon 
in the behavior of oscillatory integrals
in Remark~\ref{rem:4.3}.

In this section,
we use the following notation, 
which were introduced in Sections~ 5 and 6.
%%%%%
\begin{itemize}
\item
$\Sigma_0$ is the fan associated with $\Gamma_+(f)$;
\item
$\Sigma$ is a simplicial subdivision of $\Sigma_0$;
\item
$(Y_{\Sigma},\pi)$ is the real resolution 
associated with $\Sigma$;
\item
$a^1(\sigma),\ldots,a^n(\sigma)$ is the skeleton of
$\sigma\in\Sigma^{(n)}$, ordered once and for all;
\item
$J_{\pi}(y)$ is the Jacobian of the mapping of $\pi$.
\end{itemize}
%%%%%

%%%%%%%%%%%%%%%%%%%%%%%%%%%%%%%%%%%%%%%%%%%%%%%%%%%%%%%%%%%%%%%%%%%%%%%

\subsection{Candidate poles}\label{sec:10.1}

First, let us state our results relating to 
the positions and the orders of candidate poles of
$Z_{\pm}(s;\varphi)$ and $Z(s;\varphi)$.
%%%%%%%%%%%%%%%%%%%%%%%%%%%%%%%%%%%%%%%%%%%
\begin{theorem}\label{thm:10.1}
Suppose that {\rm (i)} $f$ satisfies the condition $(E)$,
{\rm (ii)} $g(x)=x^p\psi(x)$, where $p\in \Z_+^n$ 
and $\psi$ is a smooth function defined on $U$.
If the support of $\varphi$ is
contained in a sufficiently small neighborhood 
of the origin, 
then 
the functions  
$Z_{\pm}(s;\varphi)$ and $Z(s;\varphi)$ 
can be analytically continued
as meromorphic functions on the whole complex plane.
More precisely, we have the following.
%%%%%%%%%%%%%%%%
\begin{itemize}
\item[(a)]
The poles of the functions $Z_{\pm}(s;\varphi)$ and $Z(s;\varphi)$
are contained in the set 
%%%%%
\begin{equation}\label{eqn:10.1}
\left\{
-\frac{\langle a,p+\1\rangle +\nu}{l_f(a)}
:\,\, \nu\in\Z_+,\,\, a\in\tilde{\Sigma}^{(1)}
\right\}
\cup (-\N),
\end{equation}
%%%%%
where $l_f(a)$ is as in $(\ref{eqn:7.1})$ and 
$\tilde{\Sigma}^{(1)}=\{a\in\Sigma^{(1)}:l_f(a)\neq 0\};$
\item[(b)]
The largest element of the first set in
$(\ref{eqn:10.1})$
is $-1/d(f,x^p);$
\item[(c)]
When
$Z_{\pm}(s;\varphi)$ and $Z(s;\varphi)$
have poles at $s=-1/d(f,x^p)$,
their orders are at most
%%%%%
\begin{eqnarray}\label{eqn:10.2}
%%%%%
\begin{cases}
m(f,x^p)& \quad \mbox{if $1/d(f,x^p)$ 
is not an integer}, \\
\min\{m(f,x^p)+1, n\}&
\quad \mbox{otherwise}.
\end{cases}
%%%%%
\end{eqnarray}
%%%%%
\end{itemize}
%%%%%%%%%%%%%%%%%
\end{theorem}
%%%%%%%%%%%%%%%%%%%%%%%%%%%%%%%%%%%%%%%%%%%%
\begin{remark}\label{rem:10.2}
When $\psi(0)\neq 0$, 
$\Gamma_+(g)=\{p\}+\R_+^n$, 
which implies $d(f,x^p)=d(f,g)$.
Note that the equation $d(f,x^p)=d(f,g)$ does not always hold
when $\psi(0)= 0$. 
\end{remark}
%%%%%%%%%%%%%%%%%%%%%%%%%%%%%%%%%%%%%%%%%%%%
\begin{remark}\label{rem:10.3}
When $\psi$ vanishes in high order at the origin, 
it is natural to expect that the set of 
the actual poles might become smaller than 
(\ref{eqn:10.1}); in other words, 
the actual largest element in (b) might become smaller 
than $-1/d(f,x^p)$.
But, there exists a simple example
of $f$ and $g$ violating this expectation 
(see Section~\ref{subsec:15.1}).
Indeed, this example shows that
the set (\ref{eqn:10.1}) is necessary to express 
the set of poles.
Furthermore, 
we can see that 
the assertions in the above theorem hold
even in the case when  
the weight $g$ is flat
(equivalently $\psi$ is flat)
and that they are optimal in the above sense.
\end{remark}
%%%%%%%%%%%%%%%%%%%%%%%%%%%%%%%%%%%%%%%%%%%%%
%%%%%%%%%%%%%%%%%%%%%%%%%%%%%%%%%%%%%%%%%%%%
%%%%%%%%%%%%%%%%%%%%%%%%%%%%%%%%%%%%%%%%%%%%%%%%%%%%%%%%%
\begin{proof}%[Proof of Theorem~\ref{th:10.1}]
We only show 
the assertions in the theorem 
in the case of the functions $Z_{\pm}(s;\varphi)$.

%%%%%%%%%%%%%%%%%%%%%%%%%%%%%%%%%%%%%%%%%%%%%%%%%%%%%%%%%%%%%%%%%%%%%%%%%%%%%%%%%%%%%%%%%%%%%%%%%%%%%%%%%%%%%
\underline{\it Step 1.}\,\, 
({\it Orthant decompositions of $Z_{\pm}(s;\varphi)$}.) 
\quad

Applying the orthant decomposition (\ref{eqn:8.6})
to the functions $Z_{\pm}(s;\varphi)$, 
we have
\begin{equation}\label{eqn:10.3}
Z_{\pm}(s;\varphi)
=\sum_{\theta\in\{-1,1\}^n}
\tilde{Z}_{\pm}(s;\varphi_{\theta};f_{\theta},g_{\theta}),
\end{equation}
where $\varphi_{\theta}(x):=\varphi(\theta_1 x_1,\ldots,\theta_n x_n)$, etc. 
for $\theta=(\theta_1,\ldots,\theta_n)\in\{-1,1\}^n$ and
\begin{equation}\label{eqn:10.4}
\tilde{Z}_{\pm}(s;\varphi;f,g)
(=\tilde{Z}_{\pm}(s;\varphi)):=
\int_{\R_+^n}f(x)_{\pm}^s g(x)\varphi(x)dx.
\end{equation}
Consider the case when $g(x)=x^p\psi(x)$.  
Since $g_{\theta}(x)=\theta^p x^p \psi_{\theta}(x)$, we have 
\begin{equation}\label{eqn:10.5}
Z_{\pm}(s;\varphi)
=\sum_{\theta\in\{-1,1\}^n}
\theta^p
\hat{Z}_{\pm}(s;\varphi_{\theta};f_{\theta},\psi_{\theta}),
\end{equation}
where $\theta^p:=\prod_{j=1}^n\theta_j^{p_j}$ and 
$$
\hat{Z}_{\pm}(s;\varphi;f,\psi)
:=\int_{\R_+^n}f(x)_{\pm}^s \psi(x)\varphi(x)x^p dx.
$$
We remark that 
$\tilde{Z}_{\pm}(s;\varphi;f,g)=
\hat{Z}_{\pm}(s;\varphi;f,\psi)$, 
while  
$\tilde{Z}_{\pm}(s;\varphi_{\theta};f_{\theta},g_{\theta})
=-\hat{Z}_{\pm}(s;\varphi_{\theta};f_{\theta},\psi_{\theta})$
if $\theta^p=-1$. 
From the equations (\ref{eqn:10.3}),(\ref{eqn:10.5}),
it suffices to 
prove the theorem 
in the case of $\tilde{Z}_{\pm}(s;\varphi;f,g)$ or 
$\hat{Z}_{\pm}(s;\varphi;f,\psi)$
instead of $Z_{\pm}(s;\varphi)$. 

%%%%%%%%%%%%%%%%%%%%%%%%%%%%%%%%%%%%%%%%%%%%%%%%%%%%%%
\underline{\it Step 2.}\,\, 
({\it Decompositions of $\tilde{Z}_{\pm}(s;\varphi)$}.) 
\quad

For the moment, we assume ${\rm Re}(s)>0$.
By using the mapping $x=\pi(y)$, 
$\tilde{Z}_{\pm}(s;\varphi)$ can be expressed as 
%%%%%
\begin{eqnarray*}
&&
\tilde{Z}_{\pm}(s;\varphi)=\int_{\R_+^n}f(x)_{\pm}^s g(x)\varphi(x)dx \\
&&
\quad 
=\int_{\tilde{Y}_{\Sigma}} 
((f\circ\pi)(y))_{\pm}^s (g\circ\pi)(y) (\varphi\circ\pi)(y) 
|J_{\pi}(y)|dy, 
\end{eqnarray*}
%%%%%
where $\tilde{Y}_{\Sigma}=Y_{\Sigma}\cap\pi^{-1}(\R_+^n)$ and 
$dy$ is a volume element in $Y_{\Sigma}$.
It is easy to see that 
there exists a set of $C^{\infty}_0$ functions 
$\{\chi_{\sigma}:Y_{\Sigma} \to\R_+: \sigma\in\Sigma^{(n)}\}$ 
satisfying 
the following properties:
%%%%%%%%%%%%%%%%%%%
\begin{itemize}
\item 
For each $\sigma\in\Sigma^{(n)}$, 
the support of the function $\chi_{\sigma}$ is contained 
in $\R^n(\sigma)$ and 
$\chi_{\sigma}$ identically equals one 
in some neighborhood of the origin. 
\item 
$\sum_{\sigma\in\Sigma^{(n)}}\chi_{\sigma}\equiv 1$ 
on the support of 
$\varphi\circ\pi$.  
\end{itemize}
%%%%%%%%%%%%%%%%%%%%
Applying Lemmas~\ref{lem:5.1} and \ref{lem:7.1}
and Theorem~\ref{thm:7.2},
we have 
\begin{equation}\label{eqn:10.6}
\tilde{Z}_{\pm}(s;\varphi)=
\sum_{\sigma\in\Sigma^{(n)}} 
Z_{\pm}^{(\sigma)}(s)
\end{equation} 
with  
%%%%%
\begin{equation}\label{eqn:10.7}
\begin{split}
&Z_{\pm}^{(\sigma)}(s)
:=\int_{\R_+^n} ((f\circ\pi(\sigma))(y))_{\pm}^s 
(g\circ\pi(\sigma))(y)(\varphi\circ\pi(\sigma))(y) 
\chi_{\sigma}(y)|J_{\pi(\sigma)}(y)|dy \\
&
\quad 
=\int_{\R_+^n} 
\left(
\prod_{j=1}^n y_j^{l_f(a^j(\sigma))}f_{\sigma}(y)
\right)_{\pm}^s 
\left(\prod_{j=1}^n y_j^{\langle a^j(\sigma),p\rangle}\right)
\left|
\prod_{j=1}^n y_j^{\langle a^j(\sigma)\rangle -1}
\right|
\tilde{\chi}_{\sigma}(y)dy, 
\end{split}
\end{equation}
%%%%%
where 
$\tilde{\chi}_{\sigma}(y)=
\psi(\pi(\sigma)(y))\varphi(\pi(\sigma)(y))\chi_{\sigma}(y)$.

Consider the functions $Z^{(\sigma)}_{\pm}(s)$ 
for $\sigma\in\Sigma^{(n)}$. 
We easily see the existence of finite sets of 
$C^{\infty}_0$ functions 
$\{\xi _k:\R^n\to\R_+\}$ and 
$\{\eta_l:\R^n\to\R_+\}$ satisfying the following conditions. 
%%%%%%%%%%%%%%%%%%%
\begin{itemize}
\item 
The supports of $\xi_k$ and $\eta_l$ are sufficiently small and 
$\sum_k \psi_k + \sum_l \eta_l \equiv 1$ 
on the support of $\tilde{\chi}_{\sigma}$.
\item 
For each $k$, the function
$f_{\sigma}$ is always positive or negative on the support of $\xi_k$. 
\item 
For each $l$, the support of $\eta_l$ intersects the set 
$\{y\in {\rm Supp}(\tilde{\chi}_{\sigma}):f_{\sigma}(y)=0\}$.
\item
The union of the support of $\eta_l$ for all $l$ contains the set 
$\{y\in {\rm Supp}(\tilde{\chi}_{\sigma}):f_{\sigma}(y)=0\}$.
\end{itemize}
%%%%%%%%%%%%%%%%%%%

By using the functions $\xi_k$ and $\eta_l$, we have
%%%%%
\begin{equation}\label{eqn:10.8}
Z^{(\sigma)}_{\pm}(s)=
\sum_k I^{(k)}_{\pm,\sigma}(s)+
\sum_l J^{(l)}_{\pm,\sigma}(s),
\end{equation} 
%%%%%
with
%%%%%
\begin{equation}\label{eqn:10.9}
\begin{split}
&
I^{(k)}_{\pm,\sigma}(s)=\int_{\R_+^n} 
\left(
\prod_{j=1}^n 
y_j^{l_f(a^j(\sigma))}f_{\sigma}(y)
\right)_{\pm}^s 
\left(\prod_{j=1}^n y_j^{\langle a^j(\sigma),p\rangle}\right)
\left|
\prod_{j=1}^n y_j^{\langle a^j(\sigma)\rangle -1}
\right|
\tilde{\xi}_k(y)dy, 
\\
&
J^{(l)}_{\pm,\sigma}(s)
=\int_{\R_+^n} 
\left(
\prod_{j=1}^n y_j^{l_f(a^j(\sigma))}f_{\sigma}(y)
\right)_{\pm}^s 
\left(\prod_{j=1}^n y_j^{\langle a^j(\sigma),p\rangle}\right)
\left|
\prod_{j=1}^n y_j^{\langle a^j(\sigma)\rangle -1}
\right|
\tilde{\eta}_l(y)dy, 
\end{split}
\end{equation}
%%%%%
where $\tilde{\xi}_k(y)=\tilde{\chi}_{\sigma}(y)\xi_k(y)$ and 
$\tilde{\eta}_l(y)=\tilde{\chi}_{\sigma}(y)\eta_l(y)$. 
If the set 
$\{y\in {\rm Supp}(\tilde{\chi}_{\sigma}):f_{\sigma}(y)=0\}\cap\R_+^n$
is empty, then the functions $J^{(l)}_{\pm,\sigma}(s)$ 
do not appear. 

%%%%%%%%%%%%%%%%%%%%%%%%%%%%%%%%%%%%%%%%%%%%%%%
\underline{\it Step 3.}\,\, 
({\it Poles of $I^{(k)}_{\pm,\sigma}(s)$}.) 
\quad

First, 
consider properties of poles of 
the functions $I^{(k)}_{\pm,\sigma}(s)$. 
Since every $y_j$ is nonnegative in the integrand, 
we have 
%%%%%
\begin{equation}\label{eqn:10.10}
I_{\pm,\sigma}^{(k)}(s)=
\int_{\R_+^n} 
\left(\prod_{j=1}^n 
y_j^{l_f(a^j(\sigma))s+\langle a^j(\sigma),p+\1\rangle-1}
\right)f_{\sigma}(y)_{\pm}^s
\tilde{\xi}_k(y)dy.
\end{equation}
%%%%%
By applying Proposition~\ref{pro:9.2} to 
(\ref{eqn:10.10}), 
each $I_{\pm,\sigma}^{(k)}(s)$ 
can be analytically continued to the whole complex plane as 
meromorphic functions and 
their poles are contained in the set 
%%%%%
\begin{equation}\label{eqn:10.11}
\left\{
-\frac{\langle a^j(\sigma),p+\1\rangle+\nu}{l_f(a^j(\sigma))}:
\nu\in\Z_+, j\in B({\sigma})
\right\},
\end{equation}
%%%%%
where 
%%%%%
\begin{equation}\label{eqn:10.12}
B(\sigma):=\{j:l_f(a^j(\sigma))\neq 0\}
\subset\{1,\ldots,n\}.
\end{equation}
%%%%%

%%%%%%%%%%%%%%%%%%%%%%%%%%%%%%%%%%%%%%%%%%%%%
\underline{\it Step 4.}\,\, 
({\it Poles of $J^{(l)}_{\pm,\sigma}(s)$}.) 
\quad

Next, consider the case of the functions $J^{(l)}_{\pm,\sigma}(s)$. 
By applying Theorem~\ref{thm:7.2} 
and changing the integral variables as in 
Remark~\ref{rem:7.3}, 
$J_{\pm,\sigma}^{(l)}(s)$ can be expressed as follows.
%%%%%
\begin{equation}\label{eqn:10.13}
\begin{split}
&J_{\pm,\sigma}^{(l)}(s)=
\int_{\R_+^n} 
\left((u_i-b_i)\prod_{j\in B_l(\sigma)}
u_j^{l_f(a^j(\sigma))}
\right)_{\pm}^s 
\left(\prod_{j\in B_l(\sigma)}u_j^{\langle a^j(\sigma),p\rangle}\right)\\
&\quad\quad\quad\quad\quad\times
\left|\prod_{j\in B_l(\sigma)}u_j^{\langle a^j(\sigma)\rangle -1}\right|
\hat{\eta}_l(u_1,\ldots,u_i-b_i,\ldots,u_n)du, 
\end{split}
\end{equation}
%%%%%
where $B_l(\sigma)\subsetneq\{1,\ldots,n\}$, 
$i\in \{1,\ldots,n\}\setminus B_l(\sigma)$, 
$b_i>0$ and 
$\hat{\eta}_l\in C_0^{\infty}(\R^n)$
has a support containing the origin.
%We remark that $\phi({\rm Supp}(\hat{\eta}_l))$ is contained in 
%${\rm Supp}(\tilde{\eta}_l)$,
%where $\phi$ is a local diffeomorphism as in 
%Remark~\ref{rem:7.3}.
In a similar fashion to the case of 
$I_{\pm,\sigma}^{(k)}(s)$, 
we have 
%%%%%
%%%%%
\begin{equation}\label{eqn:10.14}
J_{\pm,\sigma}^{(l)}(s)=
\int_{\R_+^n} \left(u_i^s
\prod_{j\in B_l(\sigma)} 
u_j^{l_f(a^j(\sigma))s+\langle a^j(\sigma),p+\1\rangle-1}
\right)
\hat{\eta}_l(u_1,\ldots,\pm u_i,\ldots,u_n)du.
\end{equation}
%%%%%
By applying Proposition~\ref{pro:9.2} 
to (\ref{eqn:10.14}), 
each $J_{\pm,\sigma}^{(l)}(s)$
can be analytically continued to the complex plane as 
meromorphic functions and
their poles are contained in 
the set  
%%%%%
\begin{equation}\label{eqn:10.15}
\left\{
-\frac{\langle a^j(\sigma),p+\1\rangle+\nu}{l_f(a^j(\sigma))}:
\nu\in\Z_+, j\in \tilde{B}_l(\sigma) 
\right\}\cup(-\N), 
\end{equation}
%%%%%
where $\tilde{B}_l(\sigma):=
\{j\in B_l(\sigma):l_f(a^j(\sigma))\neq 0\}$.

%%%%%%%%%%%%%%%%%%%%%%%%%%%%%%%%%%%%%%%%%%%%%%%%%%%%%%%%%%%%%%%%%%%%%
From the relationships (\ref{eqn:10.6}) and (\ref{eqn:10.8}), 
the analyses in Steps~3 and 4 implies that
$Z_{\pm}(s)$ also become meromorphic functions on $\C$ and
their poles are contained in the union of the sets
(\ref{eqn:10.11}) and (\ref{eqn:10.15}) 
for all $\sigma\in \Sigma^{(n)}$,
which is the assertion (a) in the theorem.

%%%%%%%%%%%%%%%%%%%%%%%%%%%%%%%%%%%%%%%%%%%%%%%%%%%%%%%%%%%%%%%%%%%%%%%%%%%%%%%%%%%%%%%%%%%%%%%%%%%%%%%%%%%%%%%%%%%%%%%%%%%%%%%%%%%%%%%%%%

%%%%%%%%%%%%%%%%%%%%%%%%%%%%%%%%%%%%%%%%%%%%%

Now, for $p\in\Z_+^n$, we define
%%%%%
\begin{equation}\label{eqn:10.16}
\beta(p):=\max\left\{
-\dfrac{\langle a,p+\1\rangle}{l_f(a)}: 
a\in\tilde{\Sigma}^{(1)} 
\right\}.
\end{equation}
%%%%%
Note that $\beta(p)$ is the largest element of the first set in
(\ref{eqn:10.1}).

%%%%%%%%%%%%%%%%%%%%%%%%%%%%%%%%%%%%%%%%%%%%%%%%%%%%%%%%%%%%%%%%%%%%%%%%%%%%%%%%%%%%%%%%%%%%%%%%%%%%%%%%%%%%%
\underline{\it Step 5.}\,\, 
({\it Geometrical meanings of} $\b(p)$.) 
\quad

The assertion~(b) in the theorem 
follows from the following lemma, which
reveals the relationship between 
``the values of $\beta(p)$'' and 
``the geometrical conditions of $\Gamma_+(f)$ and $p$''.
%%%%%%%%%%%%%%%%%%%%%%%%%%%%%%%%%%%%%%%%%%%%%%
\begin{lemma}\label{lem:10.4}
Let $q=(q_1,\ldots,q_n)$ be the point of the 
intersection of $\partial\Gamma_+(f)$ 
with the line joining the origin 
and the point $p+\1=(p_1+1,\ldots,p_n+1)$. 
Then 
%%%%%
\begin{eqnarray*}
\beta(p)=-\frac{p_1+1}{q_1}=\cdots=-\frac{p_n+1}{q_n}
=-\frac{\langle p\rangle +n}{\langle q\rangle }=-\frac{1}{d(f,x^p)}.
\end{eqnarray*}
In particular, we have
$q=d(f,x^p)(p+\1)$.
%%%%%
\end{lemma}
%%%%%%%%%%%%%%%%%%%%%%%%%%%%%%%%%%%%%%%%%%%%%%%

%%%%%%%%%%%%%%%%%%%%%%%%%%%%%%%%%%%%%%%%%%%%%%%%
\begin{proof}
For $a \in \Sigma^{(1)}$, 
we denote by $q(a)$ 
the point of the intersection of the hyperplane 
$H(a,l_f(a))$ with the line $\{t\cdot (p+\1):t\in\R\}$, 
where $H(\cdot,\cdot)$ is as in (\ref{eqn:2.1}). 
Then it is easy to see 
%%%%%
\begin{equation}\label{eqn:10.17}
q(a)=\frac{l_f(a)}{\langle a,p+\1\rangle }\cdot (p+\1). 
\end{equation}
%%%%%
From (\ref{eqn:10.17}),  
the condition that 
$-\langle a,p+\1\rangle /l_f(a)$ takes the maximum
is equivalent to 
the geometrical condition that 
$q(a)$ is as far as possible from the origin. 
To be more precise, we have the following equivalences: 
For $a\in\tilde{\Sigma}^{(1)}$,  
%%%%%
\begin{equation}\label{eqn:10.18}
\beta(p)=-\frac{\langle a,p+\1\rangle }{l_f(a)} \, \Longleftrightarrow \,
q=q(a) \, \Longleftrightarrow \,
q\in H(a,l_f(a)). 
\end{equation}
%%%%%
From (\ref{eqn:10.17}) and (\ref{eqn:10.18}), we have
$-\beta(p)=(p_1+1)/q_1=
\cdots=(p_n+1)/q_n=(\langle p\rangle +n)/\langle q\rangle $.
From the definition of $d(\cdot,\cdot)$, 
the above value equals $1/d(f,x^p)$. 
\end{proof}

%%%%%%%%%%%%%%%%%%%%%%%%%%%%%%%%%%%%%%%%%%%%%%%%%%%%%%%%%%%%%%%%%%%%%%%%%%%%%%%%%%%%%%%%%%%%%%%%%%%%%%%%%%%%%%
\underline{\it Step 6.}\,\, 
({\it Orders of the poles at} $\b(p)$.) 
\quad

Let us consider the orders of the poles of 
$\tilde{Z}_{\pm}(s;\varphi)$
at $s=\b(p)$.
For $\sigma\in\Sigma^{(n)}$, let  
%%%%%
\begin{equation}\label{eqn:10.19}
A_p(\sigma):=\left\{
j\in B(\sigma):
\beta(p)=-\frac{\langle a^j(\sigma),p+\1\rangle }{l_f(a^j(\sigma))} 
\right\}\subset\{1,\ldots,n\},
\end{equation}
%%%%%
where $B(\sigma)$ is as in (\ref{eqn:10.12}).
From (\ref{eqn:10.6}), (\ref{eqn:10.8}),
it suffices to analyze the poles of 
$I_{\pm,\sigma}^{(k)}(s)$ and 
$J_{\pm,\sigma}^{(l)}(s)$.
Applying Proposition~\ref{pro:9.2}
to the integrals 
(\ref{eqn:10.10}) and (\ref{eqn:10.14}), 
we see the upper bounds of orders of the poles 
at $s=\beta(p)$ 
of $I_{\pm,\sigma}^{(k)}(s)$
and $J_{\pm,\sigma}^{(l)}(s)$ as follows.
%%%%%%%%%%%%%%%%%%%%%%%%%
\begin{center}
\begin{tabular}{l|l} \hline
\textit{$I_{\pm,\sigma}^{(k)}(s)$} & 
       \quad $\# A_p(\sigma)$ \\ \hline
\textit{$J_{\pm,\sigma}^{(l)}(s)$} & 
       $\min\{\# A_p(\sigma),n-1\}$ if $\beta(p)\not\in (-\N)$\\
             & $\min\{\# A_p(\sigma)+1,n\}$ if $\beta(p) \in (-\N)$ 
\\ \hline 
\end{tabular}
\end{center}
%%%%%%%%%%%%%%%%%%%%%%%%%

The assertion (c) in the theorem follows 
from the above table and Lemma~\ref{lem:10.5}
below, 
which shows the geometrical meaning of 
$\# A_p(\sigma)$.
(The estimate 
$\# A_p(\sigma)\leq m(f,x^p)$ 
is sufficient for the proof of the assertion (c). 
The existence of $\sigma\in\Sigma^{(n)}$ 
attaining the equality will be used
in the analysis in the next subsection.) 
\end{proof}

%%%%%%%%%%%%%%%%%%%%%%%%%%%%%%%%%%%%%%%%%%%%%%%%%%%%
\begin{lemma}\label{lem:10.5}
Let $q$ be as in Lemma~$\ref{lem:10.4}$. 
%which is characterized by $q=d(f,x^p)(p+\1)$.
Then the equations $m(f,x^p)=n-\dim(\tau_f(q))=
\max\left\{\# A_p(\sigma)
:\sigma\in\Sigma^{(n)}
\right\}$
hold.
\end{lemma}
%%%%%%%%%%%%%%%%%%%%%%%%%%%%%%%%%%%%%%%%%%%%%%%%%%%%

%%%%%%%%%%%%%%%%%%%%%%%%%%%%%%%%%%%%%%%%%%%%%%%%%%%%
\begin{proof}%[Proof of Lemma~$\ref{lem:10.5}$]
In the case when $g(x)=x^p$, 
we have $\Gamma_+(g)=\{p\}+\R_+^n$.
Since
each nonempty proper face of $\Gamma_+(g)$
contains the vertex $p$,
there exists only one principal face 
$\gamma_*$ of $\Gamma_+(g)$ 
and $p$ is contained in $\gamma_*$,
which implies that 
$n-\dim(\tau_f(q))=m(f,x^p)$.
From the definition of $A_p(\sigma)$
and (\ref{eqn:10.18}), we have 
%%%%%
\begin{eqnarray*}
&&A_p(\sigma)
=\{j : q\in H(a^j(\sigma),l_f(a^j(\sigma)))\} \\
&& \quad=\{j : 
\tau_f (q)\subset H(a^j(\sigma),l_f(a^j(\sigma)))\}
=I(\tau_f (q),\sigma),
\end{eqnarray*}
%%%%%
where $I(\cdot,\cdot)$ is as in (\ref{eqn:6.5}).
Lemma~\ref{lem:6.1} (ii)
implies that $\dim(\tau_f (q))
\leq n-\#  I(\tau_f (q),\sigma)
= n-\#  A_p(\sigma)$ for any 
$\sigma\in\Sigma^{(n)}$.
%%%
On the other hand, Lemma~\ref{lem:6.2} (ii)
implies that $\Sigma^{(n)}(\tau_f(q))\neq\emptyset$, 
i.e.,
there exists $\sigma\in\Sigma^{(n)}$ such that 
$\dim (\tau_f(q))=n-\#  A_p(\sigma)$.
%which implies $n-\dim(\tau_f(q))=
%\max\{\# A_p(\sigma);\sigma\in\Sigma^{(n)}\}$. 
\end{proof}
%%%%%%%%%%%%%%%%%%%%%%%%%%%%%%%%%%%%%%%%%%%%
%%%%%%%%%%%%%%%%%%%%%%%%%%%%%%%%%%%%%%%%%%%%%%%%%%%%%%%
%%%%%%%%%%%%%%%%%%%%%%%%%%%%%%%%%%%%%%%%%%%%%%%%%%%%%%%%%%%%%%%%%%%%%%%%%%%%%%%%%%%%%%%%%%%%%%%%%%%%%%%%%%%%%%%%%%%%%%%%%%%%%%

%%%%%%%%%%%%%%%%%%%%%%%%%%%%%%%%%%%%%%%%%%%%%%%%%%%%%%%%%%%%%%%%%%%%%%%%%%%%%%%%%%%%%%%%%%%%%%%%%%%%%%%%%%%%%%%%%%%%%%%%%%%%%%
\subsection{First coefficients}
Next, let us consider the coefficients 
of the most important term of the Laurent expansions
of $\tilde{Z}_{\pm}(s;\varphi)$.

We define the subset $\Sigma_p^{(n)}$ 
in $\Sigma^{(n)}$ consisting of 
important cones as 
%%%%%
\begin{equation}\label{eqn:10.20}
\Sigma_p^{(n)}:=\left\{\sigma\in\Sigma^{(n)}:
\#A_p(\sigma)=m(f,x^p) \right\}.
\end{equation}
%%%%%
It follows from Lemma~\ref{lem:10.5} that 
$\Sigma_p^{(n)}$ is nonempty. 
From the definition of $m(f,x^p)$ and 
Lemma~\ref{lem:6.2}, 
we can see the following:
%%%%%
\begin{eqnarray}\label{eqn:10.21}
&&\quad\quad\sigma\in\Sigma_p^{(n)} \nonumber\\
&&\Longleftrightarrow
\dim(\tau_f (q))=n-m(f,x^p)=n-\#  A_p(\sigma)
=n-\# I(\tau_f (q),\sigma) \nonumber\\
&&%\quad\quad
\Longrightarrow
\tau_f (q)=\bigcap_{j\in A_p(\sigma)}
H(a^j(\sigma),l(a^j(\sigma)))\cap\Gamma_+(f),
\end{eqnarray}
%%%%%
where $q:=d(f,x^p)(p+\1)$.
Note the equation 
$A_p(\sigma)=I(\tau_f (q),\sigma)$ in the proof of 
Lemma~\ref{lem:10.5}.
Hence, 
the equation $\gamma(I,\sigma)=\gamma$ holds
for $\sigma\in\Sigma_p^{(n)}$,
$I=A_p(\sigma)$
and $\gamma=\tau_f (q)$,  
which was an important condition in Lemma~7.1.

Let us compute the coefficients of 
$(s-\beta(p))^{-m(f,x^p)}$ 
in the Laurent expansions of 
$\tilde{Z}_{\pm}(s;\varphi)$. 
%%%%%
Respectively, we define
$$ 
\tilde{C}_{\pm}:=\lim_{s\to\beta(p)} 
(s-\beta(p))^{m(f,x^p)} 
\tilde{Z}_{\pm}(s;\varphi). 
$$
%%%%%

%%%%%%%%%%%%%%%%%%%%%%%%%%%%%%%%%%%%%%%%%%%%%%%%
\begin{proposition}\label{pro:10.6}
Suppose that 
{\rm (i)} $f$ satisfies the condition $(E)$,
{\rm (ii)} $g(x)=x^p\psi(x)$, 
where $p\in \Z_+^n$ and
$\psi$ is a smooth function defined on $U$ 
and 
{\rm (iii)} at least one of the following
conditions is satisfied. 
%%%%%
\begin{enumerate}
\item[(a)] $d(f,x^p)>1;$
\item[(b)] 
$f_{\sigma}\circ T_{A_p(\sigma)}$ 
does not vanish on $\R_+^n\cap \pi(\sigma)^{-1}(U)$
for any $\sigma\in\Sigma_p^{(n)}$.
\end{enumerate}
%%%%%
Then we give explicit formulae for coefficients$:$ 
$
\tilde{C}_{\pm}=G_{\pm}(f,\psi,\varphi),
$
where 
$G_{\pm}(f,\psi,\varphi)$ are as in 
$(\ref{eqn:10.26}),(\ref{eqn:10.29}),(\ref{eqn:10.30}),(\ref{eqn:10.31})$
in the proof of this proposition. 
\end{proposition}

%%%%%%%%%%%%%%%%%%%%%%%%%%%%%%%%%%%%%%%%%%%%%%%%%

%%%%%%%%%%%%%%%%%%%%%%%%%%%%%%%%%%%%%%%%%%%%%%%
\begin{proof}
%%%%%%%%%%%%%
In this proof, 
we use the following notation and symbols 
to decrease the complexity in the expressions 
of the integrals, below. 
%%%%%
\begin{itemize}
%%
%\item $A=A_p(\sigma)\subset\{1,\ldots,n\}$;
%%
\item 
$\prod_{j\not\in A_p(\sigma)}y_j^{a_j}dy_j$ means 
$\prod_{j\not\in A_p(\sigma)}y_j^{a_j}\cdot
\prod_{j\not\in A_p(\sigma)}dy_j$ 
with $a_j \geq 0$.
\item $L_{\sigma}:=\prod_{j\in A_p(\sigma)}l_f(a^j(\sigma))^{-1}$.
\item $M_j(\sigma)
:=-l_f(a^j(\sigma))/d(f,x^p)+
\langle  a^j(\sigma),p+\1 \rangle$.
%%
%\item $\delta_{\sigma,p}(\theta):=
%\prod_{j=1}^n\theta_j^{\langle a^j(\sigma),p\rangle }\in\{-1,1\}$;
%%
%\item $\epsilon_{f,\sigma}(\theta):=
%\prod_{j=1}^n\theta_j^{l_f(a^j(\sigma))}\in\{-1,1\}$;
%%
\item 
If $a=0$, then the value of $a^{-1/d(f,x^p)}$ is defined by $0$.
\item 
$((\psi\cdot\varphi)\circ\pi(\sigma))(T_A(y))
:=(\psi\circ\pi(\sigma))(T_A(y))\cdot
(\varphi\circ\pi(\sigma))(T_A(y))$.
\end{itemize}
%%%%%
Note that $M_j(\sigma)$ is a nonnegative constant and,   
moreover, $M_j(\sigma)= 0$ if and only if $j\in A_p(\sigma)$. 

We divide the computation into the following two cases:
$m(f,x^p)<n$ and $m(f,x^p)=n$. 

\underline{The case: $m(f,x^p)<n$.}\quad

%%%%%

\underline{Under the hypothesis (a).}\quad
First, we consider the case when the hypothesis (a):
$d(f,x^p)>1$ is satisfied.
%Let $\hat{\Sigma}^{(n)}=
%\{\sigma\in\Sigma^{(n)};{}^\sharp A_p(\sigma)=\rho_f(p)\}$. 
Respectively, we define
$$ 
\tilde{C}_{\pm}(\sigma):=
\lim_{s\to\beta(p)} 
(s-\beta(p))^{m(f,x^p)} 
Z_{\pm}^{(\sigma)}(s), 
$$
where $Z_{\pm}^{(\sigma)}(s)$ 
are as in (\ref{eqn:10.7}). 
If $\sigma\not\in\Sigma_p^{(n)}$,
then $\tilde{C}_{\pm}(\sigma)=0$.
Thus, 
it suffices to consider the case when 
$\sigma\in\Sigma_p^{(n)}$. 
Considering the equations (\ref{eqn:10.8}) and
applying Proposition~\ref{pro:9.3} 
to (\ref{eqn:10.10}) and (\ref{eqn:10.14})
with $A_*=A_p(\sigma)$, 
we have
%%%%%
\begin{equation}\label{eqn:10.22}
\begin{split}
\quad \tilde{C}_{\pm}(\sigma)
=\sum_{k} G_{k,\pm}(\sigma)+\sum_{l} H_{l,\pm}(\sigma),
\end{split}
\end{equation}
%%%%%
with
%%%%%
\begin{equation}\label{eqn:10.23}
G_{k,\pm}(\sigma)=L_{\sigma}
\int_{\R_+^{n-m(f,x^p)}}
\frac{
\tilde{\xi}_{k}(T_{A_p(\sigma)}(y))
}{
f_{\sigma}(T_{A_p(\sigma)}(y))_{\pm}^{1/d(f,x^p)}
}
\prod_{j\not\in A_p(\sigma)}y_j^{M_j(\sigma)-1}dy_j, 
\end{equation}
%%%%%
and
\begin{equation}\label{eqn:10.24}
\begin{split}
&H_{l,\pm}(\sigma)= \\
&L_{\sigma}
%\cdot 
%\quad\quad
\int_{\R_+^{n-m(f,x^p)}}
\frac{
\hat{\eta}_{l}
(T_{A_p(\sigma)}(u_1,\ldots,\pm u_i,\ldots,u_n))
}
{u_i^{1/d(f,x^p)}} 
\prod_{j\in B_l(\sigma)\setminus A_p(\sigma)
%_p(\sigma)
}u_j^{M_j(\sigma)-1}
\prod_{j\not\in A_p(\sigma)}du_j,
\end{split}
\end{equation}
%%%%%
%%%
where 
$\tilde{\xi}_{k}$, $\hat{\eta}_{l}$,
$B_l(\sigma)$, $i$ are as 
in (\ref{eqn:10.10}), (\ref{eqn:10.14}).
The summations
 in (\ref{eqn:10.22}) are taken for all $k$, $l$ 
satisfying 
$T_{A_p(\sigma)}(\R^n)
\cap {\rm Supp}(\xi_k)\neq \emptyset$ 
and $A_p(\sigma)\subset B_l(\sigma)$, 
respectively. 
We remark that the values of 
$G_{k,\pm}(\sigma)$ and 
$H_{l,\pm}(\sigma)$ depend on 
the cut-off functions $\chi_{\sigma}$, $\xi_k$, $\eta_l$
in Section~\ref{sec:10.1}. 
%%%
Since $d(f,x^p)>1$, 
the integrals in (\ref{eqn:10.23}),(\ref{eqn:10.24}) 
are convergent and they are
interpreted as improper integrals. 
%%%

In (\ref{eqn:10.23}), (\ref{eqn:10.24}), 
we deform the cut-off functions $\xi_k$ and $\eta_l$ 
as the volume of the support of $\eta_l$ tends to
zero for all $l$.  
%%%
Then, it is easy to see that 
the limit of $H_{l,\pm}(\sigma)$ is zero, while 
that of $\sum_k G_{k,\pm}(\sigma)$ 
can be computed explicitly. 
%%%%%
%%%%%
Considering the equation (\ref{eqn:10.22})
and the deformation of $\xi_k$ and $\eta_l$ in the above,
we have
%%%%%
\begin{equation}\label{eqn:10.25}
\begin{split}
\tilde{C}_{\pm}(\sigma)=
L_{\sigma}
\int_{\R_+^{n-m(f,x^p)}}
\frac{
\tilde{\chi}_{\sigma}(T_{A_p(\sigma)}(y))
}
{
\left(
f_{\sigma}(T_{A_p(\sigma)}(y))
\right)_{\pm}^{1/d(f,x^p)}
}
\prod_{j\notin A_p(\sigma)}
y_j^{M_j(\sigma)-1}
dy_j,
\end{split}
\end{equation}
%%%%%
where $\tilde{\chi}_{\sigma}$ is as in (\ref{eqn:10.7}).

Furthermore, 
let us compute the limits $\tilde{C}_{\pm}$ explicitly. 
If the cut-off function $\chi_{\sigma}$ is deformed as 
the volume of the support of $\chi_{\sigma}$ tends to zero, 
then $\tilde{C}_{\pm}(\sigma)$ tends to zero. 
Notice that each $\R_+^n(\sigma)$ 
is densely embedded in $\tilde{Y}_{\Sigma}$ 
(see Section~\ref{subsec:5.2})
and 
that
$\tilde{C}_{\pm}=
\sum_{\sigma\in\Sigma_p^{(n)}}
\tilde{C}_{\pm}(\sigma)$.
Thus,  
for an arbitrary fixed cone $\sigma\in\Sigma_p^{(n)}$, 
we have $\tilde{C}_{\pm}=G_{\pm}(f,\psi,\varphi)$ with
%%%%%
\begin{equation}\label{eqn:10.26}
G_{\pm}(f,\psi,\varphi)=
L_{\sigma}
\int_{\R_+^{n-m(f,x^p)}}
\frac{((\psi \cdot \varphi)\circ \pi(\sigma))
(T_{A_p(\sigma)}(y))}
{\left(
f_{\sigma}(T_{A_p(\sigma)}(y))\right)_{\pm}^{1/d(f,x^p)}}
\prod_{j\notin A_p(\sigma)}
y_j^{M_j(\sigma)-1}
dy_j.
\end{equation}
%%%%%
From the above computation, we can see that 
the value of the above integral must be independent of 
$\sigma\in\Sigma_p^{(n)}$.
Since the function $f_{\sigma}$ in (\ref{eqn:10.26}) appears
after the process of toric resolution, 
one should look for the other formulae, 
which are expressed by $\psi,\varphi,f$ more directly. 
%%%

Let us give the other formulae of $G_{\pm}(f,\psi,\varphi)$. 
From the equation (\ref{eqn:10.21}), Lemma~\ref{lem:7.1} implies 
%%%%%
\begin{equation}\label{eqn:10.27}
(f_{\tau_f (q)}\circ\pi(\sigma))(y)
=
\left(
\prod_{j=1}^n
y_j^{l_f(a^j(\sigma))}
\right)
f_{\sigma}(T_{A_p(\sigma)}(y)).
\end{equation}
Moreover, we have
\begin{equation}\label{eqn:10.28}
(f_{\tau_f (q)}\circ\pi(\sigma))
(T_{A_p(\sigma)}^1(y))
=
\left(
\prod_{j\not\in A_p(\sigma)}
y_j^{l_f(a^j(\sigma))}
\right)
f_{\sigma}(T_{A_p(\sigma)}(y)),
\end{equation}
where $T_{A_p(\sigma)}^1(\cdot)$ is as in 
(\ref{eqn:1.4}). 
%%%%%
By using the above equation, (\ref{eqn:10.26}) 
can be rewritten as 
%%%%%
\begin{equation}\label{eqn:10.29}
\begin{split}
&G_{\pm}(f,\psi,\varphi)= \\
&L_{\sigma}
\int_{\R_+^{n-m(f,x^p)}}
\frac{
\left((\psi \cdot \varphi)\circ \pi(\sigma)\right)(T_{A_p(\sigma)}(y))
}
{
\left(
(f_{\tau_f (q)}\circ \pi(\sigma))
(T^1_{A_p(\sigma)}(y))
\right)_{\pm}^{1/d(f,x^p)}
}
\prod_{j\notin A_p(\sigma)}
y_j^{\langle a^j(\sigma)\rangle -1}
dy_j.
\end{split}
\end{equation}

%%%%%

%%%%%
\underline{Under the hypothesis (b).}\quad
Secondly, we consider the case that the hypothesis (b)
is satisfied.
In this case, 
we immediately obtain the equations (\ref{eqn:10.25})
for $\sigma\in\Sigma_p^{(n)}$ 
in the same argument as in the case of the hypothesis (a), 
because $J_{\pm,\sigma}^{(l)}(s)$ do not appear 
in the decompositions (\ref{eqn:10.8}) and 
$f_{\sigma}\circ T_{A_p(\sigma)}$ does not vanish. 
But, in this case, it must be careful that 
$\tilde{C}_{\pm}(\sigma)$ does not always vanish
even in the case when $\sigma\not\in\Sigma_p^{(n)}$.
Indeed, some $J_{\pm,\sigma}^{(l)}(s)$ may have 
at $s=-1/d(f,x^p)$ a pole of order $m(f,x^p)$ 
even if $\sigma\not\in\Sigma_p^{(n)}$.
%%%
Here, the coefficients 
$\tilde{C}_{\pm}(\sigma)$ 
for $\sigma\not\in\Sigma_p^{(n)}$ can be computed 
in a similar argument of the proof of Lemma~\ref{lem:9.1},
since $I_{\pm,\sigma}^{(k)}(s)$ do not
have at $s=-1/d(f,x^p)$ poles of order 
$m(f,x^p)$. 
From the result of this computation, 
it is easy to see that these
coefficients tend to zero 
if the volume of the support of $\chi_{\sigma}$  
tends to zero. 
Therefore, the limits $\tilde{C}_{\pm}$ can be samely 
computed as in (\ref{eqn:10.26}) and (\ref{eqn:10.29}),
where $\sigma$ is as in the hypothesis (b).
We remark that $\tilde{C}_{+}$ or 
$\tilde{C}_{-}$ is equal to zero in this case. 

%%%%%%%%%%%%%%%%%%%%%%%%%%%%%%%%%%%%%%%%%%%%%%%%%%%
\underline{The case: $m(f,x^p)=n$.}\quad

In this case, we see that 
$A_p(\sigma)=\{1,\ldots,n\}$, 
$m(f,x^p)=n$ and 
$\tau_f(q)=q$.
Similar computations give 
the following expressions corresponding to (\ref{eqn:10.26}):
%%%%%
\begin{equation}\label{eqn:10.30}
G_{\pm}(f,\psi,\varphi)= 
L
\frac{\psi(0)\varphi(0)}
{
f_{\sigma}(0)^{1/d(f,x^p)}_{\pm}},
\end{equation}
%%%%%
where $\sigma\in\Sigma_p^{(n)}$ and
$L=\sum_{\sigma\in\Sigma_p^{(n)}}L_{\sigma}$.

From the equation (\ref{eqn:10.28}),
we obtain the other expressions
corresponding to (\ref{eqn:10.29}):
%%%%%
\begin{equation}\label{eqn:10.31}
\begin{split}
G_{\pm}(f,\psi,\varphi)=
L\frac{\psi(0)\varphi(0)}
{\left(
f_{\tau_f (q)}(\1)
\right)^{1/d(f,x^p)}_{\pm}}
=
L
\frac{(q!)^{1/d(f,x^p)}\psi(0)\varphi(0)}
{(\d^{q}f(0))^{1/d(f,x^p)}_{\pm}},
\end{split}
\end{equation}
where $\1=(1,\ldots,1)$ and 
$L$ is the same as in (\ref{eqn:10.30}).

\end{proof}
%%%%%%%%%%%%%%%%%%%%%%%%%%%%%%%%%%%%%%%%%%%

\begin{remark}
Let us consider the case when $\tau_f(q)$ is compact. 
From (\ref{eqn:10.21}), Lemma~\ref{lem:6.4} implies 
that $\pi(\sigma)(T_{A_p(\sigma)}(\R^n))=0$. 
Therefore, the formulae (\ref{eqn:10.26}), (\ref{eqn:10.29})
can be expressed in more simple forms as follows. 
%%%%%%%%%%%%%%%%%%%%%
\begin{equation*}
\begin{split}
&G_{\pm}(f,\psi,\varphi)=
L_{\sigma}\psi(0)\varphi(0)
\int_{\R_+^{n-m(f,x^p)}}
\frac{
\prod_{j\notin A_p(\sigma)}
y_j^{M_j(\sigma)-1}
dy_j
}{
\left(
f_{\sigma}(T_{A_p(\sigma)}(y))\right)_{\pm}^{1/d(f,x^p)}
}
\\
&\quad\quad\quad=L_{\sigma}\psi(0)\varphi(0)
\int_{\R_+^{n-m(f,x^p)}}
\frac{
\prod_{j\notin A_p(\sigma)}
y_j^{\langle a^j(\sigma)\rangle -1}
dy_j
}{
\left(
(f_{\tau_f (q)}\circ \pi(\sigma))
(T^1_{A_p(\sigma)}(y))
\right)_{\pm}^{1/d(f,x^p)}}.
\end{split}
\end{equation*}
\end{remark}

Finally, let us compute the coefficients of 
$(s-\beta(p))^{-m(f,x^p)}$ 
in the Laurent expansions of 
$Z_{\pm}(s;\varphi)$, $Z(s;\varphi)$.
Respectively, we define 
%%%%%
$$ 
C_{\pm}:=\lim_{s\to\beta(p)} 
(s-\beta(p))^{m(f,x^p)} 
Z_{\pm}(s;\varphi), 
\quad
C:=\lim_{s\to\beta(p)} 
(s-\beta(p))^{m(f,x^p)} 
Z(s;\varphi).
$$
%%%%%

%%%%%%%%%%%%%%%%%%%%%%%%%%%%%%%%%%%%%%%%%%%%%%%%%
\begin{theorem}\label{thm:10.7}
Suppose that 
{\rm (i)} $f$ satisfies the condition $(E)$,
{\rm (ii)} $g(x)=x^p\psi(x)$, 
where $p\in \Z_+^n$ and
$\psi$ is a smooth function defined on $U$
and
{\rm (iii)}
at least one of the following
conditions is satisfied. 
%%%%%
\begin{enumerate}
\item[(a)] $d(f,g)=d(f,x^p)>1;$
\item[(b)] $f$ is nonnegative or nonpositive on $U;$
\item[(c)] $f_{\tau_*}$ does not vanish on 
$(\R\setminus\{0\})^n$, 
where $\tau_*=\tau_f(q)$ is the principal face of $\Gamma_+(f)$.
\end{enumerate}
%%%%%
Then we give explicit formulae for the coefficients in the following.
\begin{equation}\label{eqn:10.32}
C_{\pm}=
\left(
\prod_{j=1}^n\theta_j^{p_j}
\right)
\sum_{\theta\in\{-1,1\}^n}
G_{\pm}(f_{\theta},\psi_{\theta},\varphi_{\theta})
\mbox{\quad and \quad} 
C=C_+ + C_-,
\end{equation} 
where 
$f_{\theta}(x)=f(\theta_1 x_1,\ldots,\theta_n x_n)$, etc. and 
$G_{\pm}(f,\psi,\varphi)$
are as in 
$(\ref{eqn:10.26}),(\ref{eqn:10.29})$, 
$(\ref{eqn:10.30}),(\ref{eqn:10.31})$.
%%%%

Furthermore, the formulae $(\ref{eqn:10.32})$
imply the following. 
Suppose that 
{\rm (iv)} every component of $p\in \Z_+^n$ is even 
and 
{\rm (v)} 
$\psi(0)\varphi(0)$ is positive $($resp. negative$)$
and $\psi\cdot\varphi$ is nonnegative $($resp. nonpositive$)$ on $U$. 
Then if the support of $\varphi$ is sufficiently small, 
then 
$C_{\pm}$ are nonnegative 
$($resp. nonpositive$)$ and 
$C=C_{+}+C_{-}$ is positive
$($resp. negative$)$. 
\end{theorem}
%%%%%%%%%%%%%%%%%%%%%%%%%%%%%%%%%%%%%%%%%%%%%%%%%

\begin{proof}
From the equations (\ref{eqn:10.3}), (\ref{eqn:10.5}), 
in order to obtain the explicit formulae 
(\ref{eqn:10.32}), 
we must show every condition in (iii)
implies the condition (a) or (b) 
in Proposition~\ref{pro:10.6} (iii).

Since 
the case (a) is easily shown,  
we only consider the cases (b) and (c). 
It suffices to show that 
when the support of $\varphi$
is contained in a sufficiently small neighborhood 
of the origin,
the conditions
(b) and (c) imply the condition:
$f_{\sigma}\circ T_{A_p(\sigma)}$ does not 
vanish on $\R^n\cap\pi(\sigma)^{-1}(U)$ 
for any $\sigma\in\Sigma_p^{(n)}$.
Note that the difference from 
the condition (b) in Proposition~\ref{pro:10.6}
is only in ``$\R^n\cap\pi(\sigma)^{-1}(U)$''.
Let us show this by contradiction. 

Now, let us assume that 
for some $\sigma \in \Sigma_p^{(n)}$
there exists a point 
$b_0\in T_{A_p(\sigma)}(\R^n)\cap \pi(\sigma)^{-1}(U)$
such that $f_{\sigma}(b_{0})=0$.
Since $f$ is nondegenerate over $\R$ 
with respect to its Newton polyhedron,
Theorem~\ref{thm:7.2} implies that
there are points $b_1, b_2 \in T_{A_p(\sigma)}(\R^n)\cap
\pi(\sigma)^{-1}(U)$
near $b_0$ such that
$f_{\sigma}(b_1)>0$ and $f_{\sigma}(b_2)<0$.
%%%
By using the equations (\ref{eqn:7.3}),(\ref{eqn:10.27}),
it is easy to see that 
the conditions (b) and (c) induce
the contradiction to 
the existence of the above points $b_1,b_2$.

Finally, we show the sign of the values of $C_{\pm}$ 
and $C=C_++C_-$. 
Since the nonnegativities or the nonpositivities of $C_{\pm}$
can be directly seen from the explicit formulae with conditions 
(iii)-(v), 
we only show the nonvanishing of 
the value of $C=C_+ +C_-$. 
Since $f_{\sigma}(0)>0$, 
the equation (\ref{eqn:10.28})  implies that
$f_{\tau_f (q)}\circ \pi(\sigma)\circ T^1_{A_p(\sigma)}$
does not identically equal zero near the origin. 
Therefore, the conditions on $\psi,\varphi$ in (iv),(v)
imply the nonvanishing of $C=C_+ +C_-$.
\end{proof}

%%%%%%%%%%%%%%%%%%%%%%%%%%%%%%%%%%%%%%%%%%%%%%%%%%%%%%%%%%
\subsection{Poles on negative integers}\label{subsec:10.3}

Let us consider the poles of $Z_{\pm}(s;\varphi)$
at negative integers in more detail. 

For $\lambda\in-\N$,
define 
%%%%%
\begin{eqnarray*}
&&
A_{\lambda}(\sigma,p):=\{j\in B(\sigma); 
l_f(a^j(\sigma))\lambda+\langle a^j(\sigma),p+\1
\rangle-1 \in-\N\},\\
&&
\rho_{\lambda}(p):=
\min\{\max\{\#  A_{\lambda}(\sigma,p):\sigma\in\Sigma^{(n)}\},n-1\}.
\end{eqnarray*}
%%%%%
The following proposition is concerned with 
the poles of 
$Z_{\pm}(s;\varphi)$, 
which are induced by the set of zeros of $f_{\sigma}$,
and will be used in 
the computation of the coefficients 
of the asymptotic expansion (\ref{eqn:1.2})
of $I(t;\varphi)$ 
(see the proof in Section~\ref{subsec:14.3}).

%%%%%%%%%%%%%%%%%%%%%%%%
\begin{proposition}\label{pro:10.8}
Suppose that 
{\rm (i)} $f$ satisfies the condition $(E)$ and 
{\rm (ii)} $g(x)=x^p \psi(x)$,
where $p\in \Z_+^n$ and 
$\psi$ is a smooth function defined on $U$.
If the support of $\varphi$ is
contained in a sufficiently small neighborhood 
of the origin, then the orders of poles of 
$Z_{\pm}(s;\varphi)$  
at $s=\lambda\in -\N$ are not 
higher than $\rho_{\lambda}(p)+1$.
In particular, if $\lambda>-1/d(f,x^p)$, 
then these orders are not higher than $1$.    
Moreover, 
let $a_{\lambda}^{\pm}$ be the coefficients of 
$(s-\lambda)^{-\rho_{\lambda}(p)-1}$ 
in the Laurent expansions of 
$Z_{\pm}(s;\varphi)$ at $s=\lambda$, 
respectively,  
then we have $a_{\lambda}^+=(-1)^{\lambda-1}a_{\lambda}^-$ 
for $\lambda\in-\N$.
\end{proposition}
%%%%%%%%%%%%%%%%%%%%%%%%

%%%%%%%%%%%%%%%%%%%%%%%%%%%%%%%%%%%%%%%%%%%%%%%%
\begin{proof}
By observing Steps 2 and 4 in the proof of 
Theorem \ref{thm:10.1},
it suffices to investigate the poles of each 
$J_{\pm,\sigma}^{(l)}(s)$.
Applying Lemma~\ref{lem:9.4}
to (\ref{eqn:10.14}),
%and considering the second equations in (\ref{eq:10.14}),
we obtain the theorem.
\end{proof}
%%%%%%%%%%%%%%%%%%%%%%%%%%%%%%%%%%%%%%%%%%%%%%%%
 
%%%%%%%%%%%%%%%%%%%%%%%%%%%%%%%%%%%%%%%%%%%%%%%%
\subsection{Generalization of Varchenko's results to the Puiseux series case}

In this subsection, 
we introduce a new class of functions 
which admit asymptotic expansions at the origin
of the form of a fractional power series 
and generalize the results about unweighted local zeta type functions
(\ref{eqn:8.1}), (\ref{eqn:8.2}) with $g\equiv 1$ 
due to Varchenko \cite{var76} and Kamimoto and Nose \cite{kn13} 
in the case when $f$ belongs to this class. 
This generalization can be easily done
by using the results in the monomial weighted case.

Hereafter in this subsection, 
let $U$ be a small open neighborhood of the origin in $\R^n$.  
We denote by $U_+$ the restriction of $U$ to $\R^n_+$. 
%%%%%%%%%%%%%%%%%%%%%%%%%%%%%%
\begin{definition}
Let $p=(p_1,\ldots,p_n)\in\N^n$.
We write $\alpha/p:=(\alpha_1/p_1,\ldots,\alpha_n/p_n)$ 
and 
$x^{\alpha/p}
:=x_1^{\alpha_1/p_1}\cdots x_n^{\alpha_n/p_n}$
for $\alpha=(\alpha_1,\ldots,\alpha_n)\in\Z_+^n$.
%%%%%%%%%%%%%%%%%%%%%%%%%%%%%%%%
\begin{enumerate}
\item
We denote by $C_{1/p}^{\infty}(U_+)$ the set of
functions $f$ defined on $U_+$ for which 
there exist an open neighborhood $V$
of the origin in $\R^n$ and 
a smooth function $F_f$ defined on $V$ satisfying 
the condition (A) in Section~3 such that
$U_+\subset \Phi_{1/p}^{-1}(V)$ and 
$f=F_f\circ \Phi_{1/p}$ on $U_+$, where 
$\Phi_{1/p}(x):=(x_1^{1/p_1},\ldots,x_n^{1/p_n})$.
%%%%%%%%%%%%%%%%%%%%%%%%%%
\item
We denote by $\hat{\mathcal E}_{1/p}(U_+)$
the set of functions $f\in C_{1/p}^{\infty}(U_+)$ such that 
$F_f$ belongs to 
the class $\hat{\mathcal E}(V)$.
%%%%%%%%%%%%%%%%%%%%%%%%%%
\item
Let $f\in C_{1/p}^{\infty}(U_+)$.
When we write the Taylor series of $F_f$ at the origin as 
$F_f(x)\sim\sum_{\alpha\in\Z_+^n}c_{\alpha}x^{\alpha}$,
the formal series:
\begin{equation}
f(x)\sim \sum_{\alpha\in\Z_+^n}c_{\alpha}x^{\alpha/p}:=
\sum_{\alpha\in\Z_+^n}c_{\alpha}x_1^{\alpha_1/p_1}\cdots x_n^{\alpha_n/p_n}
\label{eqn:10.33}
\end{equation}
is called the {\it Puiseux series of} $f$ at the origin.
%%%%%%%%%%%%%%%%%%%%%%%%%%
\item
For $f\in C_{1/p}^{\infty}(U_+)$, 
we can naturally generalize the definitions of 
{\it the Newton polyhedron, the Newton distance and its multiplicity, 
the principal face} in the case of the Puiseux series (\ref{eqn:10.33}), 
which are samely denoted by 
$\Gamma_+(f),d(f),m(f)$, $\tau_*$, respectively.
Recall that 
$d(f),m(f)$, $\tau_*$ have been defined in 
Remark~\ref{rem:2.11} in the case when $f\in C^{\infty}(U)$.
For a compact face $\gamma$ of $\Gamma_+(f)$, 
the $\gamma$-part $f_{\gamma}$ of $f$ is
defined by 
$f_{\gamma}(x):=\sum_{\alpha/p\in\gamma} c_{\alpha}x^{\alpha/p}$.
Note that the vertices of the Newton polyhedron $\Gamma_+(f)$ are 
contained in $\Q_+^n$.
(Refer to \cite{cgp13} for the details.) 
%%%%%%%%%%%%%%%%%%%%%%%%%%%
\item
$f\in C_{1/p}^{\infty}(U_+)$ is said to {\it be nondegenerate
over $\R$ with respect to} $\Gamma_+(f)$
if the $\gamma$-part $f_{\gamma}$ satisfies
$\nabla f_{\gamma}\neq (0,\ldots,0)$ on
$\R_{>0}^n$ for every compact face 
$\gamma$ of $\Gamma_+(f)$.
Note that $f\in C_{1/p}^{\infty}(U_+)$ is smooth in 
$U_+ \cap \R_{>0}^n$.
%%%%%%%%%%%%%%%%%%%%%%%%%%
\end{enumerate}
\end{definition}
%%%%%%%%%%%%%%%%%%%%%%%%%%%%%%

We will use the following properties later.
\begin{lemma} Let $f\in C_{1/p}^{\infty}(U_+)$. Then we have 
\begin{enumerate}
\item $f$ is nondegenerate
over $\R$ with respect to $\Gamma_+(f)$ if and only if
$F_f$ is nondegenerate
over $\R$ with respect to $\Gamma_+(F_f)$. 
\item 
%\begin{equation}\label{eqn:}
$d(F_f,x^{p-\1})=d(f)$, where 
$x^{p-\1}=\prod_{j=1}^n x_j^{p_j-1}$.
%\end{equation}
\end{enumerate}
\end{lemma}
%%%%%%%%%%%%%%%%%%%%%%%%%%%%%%%%%%
\begin{proof}
(i) By using the chain rule, this equivalence can be easily seen.

(ii) 
From the definitions of the Newton distances $d(\cdot,\cdot)$ and
$d(\cdot)$, the geometrical relationship between 
$\Gamma_+(F_f)$ and $\Gamma_+(f)$ gives the desired equation. 
\end{proof}
%%%%%%%%%%%%%%%%%%%%%%%%%%%%%%%%%%

Now, let us consider the following local zeta type functions 
in the case when $f$ belongs to the class
$C_{1/p}^{\infty}(U_+)$.
%%%%%%%%%%%%%%%%%%%%%%%%%%%%%%%%%%%%%%%
\begin{equation}\label{eqn:10.34}
\tilde{Z}_{\pm}(s;\varphi)=\int_{\R_+^n} f(x)_{\pm}^s \varphi(x)dx, \quad 
\tilde{Z}(s;\varphi)=\int_{\R_+^n} |f(x)|^s \varphi(x)dx, 
\end{equation}
%%%%%%%%%%%%%%%%%%%%%%%%%%%%%%%%%%%%%%%
where $\varphi$ satisfies the condition (C) in Section~3.

By applying Theorems 10.1 and 10.8,
we obtain the following theorems, 
whose assertions are almost the same as those in the results
in \cite{var76}, \cite{kn13} 
in appearance.

%%%%%%%%%%%%%%%%%%%%%%%%%%%%%%%%%%%%%%%%%%%%%%%%
\begin{theorem}\label{thm:10.12}
Suppose that $f\in\hat{\mathcal E}_{1/p}(U_+)$ 
is nondegenerate over $\R$ with respect to $\Gamma_+(f)$.
If the support of $\varphi$ is
contained in a sufficiently small neighborhood 
of the origin, 
then 
the functions  
$\tilde{Z}_{\pm}(s;\varphi)$ and $\tilde{Z}(s;\varphi)$ 
can be analytically continued
as meromorphic functions to the whole complex plane.
More precisely, we have the following.
%%%%%%%%%%%%%%%%
\begin{itemize}
\item[(a)]
The poles of the functions 
$\tilde{Z}_{\pm}(s;\varphi)$ and $\tilde{Z}(s;\varphi)$
are contained in the set 
%%%%%
\begin{equation}\label{eqn:10.35}
\left\{
-\frac{\langle a,p\rangle +\nu}{l_{F_f}(a)}
:\,\, \nu\in\Z_+,\,\, a\in\tilde{\Sigma}^{(1)}
\right\}
\cup (-\N).
\end{equation}
%%%%%
Here $l_{F_f}(a)$ is as in $(\ref{eqn:7.1})$,
$\Sigma$ is a simplicial subdivision of the fan
associated with the Newton polyhedron 
$\Gamma_+(F_f)$, 
$\Sigma^{(1)}$ is the set of 
one-dimensional cones in $\Sigma$ and 
$\tilde{\Sigma}^{(1)}:=\{a\in\Sigma^{(1)}:l_{F_f}(a)\neq 0\};$
\item[(b)]
The largest element of the first set in
$(\ref{eqn:10.35})$
is $-1/d(f);$
\item[(c)]
When
$\tilde{Z}_{\pm}(s;\varphi)$ and $\tilde{Z}(s;\varphi)$
have poles at $s=-1/d(f)$,
their orders are at most
%%%%%
\begin{eqnarray}\label{eqn:10.2}
%%%%%
\begin{cases}
m(f)& \quad \mbox{if $1/d(f)$ 
is not an integer}, \\
\min\{m(f)+1, n\}&
\quad \mbox{otherwise}.
\end{cases}
%%%%%
\end{eqnarray}
%%%%%
\end{itemize}
%%%%%%%%%%%%%%%%%
\end{theorem}

%%%%%%%%%%%%%%%%%%%%%%%%%%%%%
%%%%%%%%%%%%%%%%%%%%%%%%%%%%%
\begin{remark}
It is natural to expect that
the set  in the assertion (a)
should be replaced by
the following one:
%%%
\begin{equation}\label{eqn:10.37}
\left\{
-\frac{\langle a\rangle +\nu}{l_{f}(a)}
:\,\, \nu\in\Z_+,\,\, a\in\tilde{\Sigma}^{(1)}
\right\}
\cup (-\N).
\end{equation}
%%%%%
Here $l_f(a)$ is as in $(\ref{eqn:7.1})$,
$\Sigma$ is a simplicial subdivision of the fan
associated with the Newton polyhedron 
$\Gamma_+(f)$, 
$\Sigma^{(1)}$ is the set of 
one-dimensional cones in $\Sigma$ and 
$\tilde{\Sigma}^{(1)}:=\{a\in\Sigma^{(1)}:l_f(a)\neq 0\}$.
%%%%%%%%%%%
But we have not yet verified this. 
\end{remark}
%%%%%%%%%%%%%%%%%
%%%%%%%%%%%%%%%%%
\begin{theorem}\label{thm:10.14}
Suppose that $f\in\hat{\mathcal E}_{1/p}(U_+)$ 
is nondegenerate on $\R$ with respect to $\Gamma_+(f)$
and that
at least one of the following
conditions is satisfied:
%%%%%
\begin{enumerate}
\item[(a)] $d(f)>1;$
\item[(b)] $f$ is nonnegative or nonpositive on $U_+;$
\item[(c)] $f_{\tau_*}$ does not vanish on 
$\R_{>0}^n$, 
where $\tau_*$ is the principal face of $\Gamma_+(f)$.
\end{enumerate}
%%%%%
Then if 
$\varphi(0)$ is positive $($resp. negative$)$
and $\varphi$ is nonnegative $($resp. nonpositive$)$ on $U$ 
and the support of $\varphi$ is sufficiently small, 
then 
$\tilde{C}_{\pm}$ are nonnegative 
$($resp. nonpositive$)$ and 
$\tilde{C}=\tilde{C}_{+}+\tilde{C}_{-}$ is positive
$($resp. negative$)$,
where
$$ 
\tilde{C}_{\pm}:=\lim_{s\to-1/d(f)} 
(s+1/d(f))^{m(f)} 
\tilde{Z}_{\pm}(s;\varphi), 
\quad
\tilde{C}:=
\lim_{s\to-1/d(f)} 
(s+1/d(f))^{m(f)} 
\tilde{Z}(s;\varphi).
$$
\end{theorem}
%%%%%%%%%%%%%%%%%%%%%%%%%%%%%%%%%%%%%%%%%%%%%%%%%

\begin{proof}[Proof of Theorems 10.12 and 10.14]
We only show the case of $\tilde{Z}_{\pm}(s;\varphi)$. 
Substituting 
$x=\Phi^{-1}_{1/p}(y)=(y_1^{p_1},\ldots,y_n^{p_n})$ 
into the integrals 
in (\ref{eqn:10.34}),
we have
\begin{equation}\label{eqn:10.38}
\tilde{Z}_{\pm}(s;\varphi)
=\int_{\R_+^n} f(x)_{\pm}^s \varphi(x)dx
=\left(\prod_{j=1}^n p_j\right)
\int_{\R_+^n} F_f(y)_{\pm}^s \varphi(\Phi_{1/p}^{-1}(y))y^{p-\1}dy,
\end{equation}
where $y^{p-\1}=\prod_{j=1}^n y_j^{p_j-1}$.
Noticing the nondegeneracy condition for $F_f$ from
Lemma~10.11 (i) and 
the equation $d(F_f,x^{p-\1})=d(f)$ in Lemma~10.11 (ii), 
we can apply Theorems~10.1 and 10.8 to the 
integrals in (\ref{eqn:10.38}). 
As a result, we can obtain the above theorems.
%%%%%%%%%%%%%%%%%%%%%%%%
\end{proof}

%%%%%%%%%%%%%%%%%%%%%%%%%%%%%%%%%%%%%%%%%%%%%%%%%%%%%%%
\section{The case of $\hat{\mathcal E}$-weight}\label{sec:11}

%%%%%%%%%%%%%%%%%%%%%%%%%%%%%%%%%%%%%%%%%%%%%%%%%%%%%%%%%%
In this section,
we more generally investigate the poles of 
$Z_{\pm}(s;\varphi)$ and $Z(s;\varphi)$
in the case when the weight $g$ belongs to the class 
$\E(U)$, 
where $U$ is an open neighborhood of the origin. 
In order to investigate this case, 
we must use a fan $\Sigma$ constructed from the two Newton polyhedra
$\Gamma_+(f)$ and $\Gamma_+(g)$. 
%We use a stronger resolution of singularities.

In this section, 
we use the following notation.
%%%%%
\begin{itemize}
\item
$\Sigma_0$ is the fan associated with $\Gamma_+(f)$
\underline{and $\Gamma_+(g)$};
\item
$\Sigma$ is a simplicial subdivision of $\Sigma_0$;
\item
$(Y_{\Sigma},\pi)$ is the real resolution 
associated with $\Sigma$;
\item
$a^1(\sigma),\ldots,a^n(\sigma)$ is the skeleton of
$\sigma\in\Sigma^{(n)}$, ordered once and for all;
\item
$J_{\pi}(y)$ is the Jacobian of the mapping of $\pi$.
\end{itemize}
%%%%%

\subsection{Candidate poles}

First, let us state our results relating to
the positions and the orders of candidate poles of
$Z_{\pm}(s;\varphi)$ and $Z(s;\varphi)$
with $\hat{\mathcal{E}}$-weights.
%%%%%%%%%%%%%%%%%%%%%%%%%%%%%%%%%%%%%%%%%%%
\begin{theorem}\label{thm:11.1}
Suppose that {\rm (i)} $f$ satisfies the condition $(E)$
and
{\rm (ii)} $g$ belongs to the class $\hat{\mathcal{E}}(U)$.
If the support of $\varphi$ is
contained in a sufficiently small neighborhood 
of the origin, 
then 
the functions  
$Z_{\pm}(s;\varphi)$ and $Z(s;\varphi)$ 
can be analytically continued
as meromorphic functions to the whole complex plane.
More precisely, we have the following.
%%%%%%%%%%%%%%%%
\begin{enumerate}
\item[(a)]
The poles of the functions $Z_{\pm}(s;\varphi)$ and $Z(s;\varphi)$
are contained in the set 
%%%%%
\begin{equation}\label{eqn:11.1}
\left\{-\frac{l_g(a)+\langle a\rangle +\nu}{l_f(a)}
:\nu\in\Z_+,a\in\tilde{\Sigma}^{(1)}\right\}\cup(-\N),
\end{equation}
%%%%%
where $l_f(a)$ and $l_g(a)$ are as in $(\ref{eqn:7.1})$
and $\tilde{\Sigma}^{(1)}=\{a\in\Sigma^{(1)}:l_f(a)\neq 0\}$$;$
%%%%%%%%%%%%%%%
\item[(b)]
The largest element of the first set of $(\ref{eqn:11.1})$ is
$-1/d(f,g)$$;$
\item[(c)]
When $Z_{\pm}(s;\varphi)$ and $Z(s;\varphi)$
have poles at $s=-1/d(f,g)$,
their orders are at most
%%%%%
\begin{equation*}
\begin{cases}
m(f,g)& \mbox{if $-1/d(f,g)$ is not an integer,}\\
\min\{m(f,g)+1,n\} & \mbox{otherwise.}
\end{cases}
\end{equation*}
%%%%%
\end{enumerate}
%%%%%%%%%%%%%%%%%
\end{theorem}
%%%%%%%%%%%%%%%%%%%%%%%%%%%%%%%%%%%%%%%%%%%%

\begin{proof}
%%%%%%%%%%%%%%%%%%%%%%%%%%%%%%%%%%%%%%%%%%%%
We only consider the case of $Z_{\pm}(s;\varphi)$. 
By applying the orthant decomposition, 
it suffices to investigate the functions:
\begin{equation*}\label{eqn:11.}
\tilde{Z}_{\pm}(s;\varphi)
(=\tilde{Z}_{\pm}(s;\varphi;f,g))=
\int_{\R_+^n}f(x)_{\pm}^s g(x)\varphi(x)dx. 
\end{equation*}
From Theorem~\ref{thm:7.2}, there exist smooth functions
$f_{\sigma}$ and $g_{\sigma}$ such that 
$f_{\sigma}(0)\cdot g_{\sigma}(0)\neq 0$ and 
%%%%%%%%%%%%%%%%%
\begin{equation}\label{eqn:11.2}
\begin{split}
&
f(\pi(\sigma)(y))
=
\left(
\prod_{j=1}^n
y_j^{l_f(a^j(\sigma))}
\right)
f_{\sigma}(y),\\
&
g(\pi(\sigma)(y))
=
\left(
\prod_{j=1}^n
y_j^{l_g(a^j(\sigma))}
\right)
g_{\sigma}(y).
\end{split}
\end{equation}
%%%%%%%%%%%%%%%%%
Using the mapping $x=\pi(y)$ and
the cut-off functions 
$\{\chi_{\sigma};\sigma\in \Sigma^{(n)}\}$
in the proof of Theorem~\ref{thm:10.1}
and substituting the above equations in 
(\ref{eqn:11.2}), 
we have
%%%%%
\begin{equation}\label{eqn:11.3}
\begin{split}
&\tilde{Z}_{\pm}(s;\varphi)
=\int_{\tilde{Y}_{\Sigma}} 
((f\circ\pi)(y))_{\pm}^s (g\circ\pi)(y) (\varphi\circ\pi)(y) 
|J_{\pi}(y)|dy\\
&
\quad\quad\quad
=\sum_{\sigma\in\Sigma^{(n)}} Z_{\pm}^{(\sigma)}(s), 
\end{split}
\end{equation}
%%%%%
%where $J_{\pi}(y)$ is the Jacobian of the mapping of $\pi$
%and  $dy$ is a volume element in $Y_{\Sigma}$. 
with
%%%%%
\begin{equation}\label{eqn:11.4}
\begin{split}
&Z_{\pm}^{(\sigma)}(s)=
\int_{\R_+^n} 
((f\circ\pi(\sigma))(y))_{\pm}^s 
(g\circ\pi(\sigma))(y) (\varphi\circ\pi(\sigma))(y) 
\chi_{\sigma}(y)
|J_{\pi(\sigma)}(y)|dy\\
&\quad
=\int_{\R_+^n}\left(
\prod_{j=1}^n y_j^{l_f(a^j(\sigma))} 
f_{\sigma}(y)\right)_{\pm}^s
\left(
\prod_{j=1}^n y_j^{l_g(a^j(\sigma))} 
g_{\sigma}(y)\right)
\left|\prod_{j=1}^n y_j^{\langle a^j(\sigma)\rangle -1}\right|
\tilde{\chi}_{\sigma}(y)dy\\
&\quad
=\int_{\R_+^n}\left(
\prod_{j=1}^n 
y_j^{l_f(a^j(\sigma))s+l_g(a^j(\sigma))+\langle a^j(\sigma)\rangle -1} 
\right)
f_{\sigma}(y)_{\pm}^s 
g_{\sigma}(y)
\tilde{\chi}_{\sigma}(y)dy,
\end{split}
\end{equation}
%%%%%
where $\tilde{\chi}_{\sigma}(y)=
(\varphi\circ\pi(\sigma))(y) \chi_{\sigma}(y)$.

By a similar argument in the proof of Theorem~\ref{thm:10.1},
we see that the poles of $Z_{\pm}^{(\sigma)}(s)$
are contained in the set
%%%%%
\begin{equation}\label{eqn:11.5}
\begin{split}
\left\{
-\frac{l_g(a^j(\sigma))+
\langle a^j(\sigma)\rangle +\nu}{l_f(a^j(\sigma))}:
\nu\in\Z_+, j\in B({\sigma})
\right\}
\cup(-\N),
\end{split}
\end{equation}
%%%%%
where
$B(\sigma)=\{j:l_f(a^j(\sigma))\neq 0\}$. 
%is as in (\ref{eqn:10.12}).
From (\ref{eqn:11.5}), 
we can see (a) in the theorem. 

Next, consider a geometrical meaning of 
the largest element of the set (\ref{eqn:11.1}). 
The following lemma implies the assertion (b) in the theorem. 
%%%%%
\begin{lemma}
$$
\displaystyle
\max\left\{
-\dfrac{l_g(a)+\langle a\rangle }{l_f(a)}: 
a\in\tilde{\Sigma}^{(1)} 
\right\}
=-\frac{1}{d(f,g)}.
$$
\end{lemma}
%%%%%%%%%%%%%%%%%%%%%%%%%%%%%%
\begin{proof}
Using the equations in Remark~2.1, 
we see the following equivalences. 
\begin{equation}\label{eqn:11.6}
\begin{split}
&\quad\quad
\Phi(\Gamma_+(g))\subset\Gamma_+(f)\\
&\Longleftrightarrow
d(f,g)\cdot(\Gamma_+(g)+\1)\subset\Gamma_+(f)\\
&\Longleftrightarrow
d(f,g)\cdot(H^+(a,l_g(a))+\1)\subset H^+(a,l_f(a)) 
\quad\quad \mbox{for any $a\in\Z_+^n$} \\
&\Longleftrightarrow
H^+(a,d(f,g)(l_g(a)+\langle a\rangle ))\subset H^+(a,l_f(a)) 
\quad\quad \mbox{for any $a\in\Z_+^n$} \\
&\Longleftrightarrow
d(f,g)(l_g(a)+\langle a\rangle ) \geq l_f(a)
\quad\quad \mbox{for any $a\in\Z_+^n$} \\
&\Longleftrightarrow
-\frac{l_g(a)+\langle a\rangle }{l_f(a)}
\leq -\frac{1}{d(f,g)}
\quad\quad \mbox{for any $a\in\Z_+^n$.} 
\end{split}
\end{equation}
%%%%%
Moreover, from the construction of 
the fan $\Sigma$,  we see  
the existence of $a\in\Sigma^{(1)}$ satisfying the 
equality in the last inequality.
Therefore, we have the equation in the lemma. 
\end{proof}
%%%%%%%%%%%%%%%%%%%%%%%%%%%%%%%%%
Finally, let us consider the orders of the poles of 
$\tilde{Z}_{\pm}(s;\varphi)$
at $s=-1/d(f,g)$.
For $\sigma\in\Sigma^{(n)}$, let  
%%%%%
\begin{equation}\label{eqn:11.7}
A(\sigma):=\left\{
j\in B(\sigma):
-\frac{1}{d(f,g)}=
-\frac{l_g(a^j(\sigma))+\langle a^j(\sigma)\rangle }{l_f(a^j(\sigma))} 
\right\}\subset\{1,\ldots,n\}.
\end{equation}
%%%%%%%%%%%%%%%%%%%%%%%
%%%%%%%%%%%%%%%%%%%%%%%
In order to prove the assertion (c) in the theorem, 
it suffices to show the estimate 
$\# A(\sigma)\leq m(f,g)$, 
which will be shown in Lemma~\ref{lem:11.4} (ii) 
in the next subsection. 

%$m(f,g)=\max\{\# A(\sigma);\sigma\in \Sigma^{(n)}\}$.
%%%%%%%%%%%%%%%%%%%%%%%%%%%%%%%%%%%%%%%%%%%%%%%%
\end{proof}
%%%%%%%%%%%%%%%%%%%%%%%%
\begin{remark}\label{rem:11.3}
Carefully observing the equivalences in (\ref{eqn:11.6}), 
we see that for $\sigma\in\Sigma^{(n)}$, 
\begin{equation}\label{eqn:11.8}
%\begin{split}
j\in A(\sigma)\Longleftrightarrow
\Phi(H(a^j(\sigma),l_g(a^j(\sigma))))
= H(a^j(\sigma),l_f(a^j(\sigma))).
%\\  
%&\quad\quad
%\Longleftrightarrow
%\Phi(H(a^j(\sigma),l_g(a^j(\sigma))))
%\cap H(a^j(\sigma),l_f(a^j(\sigma)))\neq \emptyset.
%\end{split}
\end{equation}
\label{eqn:}
\end{remark}
%%%%%%%%%%%%%%%%%%%%%%%%

%%%%%%%%%%%%%%%%%%%%%%%%
%%%%%%%%%%%%%%%%%%%%%%%%
\subsection{Properties of the principal faces}\label{subsec:11.2}
%%%%%%%%%%%%%%%%%%%%%%%%
In order to investigate more precise 
properties of the leading poles 
of $Z_{\pm}(s;\varphi)$, 
we must understand more exact relationship
between the cones of
the fan $\Sigma$ and the faces of
the Newton polyhedra $\Gamma_+(f)$ and $\Gamma_+(g)$.
In this subsection, 
after our interest focuses essentially important cones 
and faces,  
their properties and relationship are investigated in detail. 

Now, let us define a class of %important cones and 
important faces of $\Gamma_+(f)$, $\Gamma_+(g)$. 
%Let $\Sigma_*^{(n)}$ be a subset in $\Sigma^{(n)}$ defined by 
%%%%%%%%%%%%%%%%%%%%%%%%%
%\begin{equation}\label{eqn:11.9}
%\Sigma_*^{(n)}:=\left\{
%\sigma\in\Sigma^{(n)};
%\#A(\sigma)=m(f,g) 
%\right\}.
%\end{equation}
%%%%%%%%%%%%%%%%%%%%%%%%
%%%%%%%%%%%%%%%%%%%%%%%%%
%%%%%%%%%%%%%%%%%%%%%%%%
For a cone $\sigma\in\Sigma^{(n)}$ with $A(\sigma)\neq\emptyset$,
let $\tau_*(\sigma)$ (resp. $\gamma_*(\sigma)$) be the face of 
$\Gamma_+(f)$ (resp. $\Gamma_+(g)$) defined by 
\begin{eqnarray}
&&
\tau_*(\sigma):=\bigcap_{j\in A(\sigma)}
H(a^j(\sigma),l_f(a^j(\sigma)))\cap \Gamma_+(f),
\label{eqn:11.12}\\
&&
\gamma_*(\sigma):=\bigcap_{j\in A(\sigma)}
H(a^j(\sigma),l_g(a^j(\sigma)))\cap \Gamma_+(g).
\label{eqn:11.13}
\end{eqnarray}
When $A(\sigma)=\emptyset$, we define 
$\tau_*(\sigma)=\gamma_*(\sigma):=\emptyset$.
From the construction 
of the fan $\Sigma$, 
it is easy to see that 
$\tau_*(\sigma)$ (resp. $\gamma_*(\sigma)$)
is a nonempty face of $\Gamma_+(f)$ 
(resp. $\Gamma_+(g)$) if $A(\sigma)\neq\emptyset$.
%%%%%%%%%%%%%%%%%%%%%%%
\begin{remark}\label{rem:11.5}
In terms of the symbol (\ref{eqn:6.4}) in Section~\ref{sec:6},
the above equations can be written as
$\tau_*(\sigma)$(resp. $\gamma_*(\sigma)$)
$=\gamma(A(\sigma),\sigma)$
with $P=\Gamma_+(f)$ (resp. $P=\Gamma_+(g)$). 
\end{remark}
%%%%%%%%%%%%%%%%%%%%%%%

%%%%%%%%%%%%%%%%%%%%%%%
\begin{lemma}\label{lem:11.4}
Suppose that $\tau_*$ is a principal face of $\Gamma_+(f)$.
Then we have the following 
$($see $(\ref{eqn:6.5}), (\ref{eqn:6.6})$ 
for the definitions of $I(\cdot,\cdot)$ and 
$\Sigma^{(n)}(\cdot)$$)$. 
\begin{enumerate}
\item 
$I(\tau_*,\sigma)\subset A(\sigma)$ 
for any $\sigma\in\Sigma^{(n)}$.
%%%%%%%%%%%
\item
$\# A(\sigma)\leq m(f,g)$ for any $\sigma\in\Sigma^{(n)}$. 
%%%%%%%%%%%
\item 
$\Sigma^{(n)}(\tau_*)\subset\Sigma_*^{(n)}:=
\left\{
\sigma\in\Sigma^{(n)}:
\#A(\sigma)=m(f,g) 
\right\}$. 
%%%%%%%%%%%
\item
$\Sigma_*^{(n)}\neq\emptyset$.
\end{enumerate}
\end{lemma}
%%%%%%%%%%%%%%%%%%%%%%%%
\begin{proof}
(i) \quad 
Suppose that 
$j\in I(\tau_*,\sigma)$, i.e., 
$\tau_*\subset H(a^j(\sigma),l_f(a^j(\sigma)))$.
Let $\gamma_*$ be the principal face of 
$\Gamma_+(g)$ associated to $\tau_*$, 
i.e. $\Psi_*(\tau_*)=\gamma_*$. 
Since $\Phi(\gamma_*)\subset\Psi_*^{-1}(\gamma_*)=\tau_*$,
we have
\begin{equation}\label{eqn:11.10}
\begin{split}
&\tau_*\subset H(a^j(\sigma),l_f(a^j(\sigma)))
\Longrightarrow
\Phi(\gamma_*)\subset H(a^j(\sigma),l_f(a^j(\sigma)))\\
&\quad\quad\quad\quad\quad\quad\quad\quad\quad
\quad\Longleftrightarrow
\gamma_*\subset \Phi^{-1}(H(a^j(\sigma),l_f(a^j(\sigma))))
\end{split}
\end{equation} 
%%%
From an easy property of the map $\Phi$,
there exists a positive number $l$ such that 
$\Phi^{-1}(H(a^j(\sigma),l_f(a^j(\sigma))))=
H(a^j(\sigma),l)$.
Since 
$\Gamma_+(g)\subset
\Phi^{-1}(H^+(a^j(\sigma),l_f(a^j(\sigma))))$ 
and
$\gamma_*$ is a nonempty proper face of 
$\Gamma_+(g)$, 
the definition of $l_g(\cdot)$ 
implies $l=l_g(a^j(\sigma))$.
Thus, we have
$$ 
H(a^j(\sigma),l_g(a^j(\sigma)))
=\Phi^{-1}(H(a^j(\sigma),l_f(a^j(\sigma)))),
$$
which implies that $j\in A(\sigma)$ 
from the equivalence (\ref{eqn:11.8})
in Remark~\ref{rem:11.3}. 
%%%%%%%%%%%%%%%%%

(ii)\quad
Since the case when $A(\sigma)=\emptyset$ is obvious, 
we assume that $A(\sigma)\neq\emptyset$.
From the equivalence (\ref{eqn:11.8}) in Remark~\ref{rem:11.3}, 
we have 
\begin{eqnarray}
&&\Phi
\left(
\bigcap_{j\in A(\sigma)}
H(a^j(\sigma),l_g(a^j(\sigma)))
\right)
\cap \Gamma_+(f)
=
\bigcap_{j\in A(\sigma)}
H(a^j(\sigma),l_f(a^j(\sigma)))
\cap \Gamma_+(f) \nonumber\\
&&\quad\quad\Longleftrightarrow
\Phi
\left(
\bigcap_{j\in A(\sigma)}
H(a^j(\sigma),l_g(a^j(\sigma)))
\cap \Phi^{-1}(\Gamma_+(f))
\right)
=
\tau_*(\sigma).
\label{eqn:11.14}
\end{eqnarray}
Since $\Phi^{-1}(\Gamma_+(f))\supset \Gamma_+(g)$, 
the left side of (\ref{eqn:11.14}) contains 
$\Phi(\gamma_*(\sigma))$, i.e. 
$\Phi(\gamma_*(\sigma))\subset\tau_*(\sigma)$.  
%%%%
Since $\Phi(\gamma_*(\sigma))$ is a nonempty set in
$\Gamma_0(f)$ and is contained in the face
$\tau_*(\sigma)$ of $\Gamma_+(f)$, 
the definition of ${\mathcal F}_0[\Gamma_+(f)]$
implies that
there exists a face 
$\tilde{\tau}\in{\mathcal F}_0[\Gamma_+(f)]$
such that
$\Phi(\gamma_*(\sigma))\subset
\tilde{\tau}\subset
\tau_*(\sigma)$. 
From the definition of $m(f,g)$, we have
\begin{equation}\label{eqn:11.15}
\dim(\tau_*(\sigma))\geq\dim(\tilde{\tau})\geq 
n-m(f,g).
\end{equation}
%%%%
On the other hand, we have
\begin{equation}\label{eqn:11.16}
\dim(\tau_*(\sigma))
\leq
\dim
\left(
\bigcap_{j\in A(\sigma)}
H(a^j(\sigma),l_f(a^j(\sigma)))
\right)
=
n-\# A(\sigma).
\end{equation}
%%%%
Putting (\ref{eqn:11.15}), (\ref{eqn:11.16}) together, we have
$\#A(\sigma)\leq m(f,g)$.
%%%%%%%%%%%%%%%%%%%%%%%%%%%%%%

(iii)\quad
From the definition (\ref{eqn:6.6}), we have
\begin{equation}\label{eqn:11.11}
%\begin{split}
\Sigma^{(n)}(\tau_*)=
%\{\sigma\in\Sigma^{(n)}:
%\dim(\tau_*)=n-\#I(\tau_*,\sigma)\}&=
\{\sigma\in\Sigma^{(n)}:
\#I(\tau_*,\sigma)=m(f,g)\}.
%\end{split}
\end{equation}
It follows from (i), (ii) in this lemma
that $\Sigma^{(n)}(\tau_*)\subset\Sigma_*^{(n)}$.
%%%%%%%%%%%%%%%%%%%%%%%%%%%%%%%%%%

(iv)\quad 
Since $\Sigma^{(n)}(\tau_*)\neq\emptyset$ 
from Lemma~\ref{lem:6.2} (ii),
we see  $\Sigma_*^{(n)}\neq\emptyset$. 
\end{proof}
%%%%%%%%%%%%%%%%%%%%%%%%
\begin{remark}\label{rem:11.7}
From Lemma~\ref{lem:11.4} (ii), (iv), 
we see 
$\max\{\# A(\sigma):\sigma\in \Sigma^{(n)}\}=m(f,g)$,
which is a generalization of Lemma~\ref{lem:10.5}.
\end{remark}

%%%%%%%%%%%%%%%%%%%%%%%%%%%%%%%%%%%
\begin{proposition}\label{pro:11.6}
If $\sigma\in\Sigma_*^{(n)}$, then
$\tau_*(\sigma)$ $($resp. $\gamma_*(\sigma)$$)$
is a principal face of $\Gamma_+(f)$
$($resp. $\Gamma_+(g)$$)$. 
Moreover,  
$\gamma_*(\sigma)$ is associated to 
$\tau_*(\sigma)$ 
$($i.e. $\gamma_*(\sigma)=\Psi_*(\tau_*(\sigma))$$)$.
\end{proposition}
%%%%%%%%%%%%%%%%%%%%%%%%%%%%%%%%%%%

%%%%%%%%%%%%%%%%%%%%%%%%
\begin{proof}
Suppose that $\sigma\in\Sigma_*^{(n)}$, i.e., 
$\# A(\sigma)=m(f,g)$. 
From (\ref{eqn:11.15}), (\ref{eqn:11.16}),
we have $\dim(\tau_*(\sigma))=n-m(f,g)$ and,
moreover, 
$\tau_*(\sigma)=\tilde{\tau}\in{\mathcal F}_0[\Gamma_+(f)]$. 
It follows from these equations that 
$\tau_*(\sigma)$ is a principal face of $\Gamma_+(f)$.
Moreover, 
%since $\Gamma_+(g)\subset\Phi^{-1}(\Gamma_+(f))$, 
the equation (\ref{eqn:11.14}) implies  
\begin{eqnarray*}
&&\quad\quad
\Phi^{-1}\left(
\tau_*(\sigma)
\right)=
\bigcap_{j\in A(\sigma)}
H(a^j(\sigma),l_g(a^j(\sigma)))
\cap \Phi^{-1}(\Gamma_+(f))
\\
&&\Longrightarrow
\Phi^{-1}\left(
\tau_*(\sigma)
\right)\cap\Gamma_+(g)=
\bigcap_{j\in A(\sigma)}
H(a^j(\sigma),l_g(a^j(\sigma)))
\cap \Phi^{-1}(\Gamma_+(f))
\cap \Gamma_+(g) \\
&&\quad\quad\quad\quad\quad\quad
=\bigcap_{j\in A(\sigma)}
H(a^j(\sigma),l_g(a^j(\sigma)))
\cap \Gamma_+(g)=\gamma_*(\sigma).
\label{eqn:}
\end{eqnarray*}
From the above equations and the definition 
(\ref{eqn:2.11}), 
we see that 
$\gamma_*(\sigma)$ is a principal face of 
$\Gamma_+(g)$ associated to $\tau_*(\sigma)$.
\end{proof}
%%%%%%%%%%%%%%%%%%%%%%%

From Proposition~\ref{pro:11.6}, 
the map from 
$\Sigma_*^{(n)}$ to ${\mathcal F}_*[\Gamma_+(f)]$ 
(resp. to ${\mathcal F}_*[\Gamma_+(g)]$) 
is naturally defined
(i.e., $\sigma\mapsto \tau_*(\sigma)$ 
(resp. $\sigma\mapsto \gamma_*(\sigma)$)).
%%%%%%%%%%%%%%%%%%%%%%%%%%%%%%%%%
\begin{lemma}\label{lem:11.8}
The above maps are surjective. 
\end{lemma}
\begin{proof}
Let $\tau_*$ be an arbitrary principal face of 
$\Gamma_+(f)$ and $\sigma$ a cone in $\Sigma^{(n)}(\tau_*)$. 
Since $\Sigma^{(n)}(\tau_*)\subset\Sigma^{(n)}_*$ from
Lemma~\ref{lem:11.4} (iii), we have
$\#I(\tau_*,\sigma)=\#A(\sigma)=m(f,g)$. 
Moreover, since $I(\tau_*,\sigma)\subset A(\sigma)$ 
from Lemma~\ref{lem:11.4} (i), 
we have $A(\sigma)=I(\tau_*,\sigma)$. 
As a result, we have
\begin{equation*}
\begin{split}
&\tau_*=
\bigcap_{j\in I(\tau_*,\sigma)}
H(a^j(\sigma),l_f(a^j(\sigma)))
\cap \Gamma_+(f)\\
&\quad=
\bigcap_{j\in A(\sigma)}
H(a^j(\sigma),l_f(a^j(\sigma)))
\cap \Gamma_+(f)=  
\tau_*(\sigma).
\end{split}
\end{equation*}
Note that the first equality follows from Lemma~\ref{lem:6.2} (iii)
with Remark~\ref{rem:11.5}. 
Of course, if $\gamma_*$ is a principal face of $\Gamma_+(g)$
associated to $\tau_*$, 
then we see that $\gamma_*=\Psi_*(\tau_*)=\Psi_*(\tau_*(\sigma))
=\gamma_*(\sigma)$.

\end{proof}

%%%%%%%%%%%%%%%%%%%%%%%%%%%%%%%%%%%%%%%%%%%%%
\subsection{First coefficients}\label{subsec:11.3}

Let us return to the analysis of the functions $Z_{\pm}(s;\varphi)$.
First, 
we compute the coefficients of 
$(s+1/d(f,g))^{-m(f,g)}$ 
in the Laurent expansions of 
$\tilde{Z}_{\pm}(s;\varphi)$.
Respectively, we define 
%%%%%
\begin{equation}\label{eqn:11.17}
\begin{split}
\tilde{C}_{\pm}=\lim_{s\to -1/d(f,g)} 
(s+1/d(f,g))^{m(f,g)} 
\tilde{Z}_{\pm}(s;\varphi).
\end{split}
\end{equation}
%%%%%

%%%%%%%%%%%%%%%%%%%%%%%%%%%%%%%%%%%%%%%%%%%%%%%%
\begin{proposition}\label{pro:11.9}
Suppose that $({\rm i})$ $f$ satisfies the condition $(E)$,
$({\rm ii})$ $g$ belongs to the class $\hat{\mathcal E}(U)$
and 
$({\rm iii})$ at least one of the following
conditions is satisfied:
%%%%%
\begin{enumerate}
\item[(a)] $d(f,g)>1;$
\item[(b)] 
%$\gamma_*=\gamma_{\sigma}$ and 
$f_{\sigma}\circ T_{A(\sigma)}$ 
does not vanish on $\pi(\sigma)^{-1}(U)$ 
for any $\sigma\in\Sigma_*^{(n)}$.
\end{enumerate}
%%%%%
Then we give explicit formulae for coefficients:
$
\tilde{C}_{\pm}=
G_{\pm}(f,g,\varphi)
$, where 
$G_{\pm}(f,g,\varphi)$ are as in 
$(\ref{eqn:11.20}),(\ref{eqn:11.22}),(\ref{eqn:11.23}),(\ref{eqn:11.24})$
in the proof of this proposition. 
\end{proposition}
%%%%%%%%%%%%%%%%%%%%%%%%%%%%%%%%%%%%%%%%%%%%%%%%%

%%%%%%%%%%%%%%%%%%%%%%%%%%%%%%%%%%%%%%%%%%%%%%%
\begin{proof}
%%%%%%%%%%%%%
In this proof, 
we use the following notation and symbols 
to decrease the complexity in the integral expression. 
%%%%%
\begin{itemize}
\item 
$\prod_{j\not\in A(\sigma)}y_j^{a_j}dy_j$ means 
$\prod_{j\not\in A(\sigma)}y_j^{a_j}\cdot
\prod_{j\not\in A(\sigma)}dy_j$ 
with $a_j>0$; 
\item 
$L_{\sigma}:=\prod_{j\in A(\sigma)}l_f(a^j(\sigma))^{-1}$;
\item 
$M_j(\sigma)
:=-l_f(a^j(\sigma))/d(f,g)+
l_g(a^j(\sigma))+
\langle  a^j(\sigma) \rangle$.
\item 
If $a=0$, then the value of $a^{-1/d(f,g)}$ is defined by $0$.
\end{itemize}
%%%%%
Note that $M_j(\sigma)$ is a nonnegative constant and,   
moreover, $M_j(\sigma)= 0$ if and only if $j\in A(\sigma)$. 
%The nonnegativity of $g_{\gamma_*}$,Proposition \ref{pro:4.6},
%Lemma \ref{le:9.5} and \ref{le:11.3}
%gives that $l_g(a^j(\sigma))$ is even
%for $\sigma\in\Sigma_*^{(n)}$ and $j\in \{1,\ldots,n\}$.

Let us compute the limits $\tilde{C}_{\pm}$ exactly. 
We divide the computation into the following two cases:
$m(f,g)<n$ and $m(f,g)=n$. 

\underline{The case: $m(f,g)<n$.}\quad

\underline{Under the hypothesis (a).}\quad
First, we consider the case that the hypothesis (a):
$d(f,g)>1$ is satisfied.
Let us explicitly compute the following limits:
\begin{equation*}
\tilde{C}_{\pm}(\sigma):=
\lim_{s\to-1/d(f,g)} 
(s+1/d(f,g))^{m(f,g)} 
Z^{(\sigma)}_{\pm}(s),
\end{equation*} 
where $Z_{\pm}^{(\sigma)}(s)$ are 
as in (\ref{eqn:11.4}).
We only consider the case that
$$
\sigma\in\Sigma_*^{(n)}=\left\{
\sigma\in\Sigma^{(n)}:
\#A(\sigma)=m(f,g) 
\right\},
$$
since $\tilde{C}_{\pm}(\sigma)=0$
if $\sigma\not\in\Sigma_*^{(n)}$. 
We remark that $\Sigma_*^{(n)}$ is not empty from
Lemma~\ref{lem:11.4} (iv).
By applying the argument as in the proof of 
Theorem~\ref{thm:10.1} and Proposition~\ref{pro:10.6}
to $Z_{\pm}^{(\sigma)}(s)$
(compare with $Z_{\pm}^{(\sigma)}(s)$ 
defined in (\ref{eqn:10.7})),
we see that if $\sigma\in\Sigma_*^{(n)}$, then 
%%%%%%%%%%%
\begin{equation}\label{eqn:11.18}
\tilde{C}_{\pm}(\sigma)=
L_{\sigma}
\int_{\R_+^{n-m(f,g)}}
\frac{
g_{\sigma}(T_{A(\sigma)}(y))
\tilde{\chi}_{\sigma}(T_{A(\sigma)}(y))
}{
f_{\sigma}(T_{A(\sigma)}(y))_{\pm}^{1/d(f,g)}
}
\prod_{j\not\in A(\sigma)}
y_j^{M_j(\sigma)-1}dy_j.
\end{equation}
%%%%%
Furthermore, 
let us compute the limits $\tilde{C}_{\pm}$ explicitly. 
If the cut-off function $\chi_{\sigma}$ is deformed as 
the volume of the support of $\chi_{\sigma}$ tends to zero, 
then $\tilde{C}_{\pm}(\sigma)$ tends to zero. 
Notice that each $\R_+^n(\sigma)$ 
is densely embedded in $\tilde{Y}_{\Sigma}$ 
(see Section~\ref{subsec:5.2})
and that
$
\tilde{C}_{\pm}=
\sum_{\sigma\in\Sigma_*^{(n)}}\tilde{C}_{\pm}(\sigma).
$
Thus,  
for an arbitrary fixed cone $\sigma\in\Sigma_*^{(n)}$, 
we have 
%%%%%
\begin{equation}\label{eqn:11.19}
\tilde{C}_{\pm}=G_{\pm}(f,g,\varphi)
\end{equation}
with 
\begin{equation}\label{eqn:11.20}
\begin{split}
&G_{\pm}(f,g,\varphi)=\\
&L_{\sigma}
\int_{\R_+^{n-m(f,g)}}
\frac{
g_{\sigma}(T_{A(\sigma)}(y))
(\varphi\circ\pi(\sigma))(T_{A(\sigma)}(y))
}{
f_{\sigma}(T_{A(\sigma)}(y))_{\pm}^{1/d(f,g)}
}
\prod_{j\not\in A(\sigma)}
y_j^{M_j(\sigma)-1}dy_j.
\end{split}
\end{equation}

Let us give the other formulae of $G_{\pm}(f,g,\varphi)$, 
which are more directly expressed by $f,g,\varphi$. 
From Lemma~\ref{lem:7.1} with Remark~\ref{rem:11.5}, 
we obtain
%%%%%
\begin{equation}\label{eqn:11.21}
\begin{split}
&
(f_{\tau_*(\sigma)}\circ\pi(\sigma))(T_{A(\sigma)}^1(y))
=
\left(
\prod_{j\not\in A(\sigma)}
y_j^{l_f(a^j(\sigma))}
\right)
f_{\sigma}(T_{A(\sigma)}(y)),\\
&
(g_{\gamma_*(\sigma)}\circ\pi(\sigma))(T_{A(\sigma)}^1(y))
=
\left(
\prod_{j\not\in A(\sigma)}
y_j^{l_g(a^j(\sigma))}
\right)
g_{\sigma}(T_{A(\sigma)}(y)).
\end{split}
\end{equation}
%%%%%
By using the above equations, (\ref{eqn:11.20}) 
can be rewritten as 
%%%%%
\begin{equation}\label{eqn:11.22}
\begin{split}
&G_{\pm}(f,g,\varphi)=\\
&L_{\sigma}
\int_{\R_+^{n-m(f,g)}}
\frac{
(g_{\gamma_*(\sigma)}\circ \pi(\sigma))
(T^1_{A(\sigma)}(y))
(\varphi\circ \pi(\sigma))(T_{A(\sigma)}(y))
}
{
\left( 
(f_{\tau_*(\sigma)}\circ \pi(\sigma))
(T^1_{A(\sigma)}(y)) 
\right)_{\pm}^{1/d(f,g)}
}
\prod_{j\notin A(\sigma)}
y_j^{\langle a^j(\sigma)\rangle -1}dy_j.
\end{split}
\end{equation}
%%%%%
We remark that 
$G_{\pm}(f,g,\varphi)$ depend only
on $f_{\tau_*(\sigma)}$, $g_{\gamma_*(\sigma)}$, $\varphi$
and
that $G_{\pm}(f,g,\varphi)$ must be independent of 
the selection of $\sigma\in\Sigma^{(n)}_*$. 

\underline{Under the hypothesis (b).}\quad
In a similar fashion to the proof of 
Proposition~\ref{pro:10.6},
we can show the assertion in the case 
when the hypothesis (b) is satisfied. 

\underline{The case: $m(f,g)=n$.}\quad

In this case, we see that 
$A(\sigma)=\{1,\ldots,n\}$.
We remark that
$\tau_*(\sigma)$ and $\gamma_*(\sigma)$ are vertices 
of $\Gamma_+(f)$ and $\Gamma_+(g)$,
denoted by $q$ and $p$ respectively.
Similar computations give 
the following expression corresponding to (\ref{eqn:11.20}):
%%%%%
\begin{equation}\label{eqn:11.23}
G_{\pm}(f,g,\varphi)=
L
\frac{g_{\sigma}(0)\varphi(0)}
{f_{\sigma}(0)^{1/d(f,g)}_{\pm}},
\end{equation}
%%%%%
where $\sigma\in\Sigma_*^{(n)}$ and
$L=\sum_{\sigma\in\Sigma_*^{(n)}}L_{\sigma}$.
From the equation (\ref{eqn:11.21}),
we obtain another expression
corresponding to (\ref{eqn:11.22})
%%%%%

%******************
\begin{equation}\label{eqn:11.24}
G_{\pm}(f,g,\varphi)=
L
\frac{g_{\gamma_*(\sigma)}(\1)
\varphi(0)}
{f_{\tau_*(\sigma)}(\1)^{1/d(f,g)}_{\pm}}
=L
\frac{(q!)^{1/d(f,g)}}{p!}
\frac{(\d^{p}g)(0)\varphi(0)}
{(\d^{q}f)(0)^{1/d(f,g)}_{\pm}},
\end{equation}
%%%%%
where $\1=(1,\ldots,1)$ and
$L$ is the same as in (\ref{eqn:11.23}).
\end{proof}

\begin{remark}
Let us consider the case when $\tau_*(\sigma)$ or 
$\gamma_*(\sigma)$ is compact. 
Note that the compactness of $\tau_*(\sigma)$ 
is equivalent to that of $\gamma_*(\sigma)$ 
(see Remark~\ref{rem:2.10}).  
Lemma~\ref{lem:6.4} with Remark~\ref{rem:11.5} implies 
that $\pi(\sigma)(T_{A(\sigma)}(\R^n))=0$. 
Therefore, the formulae (\ref{eqn:11.20}),(\ref{eqn:11.22})
can be expressed in more simple forms as follows. 
\begin{equation*}%\label{eqn:}
\begin{split}
&G_{\pm}(f,g,\varphi)=
L_{\sigma}\varphi(0)
\int_{\R_+^{n-m(f,g)}}
\frac{
g_{\sigma}(T_{A(\sigma)}(y))
}{
f_{\sigma}(T_{A(\sigma)}(y))_{\pm}^{1/d(f,g)}
}
\prod_{j\not\in A(\sigma)}
y_j^{M_j(\sigma)-1}dy_j\\
&=L_{\sigma}\varphi(0)
\int_{\R_+^{n-m(f,g)}}
\frac{
(g_{\gamma_*(\sigma)}\circ \pi(\sigma))
(T^1_{A(\sigma)}(y))
}
{
\left( 
(f_{\tau_*(\sigma)}\circ \pi(\sigma))
(T^1_{A(\sigma)}(y)) 
\right)_{\pm}^{1/d(f,g)}
}
\prod_{j\notin A(\sigma)}
y_j^{\langle a^j(\sigma)\rangle -1}dy_j.
\end{split}
\end{equation*}
\end{remark}

Finally,
let us compute the coefficients of 
$(s+1/d(f,g))^{-m(f,g)}$ 
in the Laurent expansions of 
$Z_{\pm}(s;\varphi)$, $Z(s;\varphi)$.
Respectively, we define
%%%%%
\begin{equation}\label{eq:11.8}
\begin{split}
&C_{\pm}=\lim_{s\to -1/d(f,g)} 
(s+1/d(f,g))^{m(f,g)} 
Z_{\pm}(s;\varphi),\\
&C=\lim_{s\to -1/d(f,g)} 
(s+1/d(f,g))^{m(f,g)} 
Z(s;\varphi).
\end{split}
\end{equation}
%%%%%
%%%%%%%%%%%%%%%%%%%%%%%%%%%%%%%%%%%%%%%%%%%
%%%%%%%%%%%%%%%%%%%%%%%%%%%%%%%%%%%%%%%%%%%
%%%%%%%%%%%%%%%%%%%%%%%%%%%%%%%%%%%%%%%%%%%
\begin{theorem}\label{thm:11.10}
Suppose that $({\rm i})$ $f$ satisfies the condition $(E)$,
$({\rm ii})$ $g$ belongs to the class $\hat{\mathcal E}(U)$
and 
$({\rm iii})$ at least one of the following
conditions is satisfied: 
%%%%%
\begin{enumerate}
\item[(a)] $d(f,g)>1$$;$
\item[(b)] $f$ is nonnegative or nonpositive on $U$$;$
\item[(c)]
There exists a principal face $\tau_*$ of $\Gamma_+(f)$
such that 
$f_{\tau_*}$ does not vanish 
on $U\cap(\R\setminus \{0\})^n$.
\end{enumerate}
%%%%%
Then we obtain explicit formulae for coefficients in the following.
\begin{equation}\label{eqn:11.26}
C_{\pm}=
\sum_{\theta\in\{-1,1\}^n}
G_{\pm}(f_{\theta},g_{\theta},\varphi_{\theta})
\mbox{\quad and \quad} 
C=C_+ + C_-, 
\end{equation} 
where $f_{\theta}(x):=f(\theta_1 x_1,\ldots,\theta_n x_n)$ etc. and 
$G_{\pm}(f,g,\varphi)$ are as in 
$(\ref{eqn:11.20}),(\ref{eqn:11.22})$,
$(\ref{eqn:11.23}),(\ref{eqn:11.24})$
in the proof of Proposition~\ref{pro:11.9}. 

Furthermore,
the formulae $(\ref{eqn:11.26})$
imply the following. 
Suppose that
$({\rm iv})$ 
$g_{\gamma_*}$ is nonnegative $($resp. nonpositive$)$ on $U$,
%$U\cap(\R\setminus \{0\})^n$,
where 
$\gamma_*$ is a principal face of 
$\Gamma_+(g)$ 
$($associated to $\tau_*$ if {\rm (c)} is satisfied in {\rm (iii)}$)$, 
and 
$({\rm v})$
$\varphi(0)>0$
and 
$\varphi$ is nonnegative on $U$.
Then,
if the support of $\varphi$ is sufficiently small, 
then 
$C_{\pm}$ are nonnegative 
$($resp. nonpositive$)$ and 
$C=C_{+}+C_{-}$ is positive
$($resp. negative$)$. 
\end{theorem}
%%%%%%%%%%%%%%%%%%%%%%%%%%%%%%%%%%%%%%%%%%%
%%%%%%%%%%%%%%%%%%%%%%%%%%%%%%%%%%%%%%%%%%%
\begin{proof}
Remembering how to show Theorem~\ref{thm:10.7}  
by using Proposition~\ref{pro:10.6}, 
one can almost similarly show the above theorem
by using Proposition~\ref{pro:11.9}.
The following show non-trivial parts. 
\begin{itemize}
\item
The surjectivities 
in Lemma~\ref{lem:11.8} are used to find a cone
associated to principal faces which satisfy the conditions
in the theorem. 
\item
In order to see $C=C_++C_-\neq 0$ from the formula (\ref{eqn:11.22}), 
it must be shown that
$g_{\gamma_*}\circ\pi(\sigma)\circ T_{A(\sigma)}^1$ 
does not identically equal zero near the origin. 
Indeed, this follows from the equations (\ref{eqn:11.21}) 
with $g_{\sigma}(0)\neq 0$.
\end{itemize}
%%%
\end{proof}
%%%%%%%%%%%%%%%%%%%%%%%%%%%%%%%%%%%%%%%%%%%

\subsection{Poles on negative integers}
Let us consider the poles of 
$Z_{\pm}(s;\varphi)$ on negative integers.
For $\lambda\in -\N$,
define 
%%%%%
\begin{equation}\label{eqn:11.27}
\begin{split}
&
A_{\lambda}(\sigma):=\{j\in B(\sigma): 
l_f(a^j(\sigma))\lambda+l_g(a^j(\sigma))+
\langle  a^j(\sigma)\rangle -1 \in - \N\},\\
&
\rho_{\lambda}:=
\min\{\max\{\# A_{\lambda}(\sigma):
\sigma\in\Sigma^{(n)}\},n-1\}.
\end{split}
\end{equation}
%%%%%
We obtain the following result corresponding to 
Proposition~\ref{pro:10.8}.
%%%%%%%%%%%%%%%%%%%%%%%%
\begin{proposition}\label{pro:11.11}
Suppose that {\rm (i)} $f$ satisfies the condition $(E)$
and that {\rm (ii)} $g$ belongs to the class 
$\hat{\mathcal E}(U)$.
If the support of $\varphi$ is
contained in a sufficiently small neighborhood 
of the origin, then the orders of poles of 
$Z_{\pm}(s;\varphi)$  
at $s=\lambda\in -\N$ are not 
higher than $\rho_{\lambda}+1$.
In particular, if $\lambda>-1/d(f,g)$, 
then these orders are not higher than $1$.    
Moreover, 
let $a_{\lambda}^{\pm}$ be the coefficients of 
$(s-\lambda)^{-\rho_{\lambda}-1}$ 
in the Laurent expansions of 
$Z_{\pm}(s;\varphi)$ at $s=\lambda$, 
respectively,  
then we have $a_{\lambda}^+=(-1)^{\lambda-1}a_{\lambda}^-$ 
for $\lambda\in-\N$.
\end{proposition}
%%%%%%%%%%%%%%%%%%%%%%%%

%%%%%%%%%%%%%%%%%%%%%%%%%%%%%%%%%%%%%%%%%%%%%%%%
\begin{proof}
The difference between (\ref{eqn:10.7}) 
and (\ref{eqn:11.4})
does not essentially affect the argument in the proof
of Proposition \ref{pro:10.8}.
The details are left to the readers.
\end{proof}
%%%%%%%%%%%%%%%%%%%%%%%%%%%%%%%%%%%%%%%%%%%%%%%%

%%%%%%%%%%%%%%%%%%%%%%%%%%%%%%%%%%%%%%%%%%%%%%%%%%%%%%%%%%%%%%%%%%%%%%%%%%%%%%%%%%%%%%%%%%%%%%%%%%%%%%%%%%%%%%%%%%%%%%%%%%%%%%%%%%%%%%%%%
%%%%%%%%%%%%%%%%%%%%%%%%%%%%%%%%%%%%%%%%%%%%%%%%%%%%%%%%%%
%%%%%%%%%%%%%%%%%%%%%%%%%%%%%%%%%%%%%%%%%%%%%%%%%%%%%%%
\section{The case of convenient $f$}\label{sec:12}

Let us consider the poles of $Z_{\pm}(s;\varphi)$
and $Z(s;\varphi)$ in the case when 
the function $f$ is convenient.
In this case, all the results in Section~\ref{sec:11} can be obtained 
without the hypothesis: $g\in\hat{\mathcal E}(U)$.

%%%

In this section, 
we only assume that $g$ is a nonflat smooth function defined on 
an open neighborhood $U$ of the origin and 
we use the same notation as in
Section~\ref{sec:11}.
The purpose of this section is to show the following theorem. 

%%%%%%%%%%%%%%%%%%%%%%%%%%%%%%%%%%%%%%%%%
\begin{theorem}\label{thm:12.1}
In Theorems~\ref{thm:11.1},\ref{thm:11.10}  
and Propositions~\ref{pro:11.9},\ref{pro:11.11},
replacing the condition {\rm (ii)}$:$
``$g$ belongs to the class $\hat{\mathcal E}(U)$''
by
the condition:``$f$ is convenient'',
we have the same assertions. 
%in the respective theorems and proposition.  
\end{theorem}
%%%%%%%%%%%%%%%%%%%%%%%%%%%%%%%%%%%%%%%%%
%%
After the above exchanging in Theorem~\ref{thm:11.1}, 
the new theorem will be expressed by Theorem~11.1$'$.
The other new theorem and propositions will be expressed 
in a similar fashion. 
%%%%%%%%%%%%%%%%%

\subsection{Why the case of covenient $f$ can be conveniently treated?}
%%%%%%%%%%%%%%%%%%%%%%%%%%%%%%%%%%%%%%%%%%%%%%%%%%%%%%%%%%%%%%%%%%%%%%%

Before exactly proving Theorem~\ref{thm:12.1},
we roughly explain essential reason why the convenience of $f$
can give the assertions 
without strong hypotheses of $g$. 
Hereafter, we assume that 
$\varphi$ is always nonnegative on $U$ and 
$\varphi(0)>0$. 

First, let us recall and consider 
geometrical relationship between 
the Newton distance $d(f,g)$ and
the two Newton polyhedra
$\Gamma_+(f)$ and $\Gamma_+(g)$.
When $d>0$ is large enough, 
the two sets $(1/d)\cdot\partial\Gamma_+(f)$
and $\Gamma_+(g)+\1$ are separate,
where $\partial\Gamma_+(f)$ is the topological 
boundary of $\Gamma_+(f)$ and 
$\1=(1,\ldots,1)$. 
As $d$ tends to zero, 
the set $(1/d)\cdot\partial\Gamma_+(f)$
gradually moves towards the set $\Gamma_+(g)+\1$ and, 
finally, $(1/d)\cdot\partial\Gamma_+(f)$
contacts $\Gamma_+(g)+\1$ from the outside.
The value of this $d$ defines the {\it Newton distance}  $d(f,g)$. 
Theorem~11.11 asserts that
\begin{equation}\label{eqn:12.1}
\mbox{
``the leading pole of 
$Z(s;\varphi)$ exists at  
$s=-1/d(f,g)$''}
\end{equation}
under some appropriate conditions.

In order to understand the validity of the convenience
condition of $f$ for the assertion (\ref{eqn:12.1}), 
let us observe a typical problem occuring in the case when
$f$ is not convenient
and $g$ does not belong to $\hat{\mathcal{E}}(U)$.
%%%
Consider the following two-dimensional example:
%%%%%%%
\begin{equation}
\begin{split}
f(x_1,x_2)=x_1^4; \quad\quad 
g(x_1,x_2)=x_1^{2}x_2^{2}+e^{-1/x_2^2}.
\end{split}
\end{equation}
%%%%%%%%%
This example has appeared in Remark~3.4 and 
it is a special case of Example 1 in Section~15
(consider the case when $c=1,p=1,q=0$).
%From the compuation in Section 15,
%$Z(s;\varphi)$ have a pole at $s=-1/2$ 
%even though $-1/d(f,g)=-3/4$.
The Newton data about these $f$ and $g$ 
can be easily seen and we have  
%%%%%%%%%
\begin{equation*}
\begin{split}
&(1/d)\cdot\Gamma_+(f)=\{(\a_1,\a_2)\in\R_+^2:
\a_1\geq 4/d\}, \\
&\Gamma_+(g)+\1=\{(\a_1,\a_2)\in\R_+^2:
\a_1\geq 3,\a_2\geq 3\}, \\ 
&d(f,g)=4/3.
\end{split}
\end{equation*}
%%%%%%%%%% 
Of course,
the flat function term $e^{-1/x_2^2}$ in $g$
does not give any geometrical information
of the Newton polyhedron of $g$. 
However, 
this term analytically affects the properties of 
$Z(s;\varphi)$ and 
it actually gives an obstruction for the assertion (\ref{eqn:12.1}).
In fact, the leading pole of $Z(s;\varphi)$ exists at $s=-1/4$, 
which does not equal $-1/d(f,g)=-3/4$.

Let us try to forcibly give a geometrical Newton data
induced by the term $e^{-1/x_2^2}$ and
to understand how the leading pole of $Z(s;\varphi)$
appears at $s=-1/4$. 
Since $e^{-1/x_2^2}=O(x_2^N)$ for
any positive integer $N$,
it might be interpreted that
there exists 
something like a {\it speck of dust} 
$\{(0,N)\}$ very far from the origin
outside of the Newton polyhedron of $g$ and, 
more exactly, the set $\Gamma_+(g)+\1$ should be 
approximated as follows.   
$$
\Gamma_+(g)+\1 \approx 
\mbox{convex hull of }
\left(
\{(\a_1,\a_2)\in\R_+^2:\a_1\geq 3,\a_2\geq 3\}
\cup\{(1,N+1)\}
\right),\quad
%\mbox{with $N>>1$.} 
$$
with $N \gg 1$.
This {\it approximated} Newton polyhedron plays 
a more essential role in the analysis of 
poles of $Z(s;\varphi)$. 
Indeed, let us replace the {\it real} Newton polyhedron of $g$
by the above {\it approximated} one and 
let $d$ tend to zero.
Before contacting with the real polyhedron $\Gamma_+(g)+\1$,
the set $(1/d)\cdot\partial\Gamma_+(f)$ intersects
the approximated one on 
the vertical edge containing the point $\{(1,N+1)\}$. 
The value of this $d$ is $4$, which 
determines the position of leading pole
of $Z(s;\varphi)$ 
from the computation in Section~15.1.
It easily follows from this observation
that these kinds of problems can occur only
when $\Gamma_+(f)$ does not 
intersect some coordinate hyperplane.
In order to refuse these problems
when $f$ is not convenient, 
we need some kind of assumption on $g$ %like $g\in\hat{\mathcal{E}}(U)$
for the assertion (\ref{eqn:12.1}), 
which reveals that 
there is {\it completely nothing} outside of the set 
$(1/d(f,g))\cdot\Gamma_+(f)$.
%$\Gamma_+(g)+\1$.
For example, 
the inclusion $\Gamma(g)+\1\subset (1/d(f,g))\cdot\partial\Gamma_+(f)$
implies that  
the condition $g\in\hat{\mathcal{E}}(U)$ is sufficient 
for this necessary assumption.

In the case when $f$ is convenient, 
fortunately,
these kinds of problem cannot occur.
Indeed,
if $f$ is convenient, then
the complement of the set
$(1/d)\cdot\Gamma_+(f)$ in $\R_+^2$ is bounded.
Even if a {\it speck of dust} induced 
from the flat function like the above
$\{(0,N)\}$ exists,  then
since it is always very far from the origin,
the set $(1/d)\cdot\partial\Gamma_+(f)$
does not intersect it  
before contacting with $\Gamma_+(g)+\1$.
This is an essential reason why the case when 
$f$ is convenient can be easily treated. 

%%%%%%%%%%%%%%%%%%%%%%%%%%%%%%%%%%%%%%%%%%%%%%%%%%%%

\subsection{Proof of Theorem~12.1}

Before proving the above theorem, 
we prepare two lemmas.
%%%%%%%%%%%%%%%%%%%%%%%%%%%%%%%%%%%%%%%%%%%%%%%%%%%%%%%%%%
%%%%%
For $a\in\Z_+^n$, we define 
\begin{eqnarray}\label{eqn:12.3}
%%%%%
\tilde{l}_g (a)=
\begin{cases}
l_g(a)& \quad \mbox{for $a\in\N^n$}, \\
  0   & \quad \mbox{for $a\in\Z_+^n\setminus\N^n$}.
\end{cases}
\end{eqnarray}
%%%%%%%%%%%%%%%%%%%%%%%%%%%%
%%%%%%%%%%%%%%%%%%%%%%%%%%%%%%%%%%%%%%%%%%%%
\begin{lemma}\label{lem:12.2}
Suppose that $\sigma \in \Sigma^{(n)}$.
Then we have the following.

{\rm (i)}\quad 
There exists a smooth function $\tilde{g}_{\sigma}$ defined 
on $\pi(\sigma)^{-1}(U)$ such that  
%%%%%
\begin{equation}\label{eqn:12.4}
g(\pi(\sigma)(y))=
\left(\prod_{j=1}^n 
y_j^{\tilde{l}_g(a^j(\sigma))}\right)
\tilde{g}_{\sigma}(y)
\quad \mbox{for $y\in \pi(\sigma)^{-1}(U)$}.
\end{equation}
%%%%%
We remark that $\tilde{g}_{\sigma}$ may vanish at the origin.

{\rm (ii)}\quad
Let $\gamma$ be a compact face of $\Gamma_+(g)$ 
defined by 
$$
\gamma:=\bigcap_{j\in I}H(a^j(\sigma),l_g(a^j(\sigma)))
\cap \Gamma_+(g)\quad 
\mbox{with $I\subset\{j:a_j(\sigma)\in\N^n\}$}
%\{1,\ldots,n\}}.
$$
$($In terms of symbols in Section~\ref{subsec:6.2}, 
$\gamma=\gamma(I,\sigma)$ with $P=\Gamma_+(g)$.$)$ 
Then, we have
\begin{equation}\label{eqn:12.5}
g_{\gamma}(\pi(\sigma)(y))=
\left(\prod_{j=1}^n y_j^{\tilde{l}_g(a^j(\sigma))}\right)
\tilde{g}_{\sigma}(T_I(y))
\quad \mbox{for $y\in \pi(\sigma)^{-1}(U)$.}
\end{equation}

%%%%%%%%%%%%%%%%%%%%%
{\rm (iii)}\quad 
Let $I$ be as in {\rm (ii)} and 
let $g_{\sigma}$ be defined by
%%%
\begin{equation}\label{eqn:12.6}
g_{\sigma}(y)
=\left(\prod_{j\not\in I} 
y_j^{\tilde{l}_g(a^j(\sigma))-l_g(a^j(\sigma))}\right)
\tilde{g}_{\sigma}(T_I(y))
\quad \mbox{for $y\in \pi(\sigma)^{-1}(U)$.}
\end{equation}
Then $g_{\sigma}$ is a polynomial of only variables $y_j$ 
with $j\not\in I$ 
and $g_{\sigma}(0)\neq 0$. 
Note that $g_{\sigma}(T_I(y))=g_{\sigma}(y)$.

%%%
{\rm (iv)}\quad
Let $\gamma$, $I$ be as in {\rm (ii)}. 
Then we have
\begin{equation}\label{eqn:12.7}
g_{\gamma}(\pi(\sigma)(T_I^1(y)))=
\left(\prod_{j\not\in I} 
y_j^{l_g(a^j(\sigma))}\right)
g_{\sigma}(T_I(y))
\quad \mbox{for $y\in \pi(\sigma)^{-1}(U)$}.
\end{equation}
\end{lemma}
%%%%%%%%%%%%%%%%%%%%%%%%%%%%%%%%%%%%%%%%%%%%

\begin{remark}\label{rem:12.3}
When $g$ does not belong to $\hat{\mathcal E}(U)$,
the smooth function $g_{\sigma}$ cannot always be defined
by applying Lemma~\ref{lem:7.1}.
When we prove Theorem~12.1, 
we use the function $g_{\sigma}$, 
which is a polynomial defined as 
in (iii) in the above lemma.
\end{remark}

\begin{proof}[Proof of Lemma \ref{lem:12.2}]
(i)\quad 
 Let $\Sigma_+^{(1)}$ be the subset of 
$\Sigma^{(1)}$ defined by 
$
\Sigma_+^{(1)}:=\{a\in\Sigma^{(1)}:a\in\N^n\}.
$
We remark that 
the convenience condition of $f$ implies that
$\Sigma_+^{(1)}$ is not an empty set. 
Let
$$
\tilde{\Gamma}_+(g):=
\bigcap_{a\in\Sigma_+^{(1)}}
H^+(a,l_g(a))\cap\R_+^n.
$$
Then, 
$\tilde{\Gamma}_+(g)$ is a polyhedron
containing $\Gamma_+(g)$.
Note that 
$g$ is convenient if and only if
${\Gamma}_+(g)=\tilde{\Gamma}_+(g)$.
It follows from Definition~\ref{def:2.16} 
that
$g$ belongs to $\hat{\mathcal E}[\tilde{\Gamma}_+(g)](U)$.
From Proposition~\ref{pro:2.17}, 
we can obtain the expression of 
the form (\ref{eqn:2.14})
with $f=g$ and $P=\tilde{\Gamma}_+(g)$.
Furthermore, substituting 
$x=\pi(\sigma)(y)$ in (\ref{eqn:5.1}) into this expression, 
we obtain
$$g(\pi(\sigma)(y))=
\sum_{p\in S}
\left(
\prod_{j=1}^n y_j^{\langle a^j(\sigma),p\rangle }
\right)\psi_p(\pi(\sigma)(y)).
$$
Now, let us set 
$$
\tilde{g}_{\sigma}(y):=
\sum_{p\in S}
\left(
\prod_{j=1}^n 
y_j^{\langle a^j(\sigma),p\rangle -\tilde{l}_g(a^j(\sigma))}
\right)\psi_p(\pi(\sigma)(y)),
$$
then $\tilde{g}_{\sigma}$ is a smooth function defined on 
the set $\pi(\sigma)^{-1}(U)$ and satisfies (\ref{eqn:12.4}).

(ii) \quad 
Let $\tilde{\gamma}$ be a face of 
$\tilde{\Gamma}_+(g)$ defined by 
$
\tilde{\gamma}:=\bigcap_{j\in I}H(a^j(\sigma),l_g(a^j(\sigma)))
\cap \tilde{\Gamma}_+(g)
$.
In terms of symbols in Section~\ref{subsec:6.2}, 
this equation is written as 
$\tilde{\gamma}=\gamma(I,\sigma)$ 
with $P=\tilde{\Gamma}_+(g)$. 
Thus, the equation (\ref{eqn:7.4}) in Lemma~7.1  implies that
$g_{\tilde{\gamma}}(\pi(\sigma)(y))$
equals the right hand side of (\ref{eqn:12.5}).
The compactness of $\gamma$ implies 
$g_{\tilde{\gamma}}=g_{\gamma}$. 
So, we obtain (\ref{eqn:12.5}).
%%%%%%%%%%%%%

%%%%%%%%%%%%%
(iii)\quad
From Proposition~\ref{pro:2.17} and the compactness of $\gamma$, 
$g_{\gamma}$ can be written as 
$g_{\gamma}(x)=\sum_{p\in\gamma\cap\Z_+^n}x^p\psi_p(0)$,
where $\psi_p(0)\neq 0$ if $p$ is a vertex of $\gamma$.
Substituting $x=\pi(\sigma)(y)$ with (\ref{eqn:5.1}), 
we have
\begin{equation}\label{eqn:12.8}
g_{\gamma}(\pi(\sigma)(y))=
\sum_{p\in\gamma\cap\Z_+^n}
\left(
\prod_{j=1}^n y_j^{\langle a^j(\sigma),p\rangle }
\right)\psi_p(0).
\end{equation}
Putting (\ref{eqn:12.5}), (\ref{eqn:12.8}) together, we have
\begin{equation}\label{eqn:12.9}
\tilde{g}_{\sigma}(T_I(y))=
\sum_{p\in\gamma\cap\Z_+^n}
\left(
\prod_{j\not\in I} y_j^{\langle a^j(\sigma),p\rangle -\tilde{l}_g(a^j(\sigma))}
\right)\psi_p(0).
\end{equation}
Moreover, putting (\ref{eqn:12.6}),(\ref{eqn:12.9}) together, we have
$$
g_{\sigma}(y)=
\sum_{p\in\gamma\cap\Z_+^n}
\left(
\prod_{j\not\in I} y_j^{\langle a^j(\sigma),p\rangle -l_g(a^j(\sigma))}
\right)\psi_p(0).
$$
It directly follows from the above equation that 
$g_{\sigma}$ is a polynomial of $y_j$ with $j\not\in I$ 
and that $g_{\sigma}(0)=\psi_p(0)(\neq 0)$ where $p$ is some 
vertex of $\Gamma_+(g)$.

(iv)\quad
From (\ref{eqn:12.5}), (\ref{eqn:12.6}),
we obtain (\ref{eqn:12.7}).
\end{proof}
%%%%%%%%%%%%%%%%%%%%%%%%%%%%%%%%%%%%%%%%%%%%

%%%%%%%%%%%%%%%%%%%%%%%%%%%%%%%%%%%%%%%%
\begin{lemma}\label{lem:12.4}
Let $a$ be in $\Z_+^n$.
If $f$ is convenient and $l_f(a)\neq 0$, then 
$a\in\N^n$.
%which implies $\tilde{l}_g(a)=l_g(a)$ from Lemma~.
\end{lemma}
%%%%%%%%%%%%%%%%%%%%%%%%%%%%%%%%%%%%%%%%
%%%%%%%%%%%%%%%%%%%%%%%%%%%%%%%%%%%%%%%%
\begin{proof}%[Proof of Lemma~\ref{lem:12.3}]
Noticing $l_f(a)=0$ when $f$ is convenient and 
$a\in \Z_+^n\setminus \N^n$, 
we can obtain the lemma. 
\end{proof}
%%%%%%%%%%%%%%%%%%%%%%%%%%%%%%%%%%%%%%%
%%%%%%%%%%%
\begin{proof}[Proof of Theorem~11.1$'$]
Instead of (\ref{eqn:11.2}), 
the equation (\ref{eqn:12.4}) is 
substituted into (\ref{eqn:11.3}).
Then an integrals analogous to (\ref{eqn:11.4})
can be obtained.
From these integrals, 
we similarly see the set of candidate poles 
as follows.
%%%%%%%%%%%%%%%
\begin{equation}\label{eqn:12.10}
\left\{
-\frac{\tilde{l}_g(a^j(\sigma))+
\langle a^j(\sigma)\rangle +\nu}{l_f(a^j(\sigma))}:
\nu\in\Z_+, j\in B({\sigma})
\right\}\cup(-\N)
\end{equation}
%%%%%%%%%%%%%%%
Notice that
$j\in B(\sigma)\Longleftrightarrow 
l_f(a^j(\sigma))\neq 0\Longrightarrow 
a^j(\sigma)\in\N^n$ from Lemma~\ref{lem:12.4}, which 
implies 
$\tilde{l}_g (a^j(\sigma))=l_g(a^j(\sigma))$.
Therefore, 
we see that 
the set (\ref{eqn:12.10}) is contained in the set 
(\ref{eqn:11.1}). 
The exchanging in the assumption 
does not affect the rest of the assertions
in Theorem~\ref{thm:11.1}.
\end{proof}

%%%%%%%%%%%%%%%%%%%%%%%%%%%%%%%%%%%%%%%%%%%%%%%%

\begin{proof}[Proof of Proposition~11.9$'$]

We show only the case when $m(f,g)<n$ and the hypothesis (a), 
because the other cases can be easily shown by using the proof of 
Proposition~11.9.

Let us compute the limits $\tilde{C}_{\pm}$ in (\ref{eqn:11.17}). 
From the integrals in the proof of Theorem~11.1$'$, 
which are analogous to (\ref{eqn:11.4}), 
we can see that the limits $\tilde{C}_{\pm}$ takes the values 
analogous to $G_{\pm}(f,g,\varphi)$ in (\ref{eqn:11.20}).
Indeed, these values are obtained by exchanging 
$g_{\sigma}$, $l_g(a^j(\sigma))$ 
with $\tilde{g}_{\sigma}$, $\tilde{l}_g(a^j(\sigma))$
in (\ref{eqn:11.20}), 
respectively. 
Substituting the first equation in (\ref{eqn:11.21}) and 
the equation in (\ref{eqn:12.5}) with $\gamma=\gamma_*(\sigma)$ 
and $I=A(\sigma)$, we obtain the equations (\ref{eqn:11.22}).
We remark that at present the nonvanishing of $G_{+}+G_{-}$ 
cannot been shown, 
because $\tilde{g}_{\sigma}(0)$ may vanish. 

Furthermore, substituting the first equation in (\ref{eqn:11.21})
and the equation in (\ref{eqn:12.5}) with $\gamma=\gamma_*(\sigma)$ 
and $I=A(\sigma)$
into (\ref{eqn:11.22}), 
we obtain (\ref{eqn:11.20}). 
\end{proof}

%%%%%%%%%%%%%%%%%%%%%%%%%%%%%%%%%%%%%%%

\begin{proof}[Proof of Theorem~11.11$'$]
In a similar fashion to the proof of Theorem~11.11, 
we can obtain Theorem~11.11$'$ by using Proposition~11.9$'$. 
Note that 
%from the above expressions obtained in the proof of 
%Proposition~\ref{pro:11.9}, 
%we can see 
the nonvanishing of $C=C_{+}+C_{-}$ can be similarly seen from 
the fact that $g_{\sigma}(0)\neq 0$.
%in the same way as in the proof of Theorem~\ref{thm:11.10}.
\end{proof}

\begin{proof}[Proof of Proposition~11.12$'$]
This can be similarly shown by 
using the integrals analogous to 
(\ref{eqn:11.4}), 
which were obtained in the proof of Theorem~11.1$'$.
%%%%%%%%%%%%%%%%%%%%%%%%%%%%%%%%%%%%%%%%%%%%%%%
\end{proof}

%%%%%%%%%%%%%%%%%%%%%%%%%%%%%%%%%%%%
%%%%%%%%%%%%%%%%%%%%%%%%%%%%%%%%%%%%%%%%%%%
%%%%%%%%%%%%%%%%%%%%%%%%%%%%%%%%
%%%%%%%%%%%%%%%%%%%%%%%%%%%%%%%%%%%%%%%%%%%%%%%%%%%%%%%%%%%%%

\section{Certain symmetry properties}\label{sec:13}

We denote by 
$\beta_{\pm}(f,g)$, $\hat{\beta}(f,g)$ 
the largest poles 
of $Z_{\pm}(s;\varphi)$, $Z(s;\varphi)$ and 
by $\eta_{\pm}(f,g)$, $\hat{\eta}(f,g)$ 
their orders, 
respectively.
In this section,
we give some symmetry properties of
$\b_{\pm}(f,g), \hat{\b}(f,g), \eta_{\pm}(f,g), \hat{\eta}(f,g)$.
%%
%%%%%%%%%%%%%%%%%%%%%%%%%%%%%%%%%%%%%%%%%%%%%%%%%%%%%%%%%%
\begin{theorem}\label{thm:13.1}
Suppose that $f,g$ satisfy the condition $(E)$
and are nonnegative or nonpositive 
on $U$. 
If the support of $\varphi$ is
contained in a sufficiently small neighborhood 
of the origin, then we have 
%%%%%
\begin{equation}\label{eqn:13.1}
\b_{\pm}(x^{\1}f,g)\b_{\pm}(x^{\1}g,f)\leq 1 
\mbox{ and } 
\hat{\b}(x^{\1}f,g)\hat{\b}(x^{\1}g,f)\leq 1
\end{equation}
%%%%%
Moreover, the following two conditions are equivalent$:$ 
%%%%%%%%%%%%%%%%%%%%%%%%%%
\begin{enumerate}
\item The equality holds in each estimate in 
$(\ref{eqn:13.1});$
\item There exists a positive rational number $d$ such that 
$\Gamma_+(x^{\1} f)=d\cdot\Gamma_+(x^{\1} g)$.
\end{enumerate}
%%%%%%%%%%%%%%%%%%%%%%%%%%
If the condition {\rm (i)} or {\rm (ii)} is satisfied, 
then we have 
$\eta_{\pm}(x^{\1}f,g)=\eta_{\pm}(x^{\1}g,f)
=\hat{\eta}(x^{\1}f,g)=\hat{\eta}(x^{\1}g,f)=n$.
\end{theorem}
%%%%%%%%%%%%%%%%%%%%%%%%%%%%%%%%%%%%%%%%%%%%%%%%%%%%%%%%%%%
%%%%%%%%%%%%%%%%%%%%%%%%%%%%%%%%%%%%%%%%%%%%%%%%%%%%%%%%%%%
\begin{proof}
We only consider the case of $Z_{\pm}(s;\varphi)$.

Admitting Lemmas~\ref{lem:13.2} and~\ref{lem:13.3} below, 
we can prove the theorem as follows. 
From the assumptions in the theorem and Lemma~\ref{lem:13.2},
Theorem~\ref{thm:11.10} implies the equations 
$\beta(x^{\1}f,g)=-1/d(x^{\1}f,g)$ 
and $\beta(x^{\1}g,f)=-1/d(x^{\1}g,f)$. 
By Lemma~\ref{lem:13.3}, 
these equations imply the inequalities
(\ref{eqn:13.1}) and the equivalence of (i) and (ii).
The condition (ii) in the theorem implies 
that $d\cdot(\Gamma_+(g)+\1)=\Gamma_+(x^{\1}f)$ 
with $d=d(x^{\1}f,g)$ and, 
moreover, that
$\Gamma_0(x^{\1}f)=\partial\Gamma_+(x^{\1}f)$
for the pair $(x^{\1}f,g)$.  
Thus, each principal face of $\Gamma_+(g)$ 
is a vertex  
(i.e., zero-dimensional face) of 
$\Gamma_+(g)$. 
(See Remark~2.7.)
This means $\eta_{\pm}(x^{\1}f,g)
=m(x^{\1}f,g)=n$. 
Similarly, we see $\eta_{\pm}(x^{\1}g,f)=n$.
\end{proof}
%%%%%%%%%%%%%%%%%%%%%%%%%%%%%%%%%%%%%%%%%%%%%%%%%%%%%%%%%%%

%%%%%%%%%%%%%%%%%%%%%%%%%%%%%%%%%%%%%%%%%%%%%%%%%%%%%%%%%%%
\begin{lemma}\label{lem:13.2}
Let $f$ be a smooth nonflat function
defined on a neighborhood of the origin in $\R^n$. 
Then $f$ is nondegenerate over $\R$
with respect to its Newton polyhedron 
if and only if so is $x^{\1} f$. 
\end{lemma}
%%%%%%%%%%%%%%%%%%%%%%%%%%%%%%%%%%%%%%%%%%%%%%%%%%%%%%%%%%%
\begin{proof}
This is a special case of Lemma~7.8.
\end{proof}
%%%%%%%%%%%%%%%%%%%%%%%%%%%%%%%%%%%%%%%%%%%%%%%%%%%%%%%%%%
\begin{lemma}\label{lem:13.3}
Let $f$,$g$ be smooth functions defined 
near the origin, then 
we have 
$$d(x^{\1}f,g) d(x^{\1}g,f)\geq 1.$$
The equality holds in the above, if and only if 
there exists a positive rational number $d$ such that
$\Gamma_+(x^{\1} f)=d\cdot\Gamma_+(x^{\1} g)$.
\end{lemma}
%%%%%%%%%%%%%%%%%%%%%%%%%%%%%%%%%%%%%%%%%%%%%%%%%%%%%%%%%%%
\begin{proof}
See Lemma 5.19 in \cite{ckn13}.
\end{proof}
%%%%%%%%%%%%%%%%%%%%%%%%%%%%%%%%%%%%%%%%%%%%%%%%%%%%%%%%%%%

%%%%%%%%%%%%%%%%%%%%%%%%%%%%%%%%%%%%%%%%%%%%%%%%%%%%%%%%%%%%%%%%%%%%%%%%%%%%%%%%%%%%%%%%%%%%%%%%%%%%%%%%%%%%%
\section{Proofs of the theorems in Section~4}\label{sec:14}

%%%%%%%%%%%%%%%%%%%%%%%%%%%%%%%%%%%%%%%%%%%%%%%%%%%%%%%%%%%%%%%%%%%%%%%%%%%%%%%%%%%%%%%%%%%%%%%%%%%%%%%%%%%%%%

\subsection{Relationship between $I(t;\varphi)$ and 
$Z_{\pm}(s;\varphi)$}\label{subsec:14.1}

It is known 
(see \cite{igu78}, \cite{agv88}, etc.) 
that 
the study of the asymptotic behavior of the 
oscillatory integral $I(t;\varphi)$ in (\ref{eqn:1.1})
can be reduced to an investigation of the poles
of the functions $Z_{\pm}(s;\varphi)$ in (\ref{eqn:8.1}). 
Let us explain an outline of this reduction. 
Let $f$ and $\varphi$ satisfy 
the conditions (A) and (C) in Section~\ref{sec:3}
respectively
and let $g$ be a smooth function on an
open neighborhood of the origin in $\R^n$
($g$ is possibly flat).
Suppose that the support of $\varphi$ is sufficiently small. 

Define the {\it weighted} Gelfand-Leray function: 
$K:\R \to \R$ as
%%%%%
\begin{equation}\label{14.1}
K(c)=\int_{W_c} g(x) \varphi(x) \omega, 
\end{equation}
%%%%%
where $W_c=\{x\in\R^n: f(x)=c\}$ and $\omega$ is 
the surface element on $W_c$ which is determined by 
$
df\wedge \omega=dx_1\wedge\cdots\wedge dx_n.
$
$I(t;\varphi)$ and $Z_{\pm}(s;\varphi)$ can be expressed by 
using $K(c)$: 
Changing the integral variables 
in (\ref{eqn:1.1}),(\ref{eqn:8.1}), we have 
%%%%%
\begin{equation}\label{eqn:14.2}
\begin{split}
&I(t;\varphi)=\int_{-\infty}^{\infty}
e^{it c}K(c)dc 
=\int_{0}^{\infty}
e^{it c}K(c)dc+
\int_{0}^{\infty}
e^{-it c}K(-c)dc,
\end{split}
\end{equation}
%%%%%
%%%%%
\begin{equation}\label{eqn:14.3}
Z_{\pm}(s;\varphi)
=\int_0^{\infty} c^s K(\pm c)dc,
\end{equation}
%%%%%
respectively. 
Applying the inverse formula of the Mellin transform
to (\ref{eqn:14.3}),  
we have 
%%%%%
\begin{equation}\label{eqn:14.4}
K(\pm c)
=\frac{1}{2\pi i}
\int_{r-i\infty}^{r+i\infty} 
Z_{\pm}(s;\varphi)c^{-s-1}ds,
\end{equation}
%%%%%
where $r>0$ and the integral contour follows
the line Re$(s)=r$ upwards. 
Recall that 
$Z_{+}(s;\varphi)$ and $Z_{-}(s;\varphi)$ are meromorphic functions and 
their poles exist on the negative part of the real axis. 
By deforming the integral contour as $r$ tends 
to $-\infty$ in (\ref{eqn:14.4}), 
the residue formula gives the
asymptotic expansions of $K(c)$ 
as $c\to\pm 0$.
Substituting these expansions of $K(c)$ into 
(\ref{eqn:14.2}), 
we can get an asymptotic expansion of $I(t;\varphi)$ 
as $t\to+\infty$.

Through the above calculation, 
we see more precise relationship for the coefficients. 
If $Z_{+}(s;\varphi)$ and $Z_{-}(s;\varphi)$ 
have the Laurent expansions at $s=-\lambda$:
%%%%%
\begin{equation*}
Z_{\pm}(s;\varphi)=\frac{B_{\pm}}{(s+\lambda)^{\rho}}+
O\left(\frac{1}{(s+\lambda)^{\rho-1}}\right),
\end{equation*}
%%%%%
respectively, 
then the corresponding part in the asymptotic
expansion of $I(t;\varphi)$ has the form
%%%%%
$$
Bt^{-\lambda}(\log t)^{\rho-1}+
O(t^{-\lambda}(\log t)^{\rho-2}).
$$ 
%%%%%
Here a simple computation gives 
the following relationship:
%%%%%
\begin{equation}\label{eqn:14.5}
B=\frac{\Gamma(\lambda)}{(\rho-1)!}
\left[
e^{i\pi \lambda/2}B_+ +e^{-i\pi \lambda/2}B_-
\right],
\end{equation}
%%%%%
where $\Gamma$ is the Gamma function. 

%%%%%%%%%%%%%%%%%%%%%%%%%%%%
\begin{remark}\label{rem:14.1}
If $\lambda$ is not an odd integer, then
$${\rm Re}(B)=
\frac{2\Gamma(\lambda)\cos(\pi\lambda/2)}
{(\rho-1)!}(B_+ +B_-).$$
In order to decide the vanishing of the coefficient, 
the above equation is helpful. 
\end{remark}
%%%%%%%%%%%%%%%%%%%%%%%%%%%%

%%%%%%%%%%%%%%%%%%%%%%%%%%%%%%%%%%%%%%%%%%%%%%%%%%%%%%%%%%%%%%%%%%%%%%%%%%%%%%%%

\subsection{The first coefficient in the asymptotics of $I(t;\varphi)$}
\label{subsec:14.2}

From the relationship between $I(t;\varphi)$
and $Z_{\pm}(s;\varphi)$ in the previous section and 
the equation (\ref{eqn:14.5}),
we give explicit formulae for the coefficient of 
the leading term of the asymptotic expansion (\ref{eqn:1.2})
of $I(t;\varphi)$ at infinity   
as follows. 
%%%%%%%%%%%%%%%%%%%%%%%%
\begin{theorem}\label{thm:14.2}
If $f$ and $g$ satisfy the conditions $(i),(ii),(iv)$ 
in Theorem~\ref{thm:4.4} and 
the support of $\varphi$ is contained in a
sufficiently small neighborhood of the origin, 
then we have
\begin{equation*}
\begin{split}
&\lim_{t\to\infty}
t^{1/d(f,g)}(\log t)^{-m(f,g)+1}\cdot I(t;\varphi)\\
&\quad\quad=
\frac{\Gamma(1/d(f,g))}{(m(f,g)-1)!}
[
e^{i\pi/(2d(f,g))}C_+ +
e^{-i\pi/(2d(f,g))}C_-
],
\end{split}
\label{eqn:}
\end{equation*}
where $C_{\pm}$ are as in $(\ref{eqn:11.26})$ 
in Theorem~\ref{thm:11.10}.
\end{theorem}

\begin{remark}
In the unweighted case, 
analogous results have been obtained in 
\cite{sch91},\cite{dns05},\cite{gre09}.
\end{remark}
%%%%%%%%%%%%%%%%%%%%%%%%%%%%%%%%%%%%%%%%%%%%%%%%%%%%%%%%%%%%%

\subsection{Proofs of 
Theorems~\ref{thm:4.1}, \ref{thm:4.4}~ and \ref{thm:4.7}}
\label{subsec:14.3}

Applying the argument in Section~14.1 to the results 
relating to $Z_{\pm}(s;\varphi)$ 
in Sections~\ref{sec:11} and \ref{sec:12}, 
we obtain the theorems in Section~\ref{sec:4}.

%%%%%%%%%%%%%%%%%%%%%%%%%%%%%%%%%%%%%%%%%%%%
\begin{proof}[Proof of Theorem~\ref{thm:4.1}.]
This theorem follows from 
Theorems~\ref{thm:11.1} and \ref{thm:12.1}.
Notice that Proposition~\ref{pro:11.11} 
and the relationship (\ref{eqn:14.5})
induce the cancellation of the coefficients of the term 
whose decay rates are larger than 
$t^{-1/d(f,g)}(\log t)^{m(f,g)-1}$.
\end{proof}
%%%%%%%%%%%%%%%%%%%%%%%%%%%%%%%%%%%%%%%%%%%%%
\begin{proof}
[Proof of the estimate (\ref{eqn:4.2}) in Remark~\ref{rem:4.3}]
This estimate follows from Theorem~10.1 by using the
relationship in Section~14.1.
One must consider the cancellation similar to that 
in the proof of Theorem~$\ref{thm:4.1}$.
\end{proof}

%%%%%%%%%%%%%%%%%%%%%%%%%%%%%%%%%%%%%%%%%%%%
%%%%%%%%%%%%%%%%%%%%%%%%%%%%%%%%%%%%%%%%%%%%
\begin{proof}[Proof of Theorem~\ref{thm:4.4}.]
This theorem follows from Theorem~\ref{thm:14.2}  
by considering the assertions 
in Theorems~\ref{thm:11.10} and \ref{thm:12.1}.
Note that the necessity of the condition in (iii-c):
``$-1/d(f,g)$ is not an odd integer'' follows
from Remark~\ref{rem:14.1}. 
\end{proof}
%%%%%%%%%%%%%%%%%%%%%%%%%%%%%%%%%%%%%%%%%%%%%%%%%%%%%%%
%%%%%%%%%%%%%%%%%%%%%%%%%%%%%%%%%%%%%%%%%%%%%%%%%%%%%%%
\begin{proof}[Proof of Theorem~\ref{thm:4.7}.]
This theorem follows from Theorem~\ref{thm:13.1}. 
\end{proof}
%%%%%%%%%%%%%%%%%%%%%%%%%%%%%%%%%%%%%%%%%%%%%%%%%%%%%%%

%%%%%%%%%%%%%%%%%%%%%%%%%%%%%%%%%%%%%%%%%%%%%%%%%%%%%%%%%%%%%%%%%%%%%%%%%%%%%%%%%%%%%%%%%%%%%%%%%%%%%%%%%%%%%%%%%%%%%%%%%%%%%%%%%%%%%%%%%%%%%%%
%\section{}

\subsection{Generalization of Varchenko's results to the Puiseux series case}
%%%%%%%%%%%%%%%%%%%%%%%%%%%%%%%%%%%%%%%%%%%%%%%%%%%%%%%%%%%%%%%%
By using the results about local zeta type functions with respect to
the class $C_{1/p}^{\infty}(U_+)$ introduced in Section~10.4, 
we can give some kind of generalization of the results about
unweighted oscillatory integrals 
due to Varchenko~\cite{var76} and Kamimoto and Nose~\cite{kn13}.

Let $U$ be an open neighborhood of the origin in $\R^n$
and let $p:=(p_1,\ldots,p_n)\in\N^n$. 
The definitions of the  function spaces $C_{1/p}^{\infty}(U_+)$
and $\hat{\mathcal E}_{1/p}(U_+)$ were given in Section~10.4. 
Recall that every element of $C_{1/p}^{\infty}(U_+)$
or $\hat{\mathcal E}_{1/p}(U_+)$ can be
considered to admit the following fractional power series 
at the origin:
\begin{equation*}
f(x)\sim \sum_{\alpha\in\Z_+^n}c_{\alpha}x^{\alpha/p}:=
\sum_{\alpha\in\Z_+^n}c_{\alpha}x_1^{\alpha_1/p_1}\cdots x_n^{\alpha_n/p_n}.
%\label{eqn:}
\end{equation*}
For $f\in C_{1/p}^{\infty}(U_+)$, we define
%%%%%%%%%%%%%%%
\begin{equation}\label{eqn:14.7}
\tilde{I}(t;\varphi):=
\int_{\R_+^n}e^{it f(x)}\varphi(x)dx,
\end{equation}
%%%%%%%%%%%%%%%
where $t>0$ and $\varphi:U\to\R$ satisfies the condition (C)
in Section~3.

The assertions of the following theorem is 
almost the same as those of Theorems~3.7 and 3.1 in appearance.

%%%%%%%%%%%%%%%%%%%%%%%%%
\begin{theorem}\label{thm:14.4}
%th:3.1.5****************
Suppose that $f$ belongs to the class 
$\hat{\mathcal E}_{1/p}(U_+)$ and is nondenegerate
over $\R$ with respect to 
the Newton polyhedron $\Gamma_+(f)$
(see Defintion~10.10 (v)).
%%%
\begin{enumerate}
\item
If the support of $\varphi$ is contained in a 
sufficiently small neighborhood of the origin,
then $\tilde{I}(t;\varphi)$ admits an asymptotic expansion
of the form
   \begin{equation*}%\label{eqn:14.8}
   \tilde{I}(t;\varphi)\sim 
   \sum_{\a}\sum_{k=1}^n
   C_{\a k}(\varphi) t^{\a} (\log t)^{k-1} \quad 
   \mbox{as $t \to +\infty$},
   \end{equation*}
where
the progression $\{\a\}$
belongs to finitely many arithmetic progressions,
which are obtained by using the theory of 
toric varieties based on the geometry of
the Newton polyhedron $\Gamma_+(F_f)$,
where $F_f$ is as in Definition~10.10 (i).

Here the {\it oscillation index}  of $(f,1)$
and  the {\it multiplicity}  of its index can
be samely defined as in Definition~1.1, 
which are denoted by $\beta(f,1)$, $\eta(f,1)$, 
respectively;
\item
 $\beta(f,1)\leq -1/d(f);$
\item
 If at least one of the following conditions 
is satisfied:
\begin{enumerate}
\item $d(f)>1;$
\item $f$ is nonnegative or nonpositive on $U;$
\item 
$1/d(f)$ is not an odd integer and $f_{\tau_*}$ 
does not vanish on $U\cap(\R\setminus\{0\})^n$,
\end{enumerate}
then 
$\beta(f,1)=-1/d(f)$ and $\eta(f,1)=m(f)$.
\end{enumerate}
%%%%%%%%%%%%%%%%%%%%%%%%%
Here the definitions of $d(f)$,$m(f)$, $\tau_*$,
$f_{\tau_*}$ are given 
in Defintion~10.10 (iv).
\end{theorem}
%%%%%%%%%%%%%%%%%%%%%%%%%
\begin{proof}
By using the relationships in Section~14.1,  
Theorems~10.11 and 10.13 easily imply the above theorem. 
\end{proof}
%%%%%%%%%%%%%%%%%%%%%%%%%
\begin{remark}
It is expected to exchange
$\Gamma_+(F_f)$ by $\Gamma_+(f)$ in Theorem~14.4 (i). 
But, until now we have not yet verified. 
(See Remark~10.12.)
\end{remark}

%%%%%%%%%%%%%%%%%%%%%%%%%%%%%%%%%%%%%%%%%%%%%%%%%%%%%%%%%%%%%%%%%%%%%%%%%%%%%%%%%%%%%%%%%%%%%%%%%%%

\section{Examples}\label{sec:15}

In this section, we give three examples of 
the phase $f$ and the weight $g$ in the oscillatory integral 
(\ref{eqn:1.1}), 
which show the delicate situation around 
the hypothesis: 
\begin{enumerate}
\item[(ii-a)]
$g$ belongs to $\hat{\mathcal E}(U)$
\end{enumerate}
in the main theorems in Section~4.
It is easy to see that 
every phase $f$ treated in this section 
satisfies the conditions (D) and (E) in Section~3. 
Since the above condition (ii-a) 
is not needed when $f$ is convenient 
(see Section~12), 
we consider only the case when $f$ is not convenient. 
We always assume that $\varphi$ satisfies the conditions (C)
in Section~\ref{sec:3}. 
%(In Examples~1,~2,~3, each $f$, $g$, $g_t$ 
%satisfies the respective condition.)

(I)\quad
In the first example, 
the weight is monomial type as in Section~10. 
First, this example shows some kind of necessity of 
the condition (ii-a). 
Indeed, Example~1 shows the existence of 
$f,g$ satisfying that 
$
g\not\in\hat{\mathcal E}(U)$ and  
$\beta(f,g)\neq -1/d(f,g)$. 
%%%
On the other hand, in this example, 
the equation $\beta(f,g)=-1/d(f,x^p)$ holds 
even in the case when $g$ is flat in this example. 
This shows the optimality of 
the estimate (\ref{eqn:4.2}) in Remark~\ref{rem:4.3}.

(II)\quad
The second example indicates that
the hypothesis (ii-a) is too strong. 
Indeed, Example~2 shows the existence of
$f,g$ satisfying that
$
g\not\in\hat{\mathcal E}(U)$  but 
the equalities 
$\beta(f,g)= -1/d(f,g)$ and 
$\eta(f,g)=m(f,g)$ hold.
Furthermore, 
observing this example carefully, 
it might be expected to weaken the condition (ii-a) by 
\begin{enumerate}
\item[(ii-a$'$)]
$g$ belongs to 
$\hat{\mathcal E}[\Phi^{-1}(\Gamma_+(f))\cap\R_+^n](U)$.
\end{enumerate}

(III)\quad 
Unfortunately, the third example violates 
this expectation. 
Indeed, Example~3 shows the existence of 
$f,g$ satisfying that the condition (ii-a$'$) 
and $\beta(f,g)=-1/d(f,g)$ but
$\eta(f,g)\neq m(f,g)$.
%%%%%%%%%%%%%%%%%%%%%%%%%%%%%%%%%%%%%%%%%%%%%%%

\begin{remark}\label{rem:15.1}
In this remark, we assume that
$f,g$ satisfy all the conditions in 
the assumptions of Theorem~4.4 
except the condition (ii). 

(1)\quad
At present, there is no example of $f,g$ 
satisfying that the condition (ii-a$'$) holds 
but  
the equality $\beta(f,g)=-1/d(f,g)$
does not hold.

(2)\quad
The following modified condition (ii-a$''$)
is available for Theorem~\ref{thm:4.4}. 
\begin{enumerate}
\item[(ii-a$''$)]
$g$ belongs to $\hat{\mathcal E}[Q(f,g)](U)$.
\end{enumerate}
Here the set  $Q(f,g)$ is defined by 
the convex hull of the set 
%$(1/d(f,g)\cdot(\Gamma_+(f))-\1)\cap \Z_+^n +\R_+^n$
$\Phi^{-1}(\Gamma_+(f)\setminus\Gamma^{(n-m(f,g)-1)})\cap \Z_+^n +\R_+^n$,
where $\Gamma^{(k)}$ is the union of the $k$-dimensional
faces of $\Gamma_+(f)$ ($\Gamma^{(-1)}:=\emptyset$). 
By the definition, 
the set $Q(f,g)$ is a polyhedron in $\R_+^n$
containing $\Gamma_+(g)$.
Theorem~\ref{thm:4.4} can be partially improved by replacing (ii-a) by 
(ii-a$''$).
(The proof is omitted.) 
\end{remark}

%%%%%%%%%%%%%%%%%%%%%%%%%%%%%%%%%%%%%%%%%%%%%%%
%%%%%%%%%%%%%%%%%%%%%%%%%%%%%%%%%%%%%%%%%%%%%%%
%%%%%%%%%%%%%%%%%%%%%%%%%%%%%%%%%%%%%%%%%%%%%%%
%%%%%%%%%%%%%%%%%%%%%%%%%%%%%%%%%%%%%%%%%%%%%%%
%%%%%%%%%%%%%%%%%%%%%%%%%%%%%%%%%%%%%%%%%%%%%%%
%%%%%%%%%%%%%%%%%%%%%%%%%%%%%%%%%%%%%%%%%%%%%%%
%%%%%%%%%%%%%%%%%%%%%%%%%%%%%%%%%%%%%%%%%%%%%%%

\subsection{Example 1}\label{subsec:15.1}

Consider the following two-dimensional example 
with parameters $p,q\in\Z_+$ and $c\in\R$: 
%%%%%
\begin{equation}
\begin{split}
&f(x_1,x_2)=x_1^4,
\\
&g(x_1,x_2)=cx_1^{2p}x_2^{2p}+x_1^{2q}x_2^{2q}e^{-1/x_2^2}
(=:g_1(x_1,x_2)+g_2(x_1,x_2)).
\end{split}
\end{equation}
%%%%% 
It is easy to see that
%$f$ is nondegenerate over $\R$ with respect to 
%its Newton polyhedron,
\begin{itemize}
\item $\Gamma_+(f)=\{(4,0)\}+\R_+^2$,
\item
If $c\neq 0$, then
$\Gamma_+(g)=\Gamma_+(g_1)
=\{(2p,2p)\}+\R_+^2$, \\ 
If $c=0$, then
$\Gamma_+(g)=\Gamma_+(g_2)=\emptyset$
 (i.e., $g$ is a nonzero flat function), 
\item
$d(f,g)=\frac{4}{2p+1}$, $m(f,g)=1$ if $c\neq 0$, 
$d(f,x_1^{2q}x_2^{2q})=\frac{4}{2q+1}$.
\end{itemize}
Notice that
\begin{eqnarray*}
\begin{split}
&
g\in\hat{\mathcal E}(U)
\Leftrightarrow 
p\leq q \mbox{ and } c\neq 0 
\Leftrightarrow 
d(f,g)=d(f,x_1^{2p}x_2^{2p})\\
&\quad\quad\quad\quad\Leftrightarrow 
\exists h_1\in C^{\infty}(U) 
\mbox{ with $h_1(0)\neq 0$ such that } 
g(x)=x_1^{2p}x_2^{2p} h_1(x); \\
&g\not\in\hat{\mathcal E}(U)
\Leftrightarrow 
p > q \mbox{ or } c = 0 \\
&\quad\quad\quad\quad
\Leftrightarrow \exists h_2\in C^{\infty}(U) 
\mbox{ with $h_2(0)= 0$ such that } 
g(x)=x_1^{2q}x_2^{2q}h_2(x).
\end{split}
\end{eqnarray*}

We decompose $\tilde{Z}_{\pm}(s;\varphi)$ into   
$\tilde{Z}_{\pm}^{(1)}(s;\varphi)+
 \tilde{Z}_{\pm}^{(2)}(s;\varphi)$, where
%%%%%
$$
\tilde{Z}_{\pm}^{(j)}(s;\varphi)=\int_{\R_+^2} 
(f(x))_{\pm}^s g_j(x)\varphi(x)dx \quad\,\,
j=1,2.
$$
%%%%% 
Note $\tilde{Z}_{-}^{(j)}(s;\varphi)=0$ for $j=1,2$. 
A simple computation gives 
%%%%%
$$
\tilde{Z}_+^{(1)}(s;\varphi)=
c\int_0^{\infty}\int_0^{\infty} 
x_1^{4s+2p}x_2^{2p} \varphi(x_1,x_2)dx_1dx_2.
$$
%%%%%
Applying Proposition~\ref{pro:9.2}, 
we see that the poles of $Z_+^{(1)}(s;\varphi)$
are simple and they 
are contained in the set 
$\{-\frac{2p+\nu}{4}:\nu\in\N\}$. 
Similarly, the poles of 
%%%%%
$$
\tilde{Z}_+^{(2)}(s;\varphi)=
\int_0^{\infty}\int_0^{\infty} 
x_1^{4s+2q} x_2^{2q} e^{-1/x_2^2}\varphi(x_1,x_2)dx_1dx_2
$$
%%%%%
are simple and they are contained in the set 
$\{-\frac{2q+\nu}{4}:\nu\in\N\}$.
Moreover, applying Proposition~\ref{pro:9.3}, 
we see that
the coefficient of $(s+\frac{2q+1}{4})^{-1}$ 
in the Laurent expansion 
of $\tilde{Z}_+^{(2)}(s;\varphi)$ is 
%%%%%
\begin{equation}\label{eqn:15.2}
\frac{1}{4}\int_0^{\infty} 
x_2^{2q}e^{-1/x_2^2}\varphi(0,x_2)dx_2>0.
\end{equation}
%%%%%

From the property of poles of the above functions, 
we see that the pattern of asymptotic expansion
in this case is the same as that in (\ref{eqn:1.2})
even if $g$ is flat. 
This means that the oscillation index and its multiplicity
can be defined in this case.
 
Needless to say,  
we can get the same result as in Theorem~\ref{thm:4.4}
in the case when $g\in\hat{\mathcal E}(U)$. 
Indeed, the inequality: $p\leq q$ implies that 
$\beta(f,g)=-1/d(f,g)=-\frac{2p+1}{4}$. 
Note that if the support of $\varphi$ is sufficiently small, 
then the value of (\ref{eqn:15.2}) is also very small.

Next, consider the case when $g\not\in\hat{\mathcal E}(U)$.
When $c\neq 0$, we see that 
$\beta(f,g)(=-\frac{2q+1}{4})>-1/d(f,g)(=-\frac{2p+1}{4})$, 
but $\beta(f,g)=-1/d(f,x_1^{2q}x_2^{2q})(=-\frac{2q+1}{4})$.
When $c=0$,
in spite of the fact that $g$ is flat,  
we have $\beta(f,g)=-1/d(f,x_1^{2q}x_2^{2q})
(=-\frac{2q+1}{4})$.

As for the multiplicity, 
it is easy to see 
$\eta(f,g)=m(f,g)=1$
when $g$ is nonflat.

%%%%%%%%%%%%%%%%%%%%%%%%%%%%%%%%%%%%%%%%%%%%%%%%%%%%%%%%%%%%%
\subsection{Example 2}\label{subsec:15.2}
Consider the following three-dimensional example
with parameters $p,q\in \Z_+$:
%%%%%
\begin{equation*}
\begin{split}
&f(x_1,x_2,x_3)=x_1^4 +x_2^4,\\
&g(x_1,x_2,x_3)=x_1^2 + x_1^px_2^q e^{-1/x_3^2}
(=: g_1(x_1,x_2,x_3)+g_2(x_1,x_2,x_3)).
\end{split}
\end{equation*}
It is easy to see that 
%$f$ is nondegenerate over $\R$
%with respect to its Newton polyhedron,
\begin{itemize}
\item
$\Gamma_+(f)=\{\alpha\in\R_+^3:\alpha_1+\alpha_2\geq 4\}$,
\item
$\Gamma_+(g)=\Gamma_+(g_1)=\{\alpha\in\R_+^3:\alpha_1\geq 2\}$,
$\Gamma_+(g_2)=\emptyset$,
\item
$d(f,g)=1$, $m(f,g)=1$.
\end{itemize}
Moreover, we can see that
\begin{eqnarray*}
g\in\hat{\mathcal E}(U)
\Leftrightarrow p\geq 2, 
\quad\quad
g\not\in\hat{\mathcal E}(U)
\Leftrightarrow 
p=0,1.
\end{eqnarray*}

We decompose $\tilde{Z}_{\pm}(s;\varphi)$ into  
$\tilde{Z}_{\pm}^{(1)}(s;\varphi)+
 \tilde{Z}_{\pm}^{(2)}(s;\varphi)$, where
$$
\tilde{Z}_{\pm}^{(j)}(s;\varphi)=\int_{\R_+^3} 
(f(x))_{\pm}^s g_j(x)\varphi(x)dx \quad\,\,
j=1,2.
$$
%%%%% 
Note that $\tilde{Z}_{-}^{(j)}(s;\varphi)=0$ for $j=1,2$
and that 
$$
\tilde{Z}_{+}^{(1)}(s;\varphi)
=\int_{\R_+^3}
(x_1^4+x_2^4)^sx_1^2\varphi(x)dx_1dx_2dx_3.
$$
Applying the computation in Sections~\ref{sec:10} and
\ref{sec:11}, 
the poles of 
$\tilde{Z}_{+}^{(1)}(s;\varphi)$ are simple
and are contained in the set 
$\{-1-\nu/4:\nu\in\Z_+\}$.
Moreover, applying 
Proposition~\ref{pro:9.3}, we see that
the coefficient of $(s+1)^{-1}$ 
in the Laurent expansion of $\tilde{Z}_{+}^{(1)}(s;\varphi)$
is
\begin{equation}\label{eqn:15.3}
\frac{1}{4}
\left(\int_0^{\infty}\frac{y_1^2}{y_1^4+1}dy_1\right)
\left(\int_0^{\infty}\varphi(0,0,y_3)dy_3\right).
\end{equation}
%%%%%%%%%%%%%%%%%%%%%%%%%%
Similar computation implies that
the poles of $\tilde{Z}_{+}^{(2)}(s;\varphi)$
are simple and 
are contained in the set 
$\{-\frac{p+q+2+\nu}{4}:\nu\in\Z_+\}$.
Moreover, the coefficient of 
$(s+\frac{p+q+2}{4})^{-1}$ 
in the Laurent expansion of 
$\tilde{Z}_{+}^{(2)}(s;\varphi)$
is
\begin{equation}\label{eqn:15.4}
\frac{1}{4}
\left(
\int_0^{\infty}\frac{y_1^2}{(y_1^4+1)^{\frac{p+q+2}{4}}}dy_1
\right)
\left(\int_0^{\infty}e^{-1/y_3^2}\varphi(0,0,y_3)dy_3\right).
\end{equation}

Observing the above (\ref{eqn:15.3}), (\ref{eqn:15.4}), 
we see that $\beta(f,g)=-1$ holds 
not only in the case when $p\geq 2$ 
($\Leftrightarrow g\in\hat{\mathcal E}(U)$),
but also in the case when $p+q\geq 2$.
Note that 
the support of $\varphi$ is sufficiently small,
then the value of (\ref{eqn:15.4}) 
is smaller than that of (\ref{eqn:15.3}). 
The condition $p+q\geq 2$ is induced from 
the condition: (ii-a$'$)
$g$ belongs to 
$\hat{\mathcal E}[\Phi^{-1}(\Gamma_+(f))\cap\R_+^n](U)$.
Indeed, 
$\Phi^{-1}(\Gamma_+(f))\cap\R_+^n=
(\Gamma_+(f)-\1)\cap\R_+^n
=\{\alpha\in\R_+^2:\alpha_1+\alpha_2\geq 2\}$.

As for the multiplicity, 
it is easy to see $\eta(f,g)=m(f,g)=1$.

%%%%%%%%%%%%%%%%%%%%%%%%%%%%%%%%%%%%%%%%%%%%%%%%%%%%%%%%%%%%%%%%%%%%%%%%%%%%%%%%%%%%%%%%%%%%%%%%%%%%%%%%%%%%%%%%%%%%%%%%%%%%%%%%%%%%%%%%%%%%%
\subsection{Example 3}\label{subsec:15.3}
Consider the following three-dimensional example:
%%%%%
\begin{equation*}
\begin{split}
&f(x_1,x_2,x_3)=x_1^4 x_2^4 x_3^4,\\
&g(x_1,x_2,x_3)=
x_1^4 x_2^4 x_3^2 +x_1^2 x_2^2 x_3^4 e^{-1/x_3^{2}}
(=: g_1(x_1,x_2,x_3)+g_2(x_1,x_2,x_3)).
\end{split}
\end{equation*}
%%%%%
It is easy to see that 
%$f$ is nondegenerate over $\R$
%with respect to its Newton polyhedron,
\begin{itemize}
\item
$\Gamma_+(f)=\{(4,4,4)\}+\R_+^3$,
\item
$\Gamma_+(g)=\Gamma_+(g_1)=\{(4,4,2)\}+\R_+^3$,
$\Gamma_+(g_2)=\emptyset$,
\item
$d(f,g)=4/3$, $m(f,g)=1$. 
\end{itemize}
Moreover, we can see that 
$$
g\not\in\hat{\mathcal E}(U) \mbox{\,\, but \,\,}
g\in\hat{\mathcal E}[\Phi^{-1}(\Gamma_+(f))\cap\R_+^n](U).
$$
Note that $\Phi^{-1}(\Gamma_+(f))=\{(2,2,2)\}+\R_+^3$.

We decompose $\tilde{Z}_{\pm}(s;\varphi)$ 
into 
$\tilde{Z}_{\pm}^{(1)}(s;\varphi)+
\tilde{Z}_{\pm}^{(2)}(s;\varphi)
$, where 
%%%%%
$$
\tilde{Z}_{\pm}^{(j)}(s;\varphi)=\int_{\R_+^3}
(f(x))_{\pm}^s g_j(x)\varphi(x)dx \quad j=1,2.
$$
%%%%%
Note that $\tilde{Z}_{-}^{(j)}(s;\varphi)=0$ for $j=1,2$ and that
%%%%%
\begin{equation*}
\begin{split}
&
\tilde{Z}_+^{(1)}(s;\varphi)=\int_{\R_+^3}
x_1^{4s+4}x_2^{4s+4}x_3^{4s+2}\varphi(x_1,x_2,x_3)dx,\\
&
\tilde{Z}_+^{(2)}(s;\varphi)=\int_{\R_+^3}
x_1^{4s+2}x_2^{4s+2}x_3^{4s+4}e^{-1/x_3^2}\varphi(x_1,x_2,x_3)dx.
\end{split}
\end{equation*}
%%%%%
Applying Proposition~\ref{pro:9.2}, we see that the poles of 
$\tilde{Z}_+^{(1)}(s;\varphi)$ and $\tilde{Z}_+^{(2)}(s;\varphi)$ 
are contained in the set 
$\{-\frac{3+\nu}{4};\nu\in\Z_+\}$.
From the above computation, the order of pole at $s=-3/4$ of 
$\tilde{Z}_+^{(1)}(s;\varphi)$ is not larger than $1$
and that of $\tilde{Z}_+^{(2)}(s;\varphi)$ is not larger than $2$.
Moreover, applying Proposition~\ref{pro:9.3}, we see that
the coefficient of $(s+3/4)^{-2}$ 
in the Laurent expansion of $\tilde{Z}_+^{(2)}(s;\varphi)$ is
%%%%%
$$
\frac{1}{16}\int_0^{\infty} x_3 e^{-1/x_3^2}\varphi(0,0,x_3)dx_3
\neq 0.
$$ 
%%%%%
Therefore, we obtain
$\b(f,g)=-1/d(f,g)=-3/4$
and $\eta(f,g)=2>m(f,g)=1$.

%%%%%%%%%%%%%%%%%%%%%%%%%%%%%%%%%%%%%%%%%%%%%%%%%%%%%%%%%%%%%%%%%%%%%%%%%%%%%%%%%%%%%%%%%%%%%%%%%%%%%%%%%%%%%%%%%%%%%%%%%%%%%%%%%%%%%%%%%
\vspace{1 em}

{\sc Acknowledgements.}\quad 
The authors would like to express their sincere gratitude 
to Mutsuo Oka and Osamu Saeki for useful discussion 
about the nondegeneracy condition with respect to the 
Newton polyhedron and 
to the referee for his/her careful reading of the manuscript
and giving the authors many valuable comments. 
In particular, 
these comments strongly reflect the contents in 
Sections~7.1, 10.4, 12.1 and 14.4.

%%%%%%%%%%%%%%%%%%%%%%%%%%%%%%%%%%%%%%%%%%%%%%%%%%%%%%%

%%%%%%%%%%%%%%%%%%%%%%%%%%%%%%%%%%%%%%%%%%%%%%%%%%%%%%%%%%%%

\end{document}